\documentclass[reqno]{amsart}
\usepackage{amsmath,amsthm,amssymb,amsfonts,amscd}
\usepackage{mathrsfs}
\usepackage{bbm}
\usepackage{bbding}
\usepackage{hyperref}
\usepackage{geometry}
\usepackage{color}
\usepackage{xcolor}

\usepackage{picture,epic}
\usepackage{tikz}

\usepackage{enumitem}

\numberwithin{equation}{section}

\theoremstyle{plain}
\newtheorem{theorem}{Theorem}[section]
\newtheorem{lemma}[theorem]{Lemma}

\newtheorem{proposition}[theorem]{Proposition}

\theoremstyle{definition}

\theoremstyle{remark}
\newtheorem{remark}[theorem]{Remark}

\renewcommand{\Re}{\operatorname{Re}}
\renewcommand{\Im}{\operatorname{Im}}

\newcommand{\sgn}{\operatorname{sgn}}

\newcommand{\supp}{\operatorname{supp}}

\newcommand{\GL}{\operatorname{GL}}
\newcommand{\SL}{\operatorname{SL}}

\renewcommand{\mod}{\operatorname{mod}\ }

\newcommand{\dd}{\mathrm{d}}

\newcommand{\arcsinh}{\operatorname{arcsinh}}

\makeatletter
\def\@tocline#1#2#3#4#5#6#7{\relax
  \ifnum #1>\c@tocdepth 
  \else
    \par \addpenalty\@secpenalty\addvspace{#2}%
    \begingroup \hyphenpenalty\@M
    \@ifempty{#4}{%
      \@tempdima\csname r@tocindent\number#1\endcsname\relax
    }{%
      \@tempdima#4\relax
    }%
    \parindent\z@ \leftskip#3\relax \advance\leftskip\@tempdima\relax
    \rightskip\@pnumwidth plus4em \parfillskip-\@pnumwidth
    #5\leavevmode\hskip-\@tempdima
      \ifcase #1
       \or\or \hskip 1em \or \hskip 2em \else \hskip 3em \fi%
      #6\nobreak\relax
    \hfill\hbox to\@pnumwidth{\@tocpagenum{#7}}\par
    \nobreak
    \endgroup
  \fi}
\makeatother

\begin{document}

\title[Uniform bounds for $\GL(3)\times \GL(2)$  $L$-functions]
{Uniform bounds for $\GL(3)\times \GL(2)$  $L$-functions}
\author{Bingrong Huang}
\address{Data Science Institute and School of Mathematics \\ Shandong University \\ Jinan \\ Shandong 250100 \\China}
\email{brhuang@sdu.edu.cn}


\date{\today}

\begin{abstract}
  In this paper, we prove uniform bounds for $\GL(3)\times \GL(2)$ $L$-functions in the $\GL(2)$ spectral aspect and the $t$ aspect by a delta method.
  More precisely, let $\phi$ be a Hecke--Maass cusp form for $\SL(3,\mathbb{Z})$ and $f$ a Hecke--Maass cusp form for $\SL(2,\mathbb{Z})$ with the spectral parameter $t_f$. Then for $t\in\mathbb{R}$ and any $\varepsilon>0$, we have
  \[ L(1/2+it,\phi\times f) \ll_{\phi,\varepsilon} (t_f+|t|)^{27/20+\varepsilon}. \]
  Moreover, we get subconvexity bounds for $L(1/2+it,\phi\times f)$  whenever $|t|-t_f \gg (|t|+t_f)^{3/5+\varepsilon}$.
\end{abstract}



\thanks{This work was supported by  the National Key Research and Development Program of China (No. 2021YFA1000700),  NSFC (Nos. 12001314 and 12031008), and the Young Taishan Scholars Program.} 

\maketitle

\section{Introduction} \label{sec:Intr}

The subconvexity problem of automorphic $L$-functions on the critical line is a far-reaching problem in number theory and has been a driving force for the development of new techniques and methods.
The principal aim is to prove bounds for a given $L$-function that are better than what the functional equation together with the Phragm\'en--Lindel\"of convexity principle would imply.
For the $\GL(1)$ case, i.e., the Riemann zeta function and Dirichlet $L$-functions, subconvexity bounds are known for a long time thanks to Weyl \cite{Weyl} and Burgess \cite{Burgess}. For the last decades, many cases of $\GL(2)$ $L$-functions, including the $\GL(2)\times \GL(2)$ Rankin--Selberg $L$-functions and the triple product $L$-functions,  have been treated (see e.g. 
Michel--Venkatesh \cite{MichelVenkatesh} and the references therein).
In the recent years, people have made progress on $\GL(3)$  $L$-functions and $\GL(3)\times \GL(2)$ Rankin--Selberg $L$-functions
(see e.g. \cite{li2011bounds,blomer2012subconvexity,Huang2021,
Munshi2015circleIII,Munshi2015circleIV,BlomerButtcane,Lin,
Munshi2018,Sharma2019,Kumar,HuangXu}).
%
In this paper we consider uniform bounds for $\GL(3)\times \GL(2)$ Rankin--Selberg $L$-functions on the critical line in both $\GL(2)$ spectral aspect and $t$ aspect.

Let $\phi$ be a Hecke--Maass cusp form of type $(\nu_1,\nu_2)$ for $\SL(3,\mathbb{Z})$
with the normalized Fourier coefficients $A(m,n)$.
The $L$-function of $\phi$ is defined as
\[ L(s,\phi) = \sum_{n\geq1} \frac{A(1,n)}{n^s}, \quad \Re(s)>1. \]
Let $f\in \mathcal{B}_0(1)$ be a Hecke--Maass cusp form with the spectral parameter $t_f$ for $\SL(2,\mathbb{Z})$, with the normalized Fourier coefficients $\lambda_f(n)$.
The $L$-function of $f$ is defined by
\[
  L(s,f) = \sum_{n\geq1} \frac{\lambda_f(n)}{n^s}, \quad \Re(s)>1.
\]
The $\GL(3)\times \GL(2)$  Rankin--Selberg $L$-function is defined as
\[
  L(s,\phi\times f) = \sum_{m\geq1}\sum_{n\geq1} \frac{A(m,n)\lambda_f(n)}{(m^2n)^s}, \quad \Re(s)>1.
\]
Those $L$-functions have analytic continuation to the whole complex plane.
Let $t\in\mathbb{R}$. We consider the bound for $L(1/2+it,\phi \times f)$ as $t$ or $t_f$ or both go to infinity.
The Phragm\'en--Lindel\"of principle implies the convexity bounds
\begin{equation}\label{eqn:convexity}
  L(1/2+it,\phi \times f) \ll_{\phi,\varepsilon} (t_f+|t|)^{3/4+\varepsilon}(|t_f-|t||+1)^{3/4+\varepsilon}.
\end{equation}
While the 
Riemann hypothesis for $L(s,\phi \times f)$ implies the Lindel\"of hypothesis, that is,
\[L(1/2+it,\phi \times f) \ll_{\phi,\varepsilon} (t_f+|t|)^{\varepsilon}.\]

Li \cite{li2011bounds} proved the first subconvexity bounds for $L(1/2,\phi \times f)$ in the $\GL(2)$ spectral aspect when $\phi$ is self-dual.
There are several improvements (see e.g. \cite{McKeeSunYe}) and generalizations (see e.g. \cite{blomer2012subconvexity,Huang2021}).
In those papers, the moment method for a family of $L$-functions is used to prove an individual subconvexity bound for one $L$-function. To make this work, one needs non-negativity of $L$-values, and hence we have the assumption that $\phi$ is self-dual and this only works for central $L$-values (at the center $1/2$).

Munshi \cite{Munshi2015circleIII} proved, for the first time, subconvexity bounds for generic $\GL(3)$ $L$-functions by applying a delta method. Recently, Munshi \cite{Munshi2018} extended his method to prove the first  subconvexity for $\rm GL (3)\times GL (2) $ $L$-functions,
\[
  L(1/2+it,\phi \times f) \ll_{\phi,f,\varepsilon} (1+|t|)^{3/2-1/42+\varepsilon}.
\]
This was improved by Lin--Sun \cite{LinSun}, and they got $L(1/2+it,\phi \times f) \ll_{\phi,f,\varepsilon} (1+|t|)^{27/20+\varepsilon}.$
Based on the work of Munshi \cite{Munshi2018},
Kumar \cite{Kumar} was able to use the DFI delta method to prove
\[
  L(1/2,\phi \times f) \ll_{\phi,\varepsilon} t_f^{3/2-1/51+\varepsilon}
\]
without the assumption that $\phi$ is self-dual.

It is interesting and challenging to prove uniform bounds in terms of both $t_f$ and $t$.
For the $\GL(2)$ case, Jutila--Motohashi \cite{JutilaMotohashi} proved uniform bounds for $L(s,f)$ on the critical line by the moment method, getting
\[
  L(1/2+it,f)\ll_\varepsilon (t_f+|t|)^{1/3+\varepsilon}.
\]
In \cite{JutilaMotohashi2006}, Jutila and Motohashi extended their method to show some uniform bounds for $\GL(2)\times \GL(2)$ $L$-functions. More precisely, they proved
\[
  L(1/2+it,g\times f) \ll_{g,\varepsilon} \left\{
  \begin{array}{ll}
    t_f^{2/3+\varepsilon}, & \textrm{for } 0\leq t\ll t_f^{2/3},\\
    t_f^{1/2+\varepsilon}t^{1/4}, & \textrm{for } t_f^{2/3}\leq t\ll t_f,\\
    t^{3/4+\varepsilon}, & \textrm{for } t_f \ll t\ll t_f^{3/2-\varepsilon},
  \end{array}\right.
\]
where $f,g$ are Hecke--Maass cusp forms for $\SL(2,\mathbb{Z})$.
Their method can not cover all cases of $t$ and $t_f$.

It seems that it is very hard to extend Jutila--Motohashi's method to solve the uniform bound problem for high degree $L$-functions such as $L(1/2+it,\phi \times f)$.
In this paper, we find a  way to prove uniform bounds for $L(1/2+it,\phi \times f)$ for any fixed $\phi$. We will use the DFI delta method. In fact, one can also solve the uniform bound problem for $\GL(2)\times \GL(2)$ $L$-functions by using our method  (together with ideas in \cite{ASS}) which may prove uniform bounds for any real $t$ and $t_f$.
Our key novelty is the treatment of integral transforms after applying summation formulas which allows us to deal with the case that both $t$ and $t_f$ are large.
The main result in this paper is the following subconvexity bounds.

\begin{theorem}\label{thm:main}
  With the notation as above. Let $t\in\mathbb{R}$.  Then if $||t|-t_f|\geq (|t|+t_f)^{5/6}$ then we have
  \[
    L(1/2+it,\phi \times f) \ll_{\phi,\varepsilon}
    (t_f+|t|)^{7/8+\varepsilon} ||t|-t_f|^{19/40},
  \]
  and if $(|t|+t_f)^{3/5} \leq ||t|-t_f|\leq (|t|+t_f)^{5/6}$ then we have
  \[
    L(1/2+it,\phi \times f) \ll_{\phi,\varepsilon}
    (t_f+|t|)^{57/56+\varepsilon} ||t|-t_f|^{17/56}.
  \]
  In particular, we have the following bound
  \begin{equation}\label{eqn:uniformbound}
    L(1/2+it,\phi \times f) \ll_{\phi,\varepsilon} (t_f+|t|)^{27/20+\varepsilon}.
  \end{equation}
\end{theorem}

\begin{remark}
  The bound \eqref{eqn:uniformbound} is uniformly subconvex on $||t|-t_f|\geq (|t|+t_f)^{4/5+\varepsilon}$.
  Our result is new even when $t=0$ and $t_f\rightarrow\infty$, which improves Kumar's result in \cite{Kumar}. The improvement comes from our treatment of the weight functions after applying summation formulas (see \S\ref{sec:Voronoi} and \S\ref{sec:Cauchy-Poisson}). We also avoid the use of the ``conductor lowering trick'' of Munshi \cite{Munshi2015circleIII,Munshi2018,Kumar} as in \cite{Aggarwal,LinSun,Huang}.  Note that in this case our subconvexity bounds are as strong as the case $t_f\ll1$ and $t\rightarrow\infty$.
\end{remark}

\begin{remark}
  By the functional equation of $L(s,\phi\times f)$, we can assume $t\geq0$ in the proof.
  The uniform bound \eqref{eqn:uniformbound} is a consequence of the subconvexity bounds in Theorem \ref{thm:main} and  the convexity bound \eqref{eqn:convexity} if $|t_f-t|\leq (t_f+|t|)^{3/5+\varepsilon}$. See \S \ref{subsec:L-functions} for more details.
  If $ |t_f-t|\leq  (t_f+|t|)^{1-\varepsilon}$, then we have the conductor drop phenomenon, which makes the subconvexity problem even harder. However, our result still gives subconvexity bounds when $|t_f-t| \geq (t_f+|t|)^{3/5+\varepsilon} $.
  One may improve our results when $|t_f-t|\leq  (t_f+|t|)^{1-\varepsilon}$ by more careful analysis of the integrals in the case $|t_f-t|^{1-\varepsilon} \ll \frac{NX}{PQ} \ll (t_f+|t|)^{1-\varepsilon}$. (See Lemmas \ref{lemma:G>>} and \ref{lem:I>>} below).
\end{remark}

\begin{remark}
  The method in this paper should work for both holomorphic and Maass forms.
  In this paper we focus on the Hecke--Maass cusp form case, as Kumar \cite{Kumar} gave details for holomorphic forms. Let $f\in H_k(1)$ be a weight $k$ holomorphic Hecke cusp form for $\SL(2,\mathbb{Z})$.
  Then our techniques may yield
  \[
    L(1/2+it,\phi \times f) \ll_{\phi,\varepsilon} (k+|t|)^{27/20+\varepsilon}.
  \]
\end{remark}

\begin{remark}
  One can combine our ideas here with the method in Huang--Xu \cite{HuangXu} to prove hybrid subconvexity bounds for twists of $\GL(3)\times \GL(2)$ $L$-functions in the spectral, $t$, and conductor aspects. See \cite{HuangXu} and the references therein for more backgrounds on such hybrid bounds. Let $\phi$ be a Hecke--Maass cusp form  for $\SL(3,\mathbb{Z})$  and $f$ a Hecke--Maass cusp form with the spectral parameter $t_f$ for $\SL(2,\mathbb{Z})$. For $t\in\mathbb{R}$ and $\chi$ a primitive Dirichlet character modulo prime $M$, our techniques should yield
  \[
    L(1/2+it,\phi\times f\times \chi) \ll_{\phi,\varepsilon} M^{23/16+\varepsilon} (|t|+t_f)^{27/20+\varepsilon}.
  \]
  By taking $\phi$ the minimal Eisenstein series for $\SL(3,\mathbb{Z})$, our techniques should yield
  \[
    L(1/2+it, f\times \chi) \ll_{\varepsilon} M^{23/48+\varepsilon} (|t|+t_f)^{9/20+\varepsilon}.
  \]
\end{remark}

\subsection{Sketch of the proof}

In this sketch,  we assume $t+t_f = T \asymp t-t_f$ and $t\geq0$.
We want to prove that $L(1/2+it,\phi\times f) \ll_{\phi,\varepsilon} T^{27/20+\varepsilon}$.
By using the approximate functional equation, we are led to consider the following sum
\[
   \sum_{n\geq1} A(1,n) \lambda_f(n) n^{-it} V\left(\frac{n}{N}\right),
\]
for some smooth function $V$  supported in $[1, 2]$ and satisfying $V^{(j)}(x) \ll_j 1$.
Hence to establish subconvexity we need to show cancellation in the above sum   for $N\ll T^{3+\varepsilon}$.

Our first step is to follow Munshi \cite{Munshi2018}.
We apply the delta method (see Lemma \ref{lemma:delta} below) directly to the above sum as a device for separation of the oscillation of the Fourier coefficients $A(1,n)$ and $\lambda_f(n) n^{-it}$, arriving at (we only consider the generic terms in this sketch)
\begin{align*}
  & \frac{1}{Q}\sum_{\substack{q\asymp Q}} \int_{x\asymp 1} \;\frac{1}{q}\; \sideset{}{^\star}\sum_{\substack{a\bmod{q} }}
   \sum_{n\asymp N} A(1,n) e\left(\frac{-an}{q}\right)
   e\left(\frac{-nx}{qQ}\right)   \\
  & \hskip 90pt \cdot
   \sum_{\substack{m\asymp N}} \lambda_f(m) e\left(\frac{m a}{q}\right)
   e\left(\frac{m x}{qQ}\right) m^{-it}
  \mathrm{d}x.
\end{align*}
Here $Q=\sqrt{\frac{N}{K}}$ for some $1\ll K= o(T)$. In fact, we have smooth weights for the sums over $m$ and $n$ and the integral over $x$.
Note that here we also avoid the use of ``conductor lowering trick'' in \cite{Munshi2018}.
The trivial bound is $O(N^2)$, so we need to save $N$ plus a little more.

Now it is standard to apply the Voronoi summation formulas for both $m$-sum and $n$-sum above.
We proceed as follows. The dual $n$-sum (after applying Voronoi) was treated in \cite{Huang} and we get
\[
   q  \sum_{\pm} \sum_{n_1|q} \sum_{n_2=1}^{\infty}
              \frac{A(n_2,n_1)}{n_1n_2} S\left(-\bar{a},\pm n_2;\frac{q}{n_1}\right)
              \Psi_x^{\pm}\left(\frac{n_1^2n_2}{q^3}\right),
\]
for certain weight function $\Psi_x^{\pm}$. Here $S(a,b;c)$ is the Kloosterman sum and $\bar{a} a\equiv 1 \pmod{q/n_1}$.
The sum over $n$ has ``conductor'' $(\frac{N}{Q^2}Q)^3=N^3/Q^3$, and hence the length of the dual sum is $N^2/Q^3$.
By using the stationary phase method, in the generic case (assuming $n_1=1$), this becomes
\[
  \frac{N^{3/2}}{Q^2}   \sum_{\substack{  n_2 \asymp N^2/Q^3}}
              \frac{A(n_2,1)}{n_2} S\left(-\bar{a},\pm n_2;q\right)
     e\left( \pm 2 \frac{n_2^{1/2} Q^{1/2}}{q x^{1/2}} \right).
\]
By the square root cancellation of the Kloosterman sums, we save
$N/(\frac{N^{3/2}}{Q^2} Q^{1/2}) = \frac{Q^{3/2}}{N^{1/2}}$ in this step.
Note that we have two oscillatory factors of $x$ which have different exponents for $x$. So it is a good place to apply the stationary phase method for the $x$-integral and this will save $\frac{N^{1/2}}{Q}$. Now we arrive at
\begin{equation}  \label{eqn:sum_n2&m}
    \frac{1}{Q}\sum_{\substack{q\asymp Q}} \frac{1}{q} \sideset{}{^\star}\sum_{\substack{a\bmod{q} }}
      \frac{N}{Q}
      \sum_{\substack{ n_2 \asymp \frac{  N^{2} } {Q^3} }}
              \frac{A(n_2, 1)}{ n_2} S\left(-\bar{a},\pm n_2;q\right)
   \sum_{\substack{m\asymp N}} \lambda_f(m) e\left(\frac{ma}{q}\right)
   e\left(\pm 3\frac{m^{1/3}  n_2^{1/3} }{q}  \right)
    m^{-it} .
\end{equation}

Consider the sum over $m$, which involves $\GL(2)$ Fourier coefficients, and has conductor
$Q^2T^2$ if $t\pm t_f \asymp T$. By applying the Voronoi summation formula, a typical term in the dual $m$-sum is
\[
  q  \sum_{m\geq1} \frac{ \lambda_{f}(m) }{m}
  e\left(\frac{\bar{a} m}{q}\right) G \left(\frac{m}{q^2} \right),
\]
where
\[
  G \left(y\right) \approx \int_{\mathbb{R}} (\pi^2 y)^{1/2-i\tau+it}
    \gamma_2(-1/2+i\tau-it) \tilde{g}(1/2-i\tau+it) \dd \tau,
\]
  \[
    \gamma_2 (-1/2+i\tau-it) = \frac{\Gamma(\frac{1/2+i\tau-iT}{2})\Gamma(\frac{1/2+i\tau-iT'}{2})} {\Gamma(\frac{1/2-i\tau+iT}{2})\Gamma(\frac{1/2-i\tau+iT'}{2})}
    +  \frac{\Gamma(\frac{3/2+i\tau-iT}{2})\Gamma(\frac{3/2+i\tau-iT'}{2})} {\Gamma(\frac{3/2-i\tau+iT}{2})\Gamma(\frac{3/2-i\tau+iT'}{2})},
  \]
\[
  \tilde{g}(s) = \int_{0}^{\infty} g(u) u^{s-1}\dd u
  , \quad g(u)= e\left(\pm 3\frac{u^{1/3}  n_2^{1/3} }{q}  \right)
    u^{-it} W\left(\frac{u}{N}\right)  ,
\]
for some nice smooth function $W$. Here we introduce two new parameters
\[ T=t+t_f \quad \textrm{and} \quad  T'=t-t_f. \]
The analysis of $G(y)$ is the main part of this paper.
Note that we also have some other expressions for $G(y)$. For example, the one with Bessel functions (see Lemma \ref{lem:VSF2}), which will also be used in the non generic cases to truncate the dual $m$-sum.
In our generic case, by applying the stationary phase method for $\tilde{g}(1/2-i\tau+it)$, we get a nice asymptotic formula of this and restrict to $\tau\asymp N/Q^2=K = o(T)$. With the assumption $T'\asymp T$, we can use the Stirling's formula to get a good approximation of $\gamma_2 (-1/2+i\tau-it)$. Finally, the stationary phase method can be applied to the $\tau$-integral, which restrict $y$  to $(yN)^{1/2}\asymp T$. By doing these, we will not lose any thing in this integral transform. This is the key to our improvement (\emph{cf.} \cite[\S 7]{Kumar}).
More importantly, our method do not really depend on the sizes of $t$ and $t_f$, but instead the sizes of $T$ and $T'$. So as long as $T'$ is not too small compared to $T$, we may get nontrivial bounds.
This is the reason why we can prove a uniform bound.
Those arguments show that the length of the dual $m$-sum is $Q^2T^2/N$ and the size of $G(y)$ is $O((yN)^{1/2})=O(T)$.
Now \eqref{eqn:sum_n2&m} essentially becomes
\begin{equation*}
   \frac{N^{1/2}  }{Q}
      \frac{N}{Q}
      \sum_{\substack{n_2 \asymp \frac{N^{2}} {Q^3} }}
              \frac{A(n_2,1)}{n_2}\\
    \cdot
              \sum_{\substack{q\asymp Q}} \;\frac{1}{q}
    \sum_{\substack{ m\asymp \frac{Q^2T^2}{N} }} \frac{ \lambda_{f}(m) }{m^{1/2}} \left(\frac{m}{q^2}\right)^{it}
  \mathcal{C}(n_2,m,q) \mathcal{I}(n_2,m,q) ,
\end{equation*}
where
\begin{equation*}
  \mathcal{C}(n_2,m,q)  =   \sideset{}{^\star}\sum_{\substack{a\bmod{q} }}
  e\left(\frac{\bar{a} m}{q}\right)
  S\left(-\bar{a},\pm n_2;q\right) \rightsquigarrow
  q \; e\left( \pm \frac{\bar{m}n_2}{q} \right)
\end{equation*}
and $\mathcal{I}(n_2,m,q)$ is certain nice oscillatory function with the phase function of size $N/Q^2$ with respect to $n_2$.
Here for the character sum $\mathcal{C}(n_2,m,q)$, the sum over $a$ becomes a Ramanujan sum.
So typically we have $\mathcal{C}(n_2,m,q)\ll q$ and we save $Q^{1/2}$. Hence we save $\frac{N}{QT}Q^{1/2} = \frac{N}{Q^{1/2}T}$ from the treatment of the $m$-sum. So we have saved $\frac{Q^{3/2}}{N^{1/2}}\cdot \frac{N^{1/2}}{Q}\cdot \frac{N}{Q^{1/2}T}=\frac{N}{T}$ in total, and it remains to save $T$ plus a little extra.

The next step involves taking Cauchy to get rid of the Fourier coefficients $A(n_2,1)$, but
this process also squares the amount we need to save, getting (essentially)
\[
   \frac{1}{T} \frac{N}{Q^{3/2}} \Bigg( \sum_{n_2 \asymp \frac{N^{2}} {Q^3}} \bigg|
     \sum_{\substack{q\asymp Q}}
    \sum_{\substack{ m\asymp \frac{Q^2T^2}{N} }}  \lambda_{f}(m)
  e\left( \pm \frac{\bar{m}n_2}{q} \right) \mathcal{I}(n_2,m,q)
   \bigg|^2 \Bigg)^{1/2}.
\]
 Opening the absolute value square, we get (essentially)
\[
   \frac{1}{T} \frac{N}{Q^{3/2}} \Bigg(
     \sum_{\substack{q\asymp Q}}
    \sum_{\substack{ m\asymp \frac{Q^2T^2}{N} }}
     \sum_{\substack{q'\asymp Q}}
    \sum_{\substack{ m'\asymp \frac{Q^2T^2}{N} }}
    \sum_{n_2 \asymp \frac{N^{2}} {Q^3}}
  e\left( \pm \frac{(\bar{m}q' - \bar{m}'q )n_2}{qq'}   \right)
  \mathcal{I}(n_2,m,q) \overline{\mathcal{I}(n_2,m',q')}  \Bigg)^{1/2}.
\]
We now apply the Poisson summation formula on the sum over $n_2$ modulo $qq'$,
arriving at
\[
   \frac{1}{T} \frac{N}{Q^{3/2}} \Bigg(
     \sum_{\substack{q\asymp Q}}
    \sum_{\substack{ m\asymp \frac{Q^2T^2}{N} }}
     \sum_{\substack{q'\asymp Q}}
    \sum_{\substack{ m'\asymp \frac{Q^2T^2}{N} }}
    \sum_{n\in \mathbb{Z}}
  \mathfrak{C}(n,q,q',m,m')
  \mathfrak{I}(n,q,q',m,m')   \Bigg)^{1/2},
\]
where
\[
  \mathfrak{C}(n,q,q',m,m') = \frac{1}{qq'} \sum_{b\bmod{qq'}}
  e\left( \pm \frac{(\bar{m}q' - \bar{m}'q + n)b}{qq'}   \right) ,
\]
and
\begin{align*}
  \mathfrak{I}(n,q,q',m,m') & = \int_{\mathbb{R}}
  \mathcal{I}(u,m,q) \overline{\mathcal{I}(u,m',q')} W\left(\frac{u}{N^2/Q^3}\right) e\left( \frac{un}{qq'}  \right) \dd u \\
  & = \frac{N^2}{Q^3} \int_{\mathbb{R}}
  \mathcal{I}\Big(\frac{N^2}{Q^3}\xi,m,q\Big)
  \overline{\mathcal{I}\Big(\frac{N^2}{Q^3}\xi,m',q'\Big)}
  W\left(\xi\right) e\left( \frac{N^2 n}{Q^3 qq'} \xi  \right) \dd \xi.
\end{align*}
For the zero frequency ($n=0$), the main contribution comes from terms with $q=q'$ and $m=m'$, in which case there is no further cancellation in the character sums. So we save $(Q \frac{Q^2T^2}{N})^{1/2} = \frac{Q^{3/2}T}{N^{1/2}}$.
Hence the final contribution from the zero frequency is $O(N^2/(\frac{N}{T}\frac{Q^{3/2}T}{N^{1/2}})) = O(N^{3/2}/Q^{3/2})$.

For the non zero frequencies ($n\neq0$), the main contribution comes from the terms in generic positions (that is, no restriction to reduce the size of the number of $q,q',m,m'$). As mentioned in Munshi \cite{Munshi2018}, we save more than the usual since the character sum boils down to an additive character. In generic case, the ``conductor'' is of the size $Q^2 \frac{N}{Q^2}=N$ and hence the length of the dual sum is $O( \frac{N}{N^2/Q^3})=O(Q^3/N)$.
By the stationary phase method, we save $(\sqrt{N/Q^2})^{1/2}$ in $\mathfrak{I}(n,q,q',m,m')$. We remark that for the non generic cases, we will apply $L^2$-norm estimate for $\mathcal{I}(u,m,q)$ instead of the stationary phase method. For the character sums $\mathfrak{C}(n,q,q',m,m')$, we save $(Q^2)^{1/2}$ since this is a Ramanujan sum with modulo $qq'$. So in this case we save $(\sqrt{N/Q^2})^{1/2}\cdot (Q^2)^{1/2} \cdot (N/Q^3)^{1/2} = N^{3/4}/Q$.
Hence the final contribution from the non zero frequencies is $O(N^2/(\frac{N}{T} \frac{N^{3/4}}{Q})) = O(N^{1/4} QT)$.

The best choice is $Q=\frac{N^{1/2}}{T^{2/5}}$, which gives a bound $O(N^{1/2}T^{27/20+\varepsilon})$ by using $N\ll T^{3+\varepsilon}$ and hence proves that $L(1/2+it,\phi\times f) \ll_{\phi,\varepsilon} T^{27/20+\varepsilon}$.

\subsection{Plan for this paper}
The rest of this paper is organized as follows.
In \S \ref{sec:preliminaries}, we introduce some notation and present some lemmas that we will need later.
The approximate functional equation allows us to reduce the subconvexity problem to estimate certain convolution sums.
In \S \ref{sec:applying_delta}, we apply the delta method to the convolution sums.
In \S \ref{sec:Voronoi}, we apply the Voronoi summation formulas and estimate the integral transforms by the stationary phase method.
In \S \ref{sec:Cauchy-Poisson}, we apply the Cauchy--Schwarz inequality and Poisson summation formula, and then analyse the character sums and integrals.
Then we deal with the contribution from the zero frequency in \S \ref{sec:zero-freq}.
The contribution from non zero frequencies is bounded in \S \ref{sec:non-zero-freq-I} and \S \ref{sec:non-zero-freq-II}.
Finally, in \S \ref{sec:Proof} we complete the proof of our main theorem.

\medskip
\textbf{Notation.}
Throughout the paper, $\varepsilon$ is an arbitrarily small positive number;
all of them may be different at each occurrence.
The weight functions $U,\ V,\ W$ may also change at each occurrence.
As usual, $e(x)=e^{2\pi i x}$.
We use $y\asymp Y$ to mean that $c_1 Y\leq |y|\leq c_2 Y$ for some positive constants $c_1$ and $c_2$,
and
$q\sim P$ means $P<q\leq 2P$.

\section{Preliminaries}\label{sec:preliminaries}

\subsection{Automorphic forms}

Let  $f\in \mathcal{B}_0(1)$ be a Hecke--Maass cusp form with the spectral parameter $t_f$ for $\SL(2,\mathbb{Z})$, with the normalized Fourier coefficients $\lambda_f(n)$.
Let $\theta_2$ be the bound toward to the Ramanujan conjecture and we have $\theta_2\leq 7/64$ due to Kim--Sarnak \cite{Kim2003}.
Rankin--Selberg theory gives (see Iwaniec \cite[Lemma 1]{iwaniec1990spectral})
\begin{equation}\label{eqn:RS2}
  \sum_{n\leq N} |\lambda_f(n)|^2 \ll t_f^{\varepsilon} N.
\end{equation}

Let $\phi$ be a Hecke--Maass cusp form of type $(\nu_1,\nu_2)$ for $\SL(3,\mathbb{Z})$
with the normalized Fourier coefficients $A(r,n)$.
Rankin--Selberg theory gives
\begin{equation}\label{eqn:RS3}
  \sum_{r^2n\leq N} |A(r,n)|^2 \ll_\phi   N.
\end{equation}

We record the Hecke relation
\[
  A(r,n) = \sum_{d\mid (r,n)} \mu(d) A\left(\frac{r}{d},1\right) A\left(1,\frac{n}{d}\right)
\]
which follows from M\"obius inversion and \cite[Theorem 6.4.11]{goldfeld2006automorphic}.
Hence we have the individual bounds
\[
  A(r,n) \ll (rn)^{\theta_3+\varepsilon},
\]
where $\theta_3\leq 5/14$ is the bound toward to the Ramanujan conjecture on $\GL(3)$ (\cite{Kim2003}).
Thus we have
\begin{equation}\label{eqn:RS3-1}
  \sum_{n\sim N} |A(r,n)| \ll \sum_{n_1\mid r^\infty} \sum_{\substack{n\sim N/n_1 \\ (n,r)=1}} |A(r,nn_1)| \leq \sum_{n_1\mid r^\infty}|A(r,n_1)| \sum_{\substack{n\sim N/n_1 \\ (n,r)=1}} |A(1,n)| \ll r^{\theta_3+\varepsilon} N
\end{equation}
and
\begin{equation}\label{eqn:RS3-2}
  \sum_{n\sim N} |A(r,n)|^2
  \ll \sum_{n_1\mid r^\infty} \sum_{\substack{n\sim N/n_1 \\ (n,r)=1}} |A(r,nn_1)|^2
  \leq \sum_{n_1\mid r^\infty}|A(r,n_1)|^2 \sum_{\substack{n\sim N/n_1 \\ (n,r)=1}} |A(1,n)|^2
  \ll r^{2\theta_3+\varepsilon} N.
\end{equation}
Those bounds depend on $\phi$ and $\varepsilon$.
Here we have used \eqref{eqn:RS3} and the fact
$
  \sum_{d\mid r^\infty} d^{-\sigma} \ll r^\varepsilon, \; \textrm{for $\sigma>0$}.
$

\subsection{The approximate functional equation} \label{subsec:L-functions}

The Rankin--Selberg $L$-function $L(s,\phi\times f)$ has the following functional equation
\[
  \Lambda(s,\phi\times f) = \epsilon_{\phi\times f} \Lambda(1-s,\tilde\phi\times f),
\]
where
\[
  \Lambda(s,\phi\times f) = \pi^{-3s} \prod_{j=1}^{3} \prod_{\pm}
  \Gamma\left(\frac{s-\alpha_j\pm i t_f}{2}\right) L(s,\phi\times f)
\]
is the completed $L$-function and $\epsilon_{\phi\times f}$ is the root number which has absolute value one.
Here $\alpha_j$ are the Langlands parameters of $\phi$, and $\tilde\phi$ is the dual form of $\phi$.
We have the following approximate functional equation.

\begin{lemma}\label{lemma:AFE}
  Assume $t\geq0$. Let $T=t+t_f$ and $T'=t-t_f$. Then we have
  \[
    L(1/2+it,\phi\times f) \ll_{\phi,\varepsilon} T^\varepsilon
    \sup_{ 1 \leq N \leq  T^{3/2+\varepsilon}(|T'|+1)^{3/2}} \frac{|S(N)|}{N^{1/2}} +  T^{-2021},
  \]
  where $S(N)$ is a sum of the form
  \[
    S(N) := \sum_{r\geq1}\sum_{n\geq1} A(r,n) \lambda_f(n) (r^2n)^{-it} V\left(\frac{r^2n}{N}\right)
  \]
  for some smooth function $V$ such that $\int_{\mathbb{R}}V(x)\dd x=1$, $\supp V \subset [1,2]$, and $V^{(j)}(x) \ll_j 1$ for any integer $j\geq0$.
\end{lemma}

\begin{proof}
  See \cite[\S5.2]{IwaniecKowalski2004analytic}.
\end{proof}

If $|T'| \leq T^{3/5}$, then Lemma \ref{lemma:AFE} gives
\[
    L(1/2+it,\phi \times f) \ll_{\phi,\varepsilon} T^{3/4+\varepsilon} (|T'|+1)^{3/4}
    \ll T^{6/5+\varepsilon},
\]
which is better than \eqref{eqn:uniformbound}. Hence to prove Theorem \ref{thm:main}, we only need to consider the case $|T'| \geq T^{3/5}$, which we assume from now on. We will always write
\[
  T=t+t_f \quad \textrm{and} \quad  T'=t-t_f.
\]

We first estimate the contribution from large values of $r$. By \eqref{eqn:RS2} and \eqref{eqn:RS3-2} we have
\begin{align*}
  \sum_{r\geq R}  & \left|\sum_{n\geq1} A(r,n) \lambda_f(n)   (r^2n)^{-it} V\left(\frac{r^2n}{N}\right)\right| \\
   & \ll \sum_{ R \leq r \ll \sqrt{N}} \left(\sum_{n\asymp N/r^2} |A(r,n)|^2 \right)^{1/2} \left(\sum_{n\asymp N/r^2}  |\lambda_f(n)|^2\right)^{1/2} \\
   & \ll \sum_{ R \leq r \ll \sqrt{N}} r^{\theta_3+\varepsilon} \frac{N}{r^2}
  \ll N \sum_{ R \leq r \ll \sqrt{N}} r^{-23/14+\varepsilon}
  \ll N^{1/2} T^{3/4+\varepsilon} |T'|^{3/4} R^{-9/14},
\end{align*}
for $N \ll T^{3/2+\varepsilon}|T'|^{3/2}$.
Take
\begin{equation}\label{eqn:R}
   R = \left\{ \begin{array}{ll}
     |T'|^{77/180}T^{-7/36},  & \textrm{ if  $T^{5/6} \leq |T'|\leq T$}, \\
     |T'|^{25/36}T^{-15/36},  & \textrm{ if $T^{3/5}\leq |T'|\leq T^{5/6}$}.
   \end{array}\right.
\end{equation}
The contribution from those terms to $L(1/2+it,\pi\times f )$ is bounded by $T^{3/4+\varepsilon} |T'|^{3/4} R^{-9/14}$, which is good enough for Theorem \ref{thm:main}.
Hence we get
\begin{equation}\label{eqn:S<<Sr}
    L(1/2+it,\phi\times f) \ll T^\varepsilon \sum_{r\leq R} \frac{1}{r}
    \sup_{  N \leq  \frac{T^{3/2+\varepsilon}|T'|^{3/2}}{r^2}} \frac{|S_r(N)|}{N^{1/2}} + T^{7/8+\varepsilon} |T'|^{19/40} +  T^{57/56+\varepsilon} |T'|^{17/56}  ,
\end{equation}
where
\[
  S_r(N) := \sum_{n\geq1} A(r,n) \lambda_f(n) n^{-it} V\left(\frac{n}{N}\right).
\]
Thus to prove Theorem \ref{thm:main}, we only need to prove the following proposition.

\begin{proposition}\label{prop:SrN}
  Assume  $|T'| \geq T^{3/5}$. For $r\leq R$ and
  $ N \leq  \frac{T^{3/2+\varepsilon}|T'|^{3/2}}{r^2}$, we have
  \[
    S_r(N) \ll N^{1/2+\varepsilon} \left(T^{7/8} |T'|^{19/40} +  T^{57/56} |T'|^{17/56} \right).
  \]
\end{proposition}

\subsection{Summation formulas}

We first recall
the Poisson summation formula over an arithmetic progression.
\begin{lemma}\label{lem:Poisson}
  Let $\beta\in\mathbb{Z}$ and $c\in \mathbb{Z}_{\geq1}$. For a Schwartz function $f:\mathbb{R}\rightarrow \mathbb{C}$, we have
  \[
    \sum_{\substack{n\in\mathbb{Z}\\ n\equiv \beta \bmod{c}}} f(n) = \frac{1}{c} \sum_{n\in\mathbb{Z}} \hat{f}\left(\frac{n}{c}\right) e\left(\frac{n\beta}{c}\right),
  \]
  where $\hat{f}(y)=\int_{\mathbb{R}} f(x) e(-xy)\dd x$ is the Fourier transform of $f$.
\end{lemma}

\begin{proof}
  See e.g. \cite[Eq. (4.24)]{IwaniecKowalski2004analytic}
\end{proof}

Now we turn to the Voronoi summation formula for $\SL(2,\mathbb{Z})$.
Let $f$ be a weight zero Hecke--Maass cusp form for $\SL(2,\mathbb{Z})$ with spectral parameter $t_f$.
Let $\epsilon_f=\pm1$ depending on $f$ even or odd.
Let $g(x),\ \psi(x)$ be smooth functions with compact support on the positive reals.
Let $q\in\mathbb{Z}_{\geq1}$ and $a\in\mathbb{Z}$ with $(q,a)=1$. Define $\bar{a}$ as the inverse of $a$ modulo $q$, i.e., $a\bar{a}\equiv 1 \pmod{q}$.
\begin{lemma}\label{lem:VSF2}
  With the notation as above.  Then we have
  \begin{equation}\label{eqn:VSF2}
  \sum_{n\geq1} \lambda_f(n) e\left(\frac{an}{q}\right) g(n)
  = q \sum_{\pm}  \sum_{n\geq1} \frac{ \lambda_{f}(n) }{n}
  e\left(\mp\frac{\bar{a} n}{q}\right) G^\pm \left(\frac{n}{q^2} \right),
  \end{equation}
  where
  \begin{equation}\label{eqn:G}
    \begin{split}
      G^\pm(y) & =  \frac{\epsilon_f^{(1\mp 1)/2}}{4\pi^2 i}  \int_{(\sigma)} (\pi^2 y)^{-s} \left( \frac{\Gamma(\frac{1+s+it_f}{2})\Gamma(\frac{1+s-it_f}{2})} {\Gamma(\frac{-s+it_f}{2})\Gamma(\frac{-s-it_f}{2})} \mp \frac{\Gamma(\frac{2+s+it_f}{2})\Gamma(\frac{2+s-it_f}{2})} {\Gamma(\frac{1-s+it_f}{2})\Gamma(\frac{1-s-it_f}{2})} \right) \tilde{g}(-s) \dd s \\
      & = \epsilon_f^{(1\mp 1)/2} y
      \int_{0}^{\infty} g(x) J^\pm_{f}\left(  4\pi\sqrt{yx} \right) \dd x,
    \end{split}
  \end{equation}
  with $\sigma>\theta_2-1$  and $\tilde{g}(s) = \int_{0}^{\infty} g(x) x^{s-1} \dd x$ the Mellin transform of $g$, and
  \[
    J^+_{f}\left( x \right) = \frac{-\pi}{\sin(\pi i t_f)} \left( J_{2it_f}(x) - J_{-2it_f}(x) \right), \qquad
    J^-_f(x) = 4\cosh(\pi t_f) K_{2it_f}(x).
  \]
\end{lemma}

\begin{proof}
  See \cite[eq. (1.12) \& (1.15)]{MillerSchmid2006automorphic} and \cite[Appendix A]{KowalskiMichelVanderKam2002rankin}. 
\end{proof}

%

We also recall the Voronoi summation formula for $\SL(3,\mathbb{Z})$.
Let $\psi$ be a smooth compactly supported function on $(0,\infty)$,
and let $\tilde{\psi}
$
be the Mellin transform of $\psi$.
For $\sigma>5/14$, we define
\begin{equation}\label{eqn:Psi}
    \Psi^{\pm}(z) := z \frac{1}{2\pi i} \int_{(\sigma)} (\pi^3z)^{-s} \gamma^\pm(s) \tilde{\psi}(1-s)\dd s ,
\end{equation}
with
\begin{equation}\label{eqn:gamma^pm}
  \gamma^\pm(s) := \prod_{j=1}^{3}
    \frac{\Gamma\left(\frac{s+\alpha_j}{2}\right)} {\Gamma\left(\frac{1-s-\alpha_j}{2}\right)}
    \pm \frac{1}{i} \prod_{j=1}^{3}
    \frac{\Gamma\left(\frac{1+s+\alpha_j}{2}\right)} {\Gamma\left(\frac{2-s-\alpha_j}{2}\right)},
\end{equation}
where $\alpha_j$ are the Langlands parameters of $\phi$ as above.
Note that changing $\psi(y)$ to $\psi(y/N)$ for a positive real number $N$ has the effect of
changing $\Psi^\pm(z)$ to $\Psi^\pm(zN)$.
The Voronoi formula on $\GL(3)$ was first proved by Miller--Schmid~\cite{MillerSchmid2006automorphic}.
The present version is due to Goldfeld--Li~\cite{goldfeld2006voronoi} with slightly renormalized variables (see Blomer \cite[Lemma 3]{blomer2012subconvexity}).
\begin{lemma}\label{lemma:VSF3}
  Let $c,d,\bar{d}\in\mathbb Z$ with $c\neq0$, $(c,d)=1$, and $d\bar{d}\equiv1\pmod{c}$.
  Then we have
  \begin{equation*}
    \begin{split}
         \sum_{n=1}^{\infty} A(r,n)e\left(\frac{n\bar{d}}{c}\right)\psi(n)
         = \frac{c\pi^{3/2}}{2} \sum_{\pm} \sum_{n_1|cr} \sum_{n_2=1}^{\infty}
              \frac{A(n_2,n_1)}{n_1n_2} S\left(rd,\pm n_2;\frac{rc}{n_1}\right)
              \Psi^{\pm}\left(\frac{n_1^2n_2}{c^3r}\right),
    \end{split}
  \end{equation*}
  where $S(a,b;c) := \mathop{{\sum}^*}_{d(c)} e\left(\frac{ad+b\bar{d}}{c}\right)$ is the classical Kloosterman sum.
\end{lemma}

%
%
%

\subsection{The delta method}
There are several oscillatory factors contributing to the convolution sums. Our method is based on separating these oscillations using the delta/circle method. In the present situation we will use a version of the delta method of Duke, Friedlander and Iwaniec. More specifically we will use the expansion (20.157) given in  \cite[\S20.5]{IwaniecKowalski2004analytic}. Let $\delta:\mathbb{Z}\rightarrow \{0,1\}$ be defined by
$$
\delta(n)=\begin{cases} 1&\text{if}\;\;n=0;\\
0&\text{otherwise}.\end{cases}
$$
We seek a Fourier expansion which matches with $\delta(n)$.
\begin{lemma}\label{lemma:delta}
  Let $Q$ be a large positive number. Then we have
  \begin{align}\label{eqn:delta-n}
    \delta(n)=\frac{1}{Q}\sum_{1\leq q\leq Q} \;\frac{1}{q}\; \sideset{}{^\star}\sum_{a\bmod{q}}e\left(\frac{na}{q}\right)
    \int_\mathbb{R} g(q,x) e\left(\frac{nx}{qQ}\right)\mathrm{d}x,
  \end{align}
  where 
  $g(q,x)$ is a weight function satisfies that
  \begin{equation}\label{eqn:g-h}
    g(q,x)=1+O\left(\frac{Q}{q}\left(\frac{q}{Q}+|x|\right)^A\right),
    \quad
     g(q,x)\ll |x|^{-A}, \quad \textrm{for any $A>1$},
  \end{equation}
  and
  \begin{equation}\label{eqn:g^(j)}
    \frac{\partial^j}{\partial x^j} g(q,x) \ll  |x|^{-j} \min(|x|^{-1},Q/q) \log Q, \quad j\geq1.
  \end{equation}
  Here the $\star$ on the sum indicates that the sum over $a$ is restricted by the condition $(a,q)=1$.
\end{lemma}

\begin{proof}
  See \cite[Lemma 15]{Huang}.
\end{proof}

\subsection{Weight functions} \label{subsec:weight_function}

Let $\mathcal{F}$ be an index set and $X=X_T:\mathcal{F}\rightarrow \mathbb{R}_{\geq1}$ be a function of $T\in\mathcal{F}$. A family of $\{w_T\}_{T\in\mathcal{F}}$ of smooth functions supported on a product of dyadic intervals in $\mathbb{R}_{>0}^d$ is called \emph{$X$-inert} if for each $j=(j_1,\ldots,j_d) \in \mathbb{Z}_{\geq0}^d$ we have
\[
  \sup_{T\in\mathcal{F}} \sup_{(x_1,\ldots,x_d) \in \mathbb{R}_{>0}^d}
  X_T^{-j_1-\cdots -j_d} \left| x_1^{j_1} \cdots x_d^{j_d} w_T^{(j_1,\ldots,j_d)} (x_1,\ldots,x_d) \right|
   \ll_{j_1,\ldots,j_d} 1.
\]

For a $T^\varepsilon$-inert function $V$, we may separate variables in $V(x_1, \ldots , x_d)$ by first inserting a redundant function $V (x_1) \cdots V (x_d)$ that is 1 on the support of $V$ and then applying Mellin inversion
\begin{multline*}
  V (x_1, \ldots , x_d) = V (x_1, \ldots , x_d)V (x_1) \cdots V (x_d)
   \\
   =  \frac{1}{(2\pi i)^d} \int_{(0)}\cdots \int_{(0)} \tilde{V}(s_1,\ldots,s_d)
  (V (x_1) \cdots V (x_d) x_1^{-s_1} \cdots x_n^{-s_d} ) \dd s_1 \cdots \dd s_d,
\end{multline*}
where $\tilde{V}(s_1,\ldots,s_d)=\int_{0}^{\infty}\cdots\int_{0}^{\infty} V(x_1, \ldots , x_d) x_1^{s_1-1} \cdots x_d^{s_d-1} \dd x_1 \cdots \dd x_d$ is the Mellin transform of $V$.
Here we can truncate the vertical integrals at height $|\Im s_j| \ll T^{2\varepsilon}$ at the cost of a negligible error $O_A(T^{-A})$.
We will often separate variables in this way without explicit mention.

\subsection{Oscillatory integrals}

We will use the following integration by parts and stationary phase lemmas several times.

\begin{lemma}\label{lemma:repeated_integration_by_parts}
  Let $Y\geq1$. Let $X,\; V,\; R,\; Q>0$ and suppose that $w=w_T$ is a smooth function with  $\supp w \subseteq [\alpha,\beta]$ satisfying $w^{(j)}(\xi) \ll_j X V^{-j}$ for all $j\geq0$.
  Suppose that on the support of $w$, $h=h_T$ is smooth and satisfies that
  $h'(\xi)\gg R$ and $ h^{(j)}(\xi) \ll Y Q^{-j}$, for all $j\geq2.$
  Then for arbitrarily large $A$ we have
    \[
      I = \int_{\mathbb{R}} w(\xi) e(h(\xi))  \dd \xi  \ll_A (\beta-\alpha)  X \left[  \left(\frac{QR}{\sqrt{Y}}\right)^{-A} + (RV)^{-A}  \right].
    \]
\end{lemma}

\begin{proof}
  See \cite[Lemma 8.1]{BlomerKhanYoung}.
\end{proof}

\begin{lemma}\label{lemma:stationary_phase}
  Suppose $w_T$ is $X$-inert in $t_1,\ldots,t_d$, supported on  $t_i\asymp X_i$ for $i=1,2,\ldots,d$. Suppose that on the support of $w_T$, $h=h_T$ satisfies that
  \[
    \frac{\partial^{a_1+a_2+\cdots +a_d}}{\partial t_1^{a_1}\cdots \partial t_d^{a_d}} h(t_1,t_2,\ldots,t_d) \ll_{a_1,\ldots,a_d}  \frac{Y}{X_1^{a_1} X_2^{a_2}\cdots X_d^{a_d}},
  \]
  for all $a_1,\ldots,a_d\in \mathbb{Z}_{\geq0}$. Let
  \[
    I = \int_{\mathbb{R}} w_T(t_1,t_2,\ldots,t_d) e^{i h(t_1,t_2,\ldots,t_d)}  \dd t_1.
  \]
   Suppose $\frac{\partial^{2}}{\partial t_1^{2}} h(t_1,t_2,\ldots,t_d) \gg \frac{Y}{X_1^2}$
  for all $(t_1,t_2,\ldots,t_d)\in \supp w_T$, and there exists $t_0 \in\mathbb{R}$ such that $ \frac{\partial}{\partial t_1} h(t_0,t_2,\ldots,t_d)=0$.
  Suppose that $Y/X^2 \geq R \geq 1$. Then
  \[
    I
    = \frac{X_1}{\sqrt{Y}} e^{i h(t_0,t_2,\ldots,t_d)} W_T(t_2,\ldots,t_d) + O_A(X_1 R^{-A}),
  \]
  for some $X$-inert family of functions $W_T$ and any $A>0$.
\end{lemma}

\begin{proof}
  See \cite[\S 8]{BlomerKhanYoung} and \cite[\S 3]{KPY}.
\end{proof}

In the applications of Lemma \ref{lemma:stationary_phase}, we will explicitly give estimates of the derivatives for the first variable. For other derivatives we will also check all those conditions, but may not write them down explicitly.

\subsection{Stirling's formula}

For fixed $\sigma\in\mathbb{R}$, real $|t|\geq10$ and any $J>0$, we have Stirling's formula
\begin{equation*}
  \Gamma(\sigma+it) = e^{-\frac{\pi}{2}|t|} |t|^{\sigma-\frac{1}{2}} \exp\left( it\log\frac{|t|}{e} \right) \left( g_{\sigma,J}(t) + O_{\sigma,J}(|t|^{-J}) \right),
\end{equation*}
where
\[
  t^j \frac{\partial^j}{\partial t^j} g_{\sigma,J}(t) \ll_{j,\sigma,J} 1
\]
for all fixed $j\in \mathbb{N}_0$.
Similarly, we have
\begin{equation*}
  \frac{1}{\Gamma(\sigma+it)} = e^{\frac{\pi}{2}|t|} |t|^{-\sigma+\frac{1}{2}} \exp\left( -it\log\frac{|t|}{e} \right) \left(  h_{\sigma,J}(t) + O_{\sigma,J}(|t|^{-J}) \right),
\end{equation*}
where
\[
  t^j \frac{\partial^j}{\partial t^j} h_{\sigma,J}(t) \ll_{j,\sigma,J} 1
\]
for all fixed $j\in \mathbb{N}_0$.
Hence
\begin{equation}\label{eqn:Stirling}
  \frac{\Gamma(\sigma+it)}{\Gamma(\sigma-it)}
  = \exp\left( 2it\log\frac{|t|}{e} \right)
  \left( w_{\sigma,J}(t) + O_{\sigma,J}(|t|^{-J}) \right),
\end{equation}
where
\[
  t^j \frac{\partial^j}{\partial t^j} w_{\sigma,J}(t) \ll_{j,\sigma,J} 1
\]
for all fixed $j\in \mathbb{N}_0$.

\subsection{Bessel functions}

We need the following asymptotic formula for Bessel functions when $x\gg T^\varepsilon |\tau|$.
For $\tau \in \mathbb{R}$, $|\tau|>1$ and $x>0$, we have \cite[Eq. 7.13.2 (17)]{EMOT}
\begin{equation}\label{eqn:Jsim}
  \frac{J_{2i\tau}(2x)}{\cosh(\pi \tau)} = \sum_{\pm} e^{\pm 2i \omega(x,\tau)} \frac{g_A^{\pm}(x,\tau)}{x^{1/2}+|\tau|^{1/2}} + O(x^{-A}),
\end{equation}
where $g_A^{\pm}(x,\tau)$ is an $1$-inert function and
\begin{equation}\label{eqn:omega}
  \omega(x,\tau) = |\tau| \cdot \arcsinh \frac{|\tau|}{x} - \sqrt{x^2+\tau^2}.
\end{equation}
For $x\geq T^\varepsilon |\tau|$, we have \cite[Eq. 7.13.2 (18)]{EMOT}
\begin{equation}\label{eqn:Ksim}
   K_{2i\tau}(2x) \cosh(\pi \tau)  \ll x^{-1/2} \exp(-2x+ \pi|\tau|) \ll x^{-6} \exp(-x),
\end{equation}
for $T$ large enough.

\section{Applying the delta method}\label{sec:applying_delta}

By the delta method (Lemma \ref{lemma:delta}) we have
\begin{align*}
  S_r(N) & = \sum_{n\geq1} A(r,n)   V\left(\frac{n}{N}\right)
   \sum_{\substack{m\geq1}} \lambda_f(m)  m^{-it} W\left(\frac{m}{N}\right)  \delta\left( m-n \right) \\
  & = \sum_{n\geq1} A(r,n)   V\left(\frac{n}{N}\right)
   \sum_{\substack{m\geq1}} \lambda_f(m)  m^{-it} W\left(\frac{m}{N}\right)  \\
  & \hskip 50pt \cdot
  \frac{1}{Q}\sum_{1\leq q\leq Q} \;\frac{1}{q}\; \sideset{}{^\star}\sum_{a\bmod{q}}e\left(\frac{(m-n)a}{q}\right)
  \int_\mathbb{R}g(q,x) e\left(\frac{(m-n)x}{qQ}\right)\mathrm{d}x,
\end{align*}
where $W$ is a fixed smooth function such that $\supp W \subset \mathbb{R}^+$ and $W(u)=1$ if $u\in \supp V$, and $W^{(j)}(u)\ll_j 1$ for any $j\geq0$.
Inserting a smooth partition of unity for the $x$-integral and a dyadic partition for the $q$-sum, we get
\begin{equation}\label{eqn:SrN<<SrNXP}
  S_r(N) \ll  N^\varepsilon \sup_{N^{-B} \ll X \ll N^\varepsilon}
  \sup_{1\ll P \ll Q} |S_r^\pm(N,X,P)|  + O_A(N^{-A})
\end{equation}
where $B=B(A)>0$ is a large constant depending on $A$ and
\begin{align*}
  S_{r}^\pm(N,X,P)  & =
     \frac{1}{Q}\sum_{\substack{q\sim P}} \int_\mathbb{R} V\left(\frac{\pm x}{X}\right) \;\frac{1}{q}\; \sideset{}{^\star}\sum_{\substack{a\bmod{q} }}
   \sum_{n\geq1} A(r,n) e\left(\frac{-an}{q}\right)
   e\left(\frac{-nx}{qQ}\right)  V\left(\frac{n}{N}\right) \\
  & \hskip 90pt \cdot
   \sum_{\substack{m\geq1}} \lambda_f(m) e\left(\frac{m a}{q}\right)
   e\left(\frac{m x}{qQ}\right) m^{-it} W\left(\frac{m}{N}\right)
  \mathrm{d}x.
\end{align*}

\section{Applying Voronoi}\label{sec:Voronoi}

We first apply the Voronoi summation formula to the sum over $n$ (see Lemma \ref{lemma:VSF3}), getting
\begin{align*}
  \sum_{n\geq1} A(r,n) e\left(\frac{-an}{q}\right) &
   e\left(\frac{-nx}{qQ}\right)  V\left(\frac{n}{N}\right) \\
  & = \frac{q\pi^{3/2}}{2} \sum_{\pm} \sum_{n_1|qr} \sum_{n_2=1}^{\infty}
              \frac{A(n_2,n_1)}{n_1n_2} S\left(-r\bar{a},\pm n_2;\frac{rq}{n_1}\right)
              \Psi_x^{\pm}\left(\frac{n_1^2n_2}{q^3r}\right),
\end{align*}
where $\psi_x(u)=e\left(-\frac{ux}{qQ}\right) V\left(\frac{u}{N}\right)$ and $\Psi_x^{\pm}$ defined as in \eqref{eqn:Psi} with $\psi$ replaced by $\psi_x$.

\begin{lemma}\label{lemma:Psi}
  We have
  \begin{itemize}
    \item [(i)] If $zN\gg T^{\varepsilon}$, then $\Psi^\pm_x(z) \ll z^{-6} T^{-A}$ is negligibly small unless $\sgn(x) = \pm$ and $\pm\frac{Nx}{qQ} \asymp (zN)^{1/3}$, in which case we have
        \begin{equation}\label{eqn:Psi=1}
            \Psi_x^\pm(z) = \left(\pm \frac{Nx}{qQ}\right)^{3/2} e\left(\pm 2 \frac{(zN)^{1/2}}{(\pm \frac{Nx}{qQ})^{1/2}}\right) w\left(\frac{ zN }{(\pm \frac{Nx}{qQ})^{3}}\right) + O(T^{-A}) \ll (zN)^{1/2},
        \end{equation}
        where $w$ is a certain compactly supported $1$-inert  function depending on $A$.
    \item [(ii)] If $zN\ll T^{\varepsilon}$ and $\frac{Nx}{qQ}\gg T^\varepsilon$, then $\Psi_x^\pm(z) \ll_A T^{-A}$ for any $A>0$.
    \item [(iii)] If $zN\ll T^{\varepsilon}$ and $\frac{Nx}{qQ}\ll T^\varepsilon$, then $\Psi_x^\pm(z) \ll T^{\varepsilon}$.
  \end{itemize}
\end{lemma}

\begin{proof}
  See \cite[\S 5.3]{Huang}
\end{proof}

\subsection{The oscillating cases}

If $\frac{NX}{PQ}\gg T^\varepsilon$, then we have
\begin{align*}
  S_{r}^\pm(N,X,P)  =
     \frac{1}{Q}&\sum_{\substack{q\sim P}} \int_\mathbb{R} V\left(\frac{\pm x}{X}\right) \;\frac{1}{q}\; \sideset{}{^\star}\sum_{\substack{a\bmod{q} }}
      \frac{q\pi^{3/2}}{2} \sum_{n_1|qr} \sum_{n_2=1}^{\infty}
              \frac{A(n_2,n_1)}{n_1n_2} S\left(-r\bar{a},\pm n_2;\frac{rq}{n_1}\right) \\
  &  \cdot  \left(\pm \frac{Nx}{qQ}\right)^{3/2}
   e\left(\pm 2 \frac{( n_1^2n_2  Q)^{1/2}}{r^{1/2} q(\pm x)^{1/2}}\right) w\left(\frac{n_1^2n_2 Q^3}{r N^2 (\pm  x )^{3}}\right) \\
  &  \cdot
   \sum_{\substack{m\geq1}} \lambda_f(m) e\left(\frac{m a}{q}\right)
   e\left(\frac{m x}{qQ}\right) m^{-it} W\left(\frac{m}{N}\right)
  \mathrm{d}x  + O(T^{-A}).
\end{align*}
We first deal with the $x$-integral. Making a change of variable $x = \pm X \xi$, we get
\begin{align*}
  S_{r}^\pm(N,X,P) =
      X &\frac{1}{Q}\sum_{\substack{q\sim P}} \;\frac{1}{q}\; \sideset{}{^\star}\sum_{\substack{a\bmod{q} }}
      \frac{q\pi^{3/2}}{2} \sum_{n_1|qr} \sum_{n_2=1}^{\infty}
              \frac{A(n_2,n_1)}{n_1n_2} S\left(-r\bar{a},\pm n_2;\frac{rq}{n_1}\right) \\
  &  \cdot  \left(\frac{NX}{qQ}\right)^{3/2}
   \sum_{\substack{m\geq1}} \lambda_f(m) e\left(\frac{ma}{q}\right)
    m^{-it} W\left(\frac{m}{N}\right)
  \\
  &  \cdot
   \int_\mathbb{R} w\left(\frac{n_1^2n_2 Q^3}{r N^2X^{3}\xi^3}\right) V\left(\xi\right)  \xi^{3/2} e\left(\frac{\pm m X\xi}{qQ}\right)
   e\left(\pm 2 \frac{( n_1^2n_2  Q)^{1/2}}{r^{1/2} qX^{1/2} \xi^{1/2}}\right)
  \mathrm{d}\xi  + O(T^{-A}).
\end{align*}
We can remove the weight function $w\left(\frac{n_1^2n_2 Q^3}{r N^2X^{3}\xi^3}\right)$ by the Mellin technique as in \S \ref{subsec:weight_function}.
Then we have
\begin{align*}
  S_{r}^\pm(N,X,P) \ll T^\varepsilon \bigg|
      &\frac{X}{Q}\sum_{\substack{q\sim P}} \; \sideset{}{^\star}\sum_{\substack{a\bmod{q} }}
        \sum_{n_1|qr} \sum_{n_2=1}^{\infty}
              \frac{A(n_2,n_1)}{n_1n_2} S\left(-r\bar{a},\pm n_2;\frac{rq}{n_1}\right)
      V_1\left(\frac{n_1^2n_2 Q^3}{r N^2X^{3}}\right) \\
  &  \cdot  \left(\frac{NX}{qQ}\right)^{3/2}
   \sum_{\substack{m\geq1}} \lambda_f(m) e\left(\frac{ma}{q}\right)
    m^{-it} W\left(\frac{m}{N}\right)
  \\
  &  \cdot
   \int_\mathbb{R}  V_2\left(\xi\right)  e\left(\frac{\pm m X\xi}{qQ} \pm 2 \frac{( n_1^2n_2  Q)^{1/2}}{r^{1/2} qX^{1/2} \xi^{1/2}}\right)
  \mathrm{d}\xi \bigg|  + O(T^{-A}),
\end{align*}
for some $T^\varepsilon$-inert functions $V_1$ and $V_2$ with support in $[1,2]$.
We now consider the $\xi$-integral above. Let (temporarily)
  \[
    h(\xi) = \frac{\pm m X\xi}{qQ}
    \pm 2 \frac{( n_1^2n_2  Q)^{1/2}}{r^{1/2} qX^{1/2} \xi^{1/2}}.
  \]
  Then
  \[
    h'(\xi) = \frac{\pm m X }{qQ}
    \mp \frac{( n_1^2n_2  Q)^{1/2}}{r^{1/2} qX^{1/2} \xi^{3/2}},
  \]
  and
  \[
    h''(\xi) = \pm \frac{3}{2} \frac{( n_1^2n_2  Q)^{1/2}}{r^{1/2} qX^{1/2} \xi^{5/2}} , \quad
    h^{(j)}(\xi) \asymp_j \frac{NX}{PQ}, \quad j\geq2.
  \]
  The solution of $h'(\xi)=0$ is $\xi_0 = \frac{( n_1^2n_2)^{1/3} Q}{r^{1/3}   m^{2/3}X}$.
  Note that
  \[
    h(\xi_0) = \pm 3\frac{m^{1/3} ( n_1^2n_2)^{1/3} }{r^{1/3} q}
  \quad \textrm{and} \quad
    h''(\xi_0) = \pm \frac{3}{2} \frac{mX}{q Q \xi_0 } .
  \]
  Now by Lemma \ref{lemma:stationary_phase} with
  \[
    \textrm{$X= T^\varepsilon$, $t_1=\xi$, $t_2=n_1^2n_2$,  $t_3=m$, $t_4=q$,
     $X_1=1$, $X_2=\frac{rN^2X^3}{Q^3}$, $X_3=N$, $X_4=P$, and  $Y=\frac{NX}{PQ}$,}
  \]
  we get
\begin{multline*}
   \int_\mathbb{R}  V_2\left(\xi\right)  e\left(\frac{\pm m X\xi}{qQ} \pm 2 \frac{( n_1^2n_2  Q)^{1/2}}{r^{1/2} qX^{1/2} \xi^{1/2}}\right)
  \mathrm{d}\xi
  \\
  = \left(\frac{NX}{qQ}\right)^{-1/2}  V_3\left(\frac{ n_1^2n_2  Q^3}{r  N^{2}X^3},\frac{m}{N},\frac{q}{P}\right)
  e\left(\pm 3\frac{m^{1/3} ( n_1^2n_2)^{1/3} }{r^{1/3} q}  \right) + O_A(T^{-A}),
\end{multline*}
where $V_3$ is a $T^\varepsilon$-inert function with compact support in $\mathbb{R}_{>0}^3$.
Hence we obtain
\begin{multline}\label{eqn:Sr=3}
  S_{r}^\pm(N,X,P)  \ll T^\varepsilon \bigg|
    \frac{X}{Q}\sum_{\substack{q\sim P}} \;\frac{1}{q}\; \sideset{}{^\star}\sum_{\substack{a\bmod{q} }}
      \frac{NX}{Q}  \sum_{n_1|qr}
      \sum_{\substack{n_2=1 \\ n_1^2n_2 \asymp \frac{r  N^{2}X^3} {Q^3} }}^{\infty}
              \frac{A(n_2,n_1)}{n_1n_2} S\left(-r\bar{a},\pm n_2;\frac{rq}{n_1}\right)
              V\left(\frac{ n_1^2n_2  Q^3}{r  N^{2}X^3}\right) \\
   \cdot
   \sum_{\substack{m\geq1}} \lambda_f(m) e\left(\frac{ma}{q}\right)
   e\left(\pm 3\frac{m^{1/3} ( n_1^2n_2)^{1/3} }{r^{1/3} q}  \right)
    m^{-it} W\left(\frac{m}{N}\right)
   \bigg|
   + O(T^{-A}).
\end{multline}
Here we have removed the weight function $V_3$ by the Mellin technique again to separate the variables $n_2$ and $m$, and  modified the weight functions $W$ and $V$ accordingly. Note that $W$ and $V$ are $T^\varepsilon$-inert functions with compact support in $\mathbb{R}_{>0}$.

We now apply the Voronoi summation formula (see Lemma \ref{lem:VSF2}) to the sum over $m$ getting
\begin{align}\label{eqn:m-sum==}
   \sum_{\substack{m\geq1}} \lambda_f(m) e\left(\frac{ma}{q}\right)
   e\left(\pm 3\frac{m^{1/3} ( n_1^2n_2)^{1/3} }{r^{1/3} q}  \right)
   &
    m^{-it} W\left(\frac{m}{N}\right)
  \nonumber
  \\
  & = q \sum_{\pm_1}  \sum_{m\geq1} \frac{ \lambda_{f}(m) }{m}
  e\left(\frac{\pm_1\bar{a} m}{q}\right) G^{\pm_1} \left(\frac{m}{q^2} \right),
\end{align}
where  $g(m)= e\left(\pm 3\frac{m^{1/3} ( n_1^2n_2)^{1/3} }{r^{1/3} q}  \right)
    m^{-it} W\left(\frac{m}{N}\right)$ and $G^{\pm_1}$ is defined as in \eqref{eqn:G}.

\begin{lemma}\label{lemma:G>>}
  Assume $x\asymp X$ and $q\sim P$. Then
  \begin{enumerate}[label=\roman*)]
    \item  If  $yN  \gg T^{2+\varepsilon} + (\frac{NX}{PQ})^{2+\varepsilon}$, then we have $G^{\pm_1}(y) \ll_A y^{-6} T^{-A}$.
    \item If $\frac{NX}{PQ}\gg |T'|^{1-\varepsilon}$, then  we have $ G^{\pm_1}(y)$ is equal to (up to an error term of size $O(T^{-A})$)
        \[
          y^{it} (yN)^{1/2} \left( \frac{PQ}{NX}\right)^{1/2}
             \int_{\mathbb{R}}  y^{-i\tau}
             \left(\frac{  n_1^2n_2}{ q^3} \right)^{i\tau}
             W_1\left(\frac{ n_1^2n_2  Q^3}{r  N^{2}X^3}, \frac{\pm \tau} { \frac{NX}{PQ} } ,\frac{q}{P} \right)   w^{\pm_1}(\tau)  \dd \tau ,
        \]
        for some function $w^{\pm_1}$ such that $w^{\pm_1}(\tau)\ll 1$ and
        some $T^\varepsilon$-inert function $W_1$  with compact support in $\mathbb{R}_{>0}^3$.
    \item If $T^\varepsilon \ll \frac{NX}{PQ}\ll |T'|^{1-\varepsilon}$,
        then $G^{\pm_1}(y) \ll y^{-6} T^{-A}$ is negligibly small unless $yN \asymp T|T'|$, in which case we have $ G^{\pm_1}(y)$ is equal to (up to an error term of size $O(T^{-A})$)
        \[
           (\pi^2 y)^{it} (Ny)^{1/2}
      e\left(  - \frac{T}{2\pi} \log \frac{T}{2e}
    - \frac{T'}{2\pi} \log \frac{|T'|}{2e}  \pm \frac{B}{2\pi} \sum_{0\leq \ell\leq L} Q_\ell^\pm \left(\frac{B}{T},\frac{B}{T'} \right) \xi_0^{\ell+1}\right)
    W_3^{\pm_1}\left( \frac{B} { \frac{NX}{PQ} },\frac{q}{P} \right),
        \]
        where $L=L(A)$ is a large enough integer, $Q_\ell^\pm$ is a certain homogeneous polynomial of degree $\ell$ with $Q_0^\pm\left(\frac{B}{T},\frac{B}{T'} \right) =3$ and   $Q_1^\pm\left(\frac{B}{T},\frac{B}{T'} \right) = \mp \frac{1}{2} \left(\frac{B}{T}+\frac{B}{T'}\right) $,
        $B=\frac{ N^{1/3} ( n_1^2n_2)^{1/3} }{ r^{1/3} q}\asymp \frac{NX}{PQ} $,
        $\xi_0 = \left(\frac{2\pi T|T'|}{yN}\right)^{1/3} \asymp 1$, and
        $W_3^{\pm_1}$ is a $T^\varepsilon$-inert function with compact support in $\mathbb{R}_{>0}^2$.
  \end{enumerate}
\end{lemma}

\begin{proof}
  (i) First we use the second expression in \eqref{eqn:G} getting
  \begin{equation*}
    G^{\pm_1}(y)  = \epsilon_f^{(1\mp_1 1)/2} y
      \int_{0}^{\infty}   e\left(\pm 3\frac{u^{1/3} ( n_1^2n_2)^{1/3} }{r^{1/3} q}  \right)
    u^{-it} W\left(\frac{u}{N}\right)
    V\left(\frac{ n_1^2n_2  Q^3}{r  u^{2}X^3}\right) J^{\pm_1}_{f}\left(  4\pi\sqrt{yu} \right) \dd u.
  \end{equation*}
  Making a change of variable $u=N\xi$, we have $G^{\pm_1}(y)$ is equal to
  \begin{equation*}
    \epsilon_f^{(1\mp_1 1)/2} y N^{1-it}
      \int_{0}^{\infty}   e\left(\pm 3\frac{N^{1/3} ( n_1^2n_2)^{1/3} }{r^{1/3} q} \xi^{1/3} \right)
    \xi^{-it} W\left(\xi\right)
  V\left(\frac{ n_1^2n_2  Q^3}{r N^2 \xi^{2}X^3}\right) J^{\pm_1}_{f}\left(  4\pi\sqrt{yN\xi} \right) \dd \xi.
  \end{equation*}
  If $yN\gg t_f^{2}T^{\varepsilon}$, then by \eqref{eqn:Ksim} we have $G^{-}(y) \ll_A y^{-6} T^{-A}$ for any $A>0$.
  If $yN\gg  T^{2+\varepsilon} + (\frac{NX}{PQ})^{2+\varepsilon}$, then by \eqref{eqn:Jsim} we have
  \begin{multline*}
    G^+(y) =  y N^{1-it}
      \int_{0}^{\infty}   e\left(\pm 3\frac{N^{1/3} ( n_1^2n_2)^{1/3} }{r^{1/3} q} \xi^{1/3} \right)
    \xi^{-it} W\left(\xi\right)
    V\left(\frac{ n_1^2n_2  Q^3}{r N^2 \xi^{2}X^3}\right) \\
    \cdot \sum_{\pm} e^{\pm 2i \omega(2\pi\sqrt{yN\xi},t_f)} \frac{g_A^{\pm}(2\pi\sqrt{yN\xi},t_f)}{(2\pi\sqrt{yN\xi})^{1/2}+t_f^{1/2}}   \dd \xi
    + O_A(y^{-6} T^{-A}).
  \end{multline*}
  Let (temporarily)
  \[
    h(\xi)  = \pm 6\pi\frac{N^{1/3} ( n_1^2n_2)^{1/3} }{r^{1/3} q} \xi^{1/3} -t\log \xi
    \pm 2 \omega(2\pi\sqrt{yN\xi},t_f).
  \]
  Then we have
  \[
    h'(\xi) = \pm 2\pi\frac{N^{1/3} ( n_1^2n_2)^{1/3} }{r^{1/3} q} \xi^{-2/3} -\frac{t}{\xi}
    \mp \frac{\sqrt{(2\pi\sqrt{yN\xi})^2+t_f^2}}{\xi}
    \gg (yN)^{1/2},
  \]
  \[
    h^{(j)}(\xi) \ll (yN)^{1/2}, \quad j\geq2 .
  \]
  By Lemma \ref{lemma:repeated_integration_by_parts} with
  \[
    \textrm{$X=V=1$,  $Y= (yN)^{1/2}$, $Q=1$, and $R= (yN)^{1/2}$,}
  \]
  we get $G^+(y) \ll_A y^{-6} T^{-A}$.
  Hence we have $G^{\pm_1}(y) \ll_A y^{-6} T^{-A}$ if
   $yN\gg T^{2+\varepsilon} + (\frac{NX}{PQ})^{2+\varepsilon}$.

  (ii) For $yN  \ll T^{2+\varepsilon} + (\frac{NX}{PQ})^{2+\varepsilon}$  we use the first expression in \eqref{eqn:G}.
    Writing $s=\sigma+i\tau$ with $\sigma=-1/2$ and making a change of variable $\tau\rightsquigarrow \tau-t$, we get
  \begin{equation}\label{eqn:G(y)}
    G^{\pm_1}(y) = \frac{\epsilon_f^{(1\mp_1 1)/2}}{4\pi^2}  \int_{\mathbb{R}} (\pi^2 y)^{1/2-i\tau+it}
    \gamma_2^{\pm_1}(-1/2+i\tau-it) \tilde{g}(1/2-i\tau+it) \dd \tau,
  \end{equation}
  where
  \[
    \gamma_2^{\pm_1}(-1/2+i\tau-it) = \frac{\Gamma(\frac{1/2+i\tau-iT}{2})\Gamma(\frac{1/2+i\tau-iT'}{2})} {\Gamma(\frac{1/2-i\tau+iT}{2})\Gamma(\frac{1/2-i\tau+iT'}{2})}
    \pm_1  \frac{\Gamma(\frac{3/2+i\tau-iT}{2})\Gamma(\frac{3/2+i\tau-iT'}{2})} {\Gamma(\frac{3/2-i\tau+iT}{2})\Gamma(\frac{3/2-i\tau+iT'}{2})}.
  \]

  If $\frac{NX}{PQ}\gg T^{\varepsilon}$, then
  \begin{equation*}
    \tilde{g}(1/2-i\tau+it) = \int_{\mathbb{R}}
    e\left(\pm 3\frac{u^{1/3} ( n_1^2n_2)^{1/3} }{r^{1/3} q}  \right)
      W\left(\frac{u}{N}\right) u^{-1/2-i\tau} \dd u.
  \end{equation*}
  Making a change of variable $u=N\xi^3$, we have
  \begin{equation*}
    \tilde{g}(1/2-i\tau+it) = N^{1/2-i\tau} \int_{\mathbb{R}}
    e\left(\pm 3\frac{N^{1/3} ( n_1^2n_2)^{1/3} }{r^{1/3} q} \xi  - \frac{3 \tau}{2\pi} \log \xi \right) W\left(\xi^3\right) 3 \xi^{1/2}  \dd \xi.
  \end{equation*}
  Let (temporarily)
  \[
    h(\xi) = \pm 3\frac{N^{1/3} ( n_1^2n_2)^{1/3} }{r^{1/3} q} \xi  - \frac{3\tau}{2\pi} \log \xi.
  \]
  Then
  \[
    h'(\xi) = \pm 3 \frac{N^{1/3} ( n_1^2n_2)^{1/3} }{r^{1/3} q}
    - \frac{3 \tau}{2\pi \xi} ,
  \]
  and
  \[
    h''(\xi) = \frac{3\tau}{2\pi \xi^2} , \quad
    h^{(j)}(\xi) \asymp_j |\tau|, \quad j\geq2.
  \]
  By Lemma \ref{lemma:repeated_integration_by_parts} 
   with
  \[
    \textrm{$X=1$, $V=T^{-\varepsilon}$,  $Y=|\tau|$, $Q=1$, and $R= |\tau|+\frac{NX}{PQ}$,}
  \]
  we have $\tilde{g}(1/2-i\tau+it)$ is negligibly small unless $\sgn(\tau)=\pm$ and $\tau\asymp \frac{N^{1/3} ( n_1^2n_2)^{1/3} }{r^{1/3} q}  \asymp \frac{NX}{PQ}$, in which case the solution of $h'(\xi)=0$ is
  \[
    \xi_0 = \frac{\pm \tau}{2\pi } \frac{r^{1/3} q} {N^{1/3} ( n_1^2n_2)^{1/3} }.
  \]
  Note that
  \[
    h(\xi_0) = -\frac{3\tau}{2\pi} \log  \frac{\pm \tau r^{1/3} q} {2\pi e N^{1/3} ( n_1^2n_2)^{1/3} }.
  \]
  Now by Lemma \ref{lemma:stationary_phase} with
  \[
    \textrm{$X= T^\varepsilon$, $t_1=\xi$, $t_2=n_1^2n_2$,  $t_3=\tau$, $t_4=q$,
    $X_1=1$, $X_2=\frac{rN^2X^3}{Q^3}$, $X_3=\frac{NX}{PQ}$, $X_4=P$,
    and $Y=\frac{NX}{PQ}$}
  \]
  we get
  \begin{equation}\label{eqn:tilde(g)}
    \tilde{g}(1/2-i\tau+it) = N^{1/2-i\tau} \left(\frac{NX}{PQ}\right)^{-1/2}
    e\left(-\frac{3\tau}{2\pi} \log  \frac{\pm \tau } {2\pi e B }\right)
    W_1\left(\frac{ n_1^2n_2  Q^3}{r  N^{2}X^3}, \frac{\pm \tau} { \frac{NX}{PQ} } ,\frac{q}{P} \right)  + O_A(T^{-A}),
  \end{equation}
  where $B=\frac{ N^{1/3} ( n_1^2n_2)^{1/3} }{ r^{1/3} q} $ and
  $W_1$ is a $T^\varepsilon$-inert function with compact support in $\mathbb{R}_{>0}^3$.
  Note that we have $B \asymp \frac{NX}{PQ}$.

   Now we consider the case $ \frac{NX}{PQ}\gg |T'|^{1-\varepsilon}$.
  By \eqref{eqn:G(y)} and \eqref{eqn:tilde(g)} we have
  \begin{multline}
    G^{\pm_1}(y) = \frac{\epsilon_f^{(1\mp_1 1)/2}}{4\pi}   (\pi^2 y)^{it}
    (yN)^{1/2}  \left(\frac{NX}{PQ}\right)^{-1/2}
    \int_{\mathbb{R}} (\pi^2 yN )^{-i\tau}
    \left(\frac{ N n_1^2n_2}{ r q^3} \right)^{i\tau}
    W_1\left(\frac{ n_1^2n_2  Q^3}{r  N^{2}X^3}, \frac{\pm \tau} { \frac{NX}{PQ} } ,\frac{q}{P} \right)
    \\
    \cdot \gamma_2^{\pm_1}(-1/2+i\tau-it)
    e\left(-\frac{3\tau}{2\pi} \log  \frac{\pm \tau } {2\pi e}\right)
     \dd \tau   + O(T^{-A}).
  \end{multline}
  Taking $w^{\pm_1}(\tau) = \frac{\epsilon_f^{(1\mp_1 1)/2}}{4\pi} \pi^{2it} (\pi^2 r)^{-i\tau} \gamma_2^{\pm_1}(-1/2+i\tau-it)
    e\left(-\frac{3\tau}{2\pi} \log  \frac{\pm \tau } {2\pi e}\right)$ and noting that $w^{\pm_1}(\tau)\ll 1$,
  we complete the proof of Lemma \ref{lemma:G>>} (ii).

  (iii)
  We now consider the case $T^\varepsilon \ll \frac{NX}{PQ}\ll |T'|^{1-\varepsilon}$.
  By Stirling's formula, for $\tau \ll |T'|^{1-\varepsilon}$, we have
  \begin{multline*}
    \gamma_2^{\pm_1}(-1/2+i\tau-it)
    = \exp\left( i(\tau-T)\log\frac{T-\tau}{2e} + i(\tau-T')\log\frac{|T'-\tau|}{2e} \right)
    \\ \cdot \left( w_{\epsilon_f,J}^{\pm_1}(\tau-T) w_{\epsilon_f,J}^{\pm_1}(\tau-T') + O_{J}(T^{-J}) \right),
\end{multline*}
where
\[
  t^j \frac{\partial^j}{\partial t^j} w_{\epsilon_f,J}^{\pm_1}(t) \ll_{j,J} 1
\]
for all fixed $j\in \mathbb{N}_0$.
Hence together with \eqref{eqn:G(y)} and \eqref{eqn:tilde(g)}, we have
 \begin{multline*}
    G^{\pm_1}(y) = \frac{\epsilon_f^{(1\mp_1 1)/2}}{4\pi }
    (\pi^2 y)^{it} (Ny)^{1/2} \left(\frac{NX}{PQ}\right)^{-1/2}
    \int_{\mathbb{R}}  w_{\epsilon_f,J}^{\pm_1}\left(\tau - T\right)
    w_{\epsilon_f,J}^{\pm_1}\left(\tau-T'\right)
    W_1\left(\frac{B^3 q^3  Q^3}{N^{3}X^3}, \frac{\pm \tau} { \frac{NX}{PQ} } , \frac{q}{P}\right)
    \\ \cdot e\left( - \frac{\tau}{2\pi}  \log(\pi^2 yN)
      + \frac{(\tau - T)}{2\pi}\log\frac{T-\tau}{2e}
      + \frac{(\tau - T')}{2\pi}\log\frac{|T'-\tau|}{2e}
      -\frac{3\tau}{2\pi} \log  \frac{\pm \tau } {2\pi e B }
      \right)   \dd \tau + O_A(T^{-A}).
  \end{multline*}
  Making a change of variable $\tau = \pm  B \xi$, we get
  \begin{multline*}
    G^{\pm_1}(y) =  (\pi^2 y)^{it} (Ny)^{1/2} \left(\frac{NX}{PQ}\right)^{1/2} \int_{\mathbb{R}}
      W_2^{\pm_1}\left(\xi, \frac{B} { \frac{NX}{PQ} },\frac{q}{P} \right)
     e\bigg( \mp \frac{B\xi}{2\pi}  \log (\pi^2 yN)
    \mp \frac{3B\xi}{2\pi} \log \frac{\xi}{2\pi e}
    \\
     + \frac{(\pm B \xi  - T )}{2\pi}\log\frac{T\mp B \xi}{2e}
     + \frac{(\pm B \xi - T')}{2\pi}\log\frac{|T'\mp B \xi|}{2e} \bigg) \dd \xi
     + O_A(T^{-A}),
  \end{multline*}
  where $W_2^{\pm_1}\left(\xi, \frac{B} { \frac{NX}{PQ} },\frac{q}{P} \right) =
  \frac{\epsilon_f^{(1\mp_1 1)/2}}{4\pi }  w_{\epsilon_f,J}^{\pm_1}\left(\pm  B \xi - T\right)
    w_{\epsilon_f,J}^{\pm_1}\left(\pm B \xi-T'\right)
    W_1\left(\frac{B^3 q^3  Q^3}{N^{3}X^3}, \frac{B \xi} { \frac{NX}{PQ} } , \frac{q}{P}\right) \frac{B}{\frac{NX}{PQ} }  $ is a $T^\varepsilon$-inert function with compact support in $\mathbb{R}_{>0}^3$.
  Let (temporarily)
  \begin{equation*}
    h(\xi) =\mp \frac{B\xi}{2\pi}  \log (\pi^2 yN)
    \mp \frac{3B\xi}{2\pi} \log \frac{\xi}{2\pi e}
     + \frac{(\pm B \xi  - T )}{2\pi}\log\frac{T\mp B \xi}{2e}
     + \frac{(\pm B \xi - T')}{2\pi}\log\frac{|T'\mp B \xi|}{2e} .
  \end{equation*}
  Then
  \begin{align*}
    h'(\xi)
    & =\mp \frac{B}{2\pi}  \log (\pi^2 yN)
    \mp \frac{3B}{2\pi} \log \frac{\xi}{2\pi}
     + \frac{\pm B}{2\pi}\log\frac{T\mp B \xi}{2}
     + \frac{\pm B}{2\pi}\log\frac{|T'\mp B \xi|}{2} \\
    & =
    \mp \frac{B}{2\pi} \log \frac{yN \xi^3}{2\pi (T\mp B \xi)|T'\mp B \xi|} ,
  \end{align*}
  and
  \[
    h''(\xi) =  \mp \frac{3B}{2\pi \xi} \pm \frac{B}{2\pi}  \frac{ \mp B }{(T\mp B\xi)}
    \pm \frac{B}{2\pi}  \frac{ \mp B }{(T'\mp B\xi)} , \quad
    h^{(j)}(\xi) \asymp_j \frac{NX}{PQ}, \quad j\geq2.
  \]
  By Lemma \ref{lemma:repeated_integration_by_parts} with
  \[
    \textrm{$X=1$, $V=T^{-\varepsilon}$,  $Y=\frac{NX}{PQ}$, $Q=1$, and $R= \frac{NX}{PQ}$,}
  \]
  we have that $G^{\pm_1}(y) \ll_A T^{-A}$ is negligibly small unless $yN \asymp T |T'|$.
  Assume $yN \asymp T |T'|$. Denote the solution of $h'(\xi)=0$ by $\xi_*$ with $\xi_*\asymp 1$. Then
  by Lemma \ref{lemma:stationary_phase} with
  \[
    \textrm{$X= T^\varepsilon$, $t_1=\xi$, $t_2=B$,  $t_3=q$, $t_4=yN$,
    $X_1=1$, $X_2=\frac{NX}{PQ}$, $X_3=P$, $X_4=T|T'|$, and $Y=\frac{NX}{PQ}$,}
  \]
   we get
  \begin{align}\label{eqn:Gpm1}
    G^{\pm_1}(y)
    & =  (\pi^2 y)^{it} (Ny)^{1/2}
     e(h(\xi_*))  W_3^{\pm_1}\left( \frac{B} { \frac{NX}{PQ} },\frac{q}{P},\frac{yN}{T|T'|} \right) + O_A(T^{-A}),
  \end{align}
  where  $W_3^{\pm_1}$  is a $T^\varepsilon$-inert function with compact support  in $\mathbb{R}_{>0}^3$.
  Note that the assumptions in Lemma \ref{lemma:stationary_phase} hold in this case.

  To simplify the expression of $G^{\pm_1}(y)$. Note that
  the solution of $h'(\xi)=0$, i.e., $yN\xi^3=2\pi T|T'|(1\mp \frac{B}{T}\xi)(1\mp \frac{B}{T'}\xi)$, can be written as
  \begin{equation}\label{eqn:xi*}
    \xi_* = \xi_0  + \xi_1 +  \xi_2 + \xi_3+ \cdots + \xi_L + \xi_{L+1},
  \end{equation}
   where $L \geq 3$ is a large integer and $\xi_{\ell+1}=o(\xi_\ell)$ ($0\leq \ell\leq L$) with
  \[
    \xi_0 = \left(\frac{2\pi T|T'|}{yN}\right)^{1/3}, \quad
    \xi_1 = \frac{1}{3} \left(\mp \frac{B}{T}\mp \frac{B}{T'}\right) \xi_0^2, \quad
    \xi_2 = \frac{B^2}{3TT'} \xi_0^3,
  \]
  and $\xi_\ell$ ($2<\ell\leq L$) is the solution of
  \begin{equation*}
    \xi_0^{-3} \sum_{\substack{0\leq i,j,k\leq \ell \\ i+j+k\leq \ell}}  \xi_i\xi_j\xi_k = 1\mp \left(\frac{B}{T}+\frac{B}{T'}\right)\sum_{0\leq i\leq \ell-1} \xi_i + \frac{B^2}{TT'} \sum_{\substack{0\leq i,j\leq \ell-2 \\ i+j\leq \ell-2}}  \xi_i\xi_j.
  \end{equation*}
  Note that
  \[
    \xi_\ell = \frac{\xi_0}{3} \bigg( \mp \left(\frac{B}{T}+\frac{B}{T'}\right) \xi_{\ell-1}
    + \frac{B^2}{TT'} \sum_{\substack{0\leq i,j\leq \ell-2 \\ i+j= \ell-2}}  \xi_i\xi_j - \xi_0^{-3} \sum_{\substack{0\leq i,j,k\leq \ell-1  \\ i+j+k=\ell}}  \xi_i\xi_j\xi_k
    \bigg), \quad  \ell \leq L.
  \]
  By induction we have
  \[
    \xi_\ell = P_{\ell}\left(\frac{B}{T},\frac{B}{T'} \right) \xi_0^{\ell+1} = O_L\left( \frac{B^\ell}{|T'|^\ell} \right), \quad 0\leq \ell\leq L,
    \quad \textrm{and} \quad
    \xi_{L+1} \ll_L \frac{B^{L+1}}{|T'|^{L+1}},
  \]
  where $P_\ell$ is a certain homogeneous polynomial of degree $\ell$.
  Note that
  \begin{align*}
    h(\xi_*) & = \mp \frac{B\xi_*}{2\pi} \log \left( \frac{yN \xi_*^3}{2\pi e (T \mp B \xi_*)|T'\mp B\xi_*|}  \right)
    - \frac{T}{2\pi} \log \frac{T\mp B \xi_*}{2e}
    - \frac{T'}{2\pi} \log \frac{|T'\mp B \xi_*|}{2e}
  \end{align*}
  Note that $\xi_*\asymp 1$ and $B/T=o(1)$. By the Taylor expansion, we get
  \begin{align*}
    h(\xi_*)
    & = - \frac{T}{2\pi} \log \frac{T}{2e}
    - \frac{T'}{2\pi} \log \frac{|T'|}{2e}
    \pm \frac{B}{2\pi} \xi_*
    + \frac{1}{2\pi} \sum_{j\geq1} \frac{1}{j} \left(\frac{(\pm B)^{j}}{T^{j-1}}+\frac{(\pm B)^{j}}{T'^{j-1}}\right)  \xi_*^{j}  \\
    & = - \frac{T}{2\pi} \log \frac{T}{2e}
    - \frac{T'}{2\pi} \log \frac{|T'|}{2e}  \pm \frac{B}{2\pi} \sum_{0\leq \ell\leq L} Q_\ell^\pm \left(\frac{B}{T},\frac{B}{T'} \right) \xi_0^{\ell+1}
    + O_L\left( \frac{B^{L+2}}{|T'|^{L+1}} \right),
  \end{align*}
  where $Q_\ell^\pm$ is a certain homogeneous polynomial of degree $\ell$.
  Note that we have $Q_0^\pm\left(\frac{B}{T},\frac{B}{T'} \right) =3$ and
  $Q_1^\pm\left(\frac{B}{T},\frac{B}{T'} \right) = \mp \frac{1}{2} \left(\frac{B}{T}+\frac{B}{T'}\right) $.
  Hence by \eqref{eqn:Gpm1} we get
  \begin{multline*}
    G^{\pm_1}(y)  =  (\pi^2 y)^{it} (Ny)^{1/2}
      e\bigg(  - \frac{T}{2\pi} \log \frac{T}{2e}
    - \frac{T'}{2\pi} \log \frac{|T'|}{2e}
     \pm \frac{B}{2\pi} \sum_{0\leq \ell\leq L} Q_\ell^\pm \left(\frac{B}{T},\frac{B}{T'} \right) \xi_0^{\ell+1}\bigg)
      \\
      \cdot W_3^{\pm_1}\left( \frac{B} { \frac{NX}{PQ} },\frac{q}{P},\frac{yN}{T|T'|} \right) + O_A(T^{-A}).
  \end{multline*}
  Here we take $L=L(A)$ to be large enough.
  This completes the proof of Lemma \ref{lemma:G>>} (iii).
\end{proof}

By \eqref{eqn:Sr=3}, \eqref{eqn:m-sum==} and Lemma \ref{lemma:G>>} we obtain
\begin{equation}\label{eqn:S<<SM1}
  S_{r}^\pm(N,X,P) \ll T^\varepsilon \sup_{M \asymp \frac{P^2 T|T'|}{N} } |S_{r}^\pm(N,X,P,M)|
  + O(T^{-A})
\end{equation}
if $T^\varepsilon \ll \frac{NX}{PQ}\ll |T'|^{1-\varepsilon}$,
and
\begin{equation}\label{eqn:S<<SM2}
  S_{r}^\pm(N,X,P) \ll T^\varepsilon \sup_{M \ll \frac{P^2 T^2}{N} + \frac{NX^2}{Q^2} } |S_{r}^\pm(N,X,P,M)|
  + O(T^{-A})
\end{equation}
if $ \frac{NX}{PQ}\gg |T'|^{1-\varepsilon}$,
where
\begin{multline*}
  S_{r}^\pm(N,X,P,M)  : =
    X \frac{N^{1/2}  }{Q}\sum_{\substack{q\sim P}} \;\frac{1}{q} U\left(\frac{q}{P}\right) \; \sideset{}{^\star}\sum_{\substack{a\bmod{q} }}
      \frac{NX}{Q}  \sum_{n_1|qr}
      \sum_{\substack{n_2=1 \\n_2 \asymp \frac{r  N^{2}X^3} {n_1^2 Q^3} }}^{\infty}
              \frac{A(n_2,n_1)}{n_1n_2} S\left(-r\bar{a},\pm n_2;\frac{rq}{n_1}\right)
      \\
    \cdot   V\left(\frac{ n_1^2n_2  Q^3}{r  N^{2}X^3}\right)
    \sum_{\pm_1}  \sum_{\substack{m\geq1 \\ m\asymp M}} \frac{ \lambda_{f}(m) }{m^{1/2}} \left(\frac{m}{q^2}\right)^{it}
  e\left(\frac{\pm_1\bar{a} m}{q}\right) W\left( \frac{m}{M}\right)
   \mathcal{I}^{\pm_1}(n_2,n_1,r,m,q) ,
\end{multline*}
where $U,\; V,\; W$ are certain $T^\varepsilon$-inert functions with compact support in $\mathbb{R}_{>0}$ and
\begin{equation}\label{eqn:I-mathcal1}
  \mathcal{I}^{\pm_1}(n_2,n_1,r,m,q) = e\left( \pm \frac{B}{2\pi} \sum_{0\leq \ell \leq L} Q_\ell^\pm \left(\frac{B}{T},\frac{B}{T'} \right) \xi_0^{\ell+1} \right)
\end{equation}
with $B=\frac{ N^{1/3} ( n_1^2n_2)^{1/3} }{ r^{1/3} q} $ and $\xi_0=\left(\frac{2\pi q^2 T|T'|}{mN}\right)^{1/3}$ if $T^\varepsilon \ll \frac{NX}{PQ}\ll |T'|^{1-\varepsilon}$, and
\begin{equation}\label{eqn:I-mathcal2}
  \mathcal{I}^{\pm_1}(n_2,n_1,r,m,q) =  \left( \frac{PQ}{NX}\right)^{1/2}
    \int_{\mathbb{R}}
             \left(\frac{  n_1^2n_2}{ mq} \right)^{i\tau}
             W_1\left(\frac{ n_1^2n_2  Q^3}{r  N^{2}X^3}, \frac{\pm \tau} { \frac{NX}{PQ} } ,\frac{q}{P} \right)   w^{\pm_1}(\tau)  \dd \tau
\end{equation}
if $ \frac{NX}{PQ}\gg |T'|^{1-\varepsilon}$.
Here we  have used the Mellin technique to remove the weight function $W^{\pm_1}_3$ to get \eqref{eqn:I-mathcal1} without writing explicitly the dependence on those new parameters.

Changing the order of summations, we get
\begin{multline}\label{eqn:S(NXPM)}
   S_{r}^\pm(N,X,P,M)  =
    X \frac{N^{1/2}  }{Q}
      \frac{NX}{Q}  \sum_{n_1\ll Pr} \frac{1}{n_1} \sum_{\pm_1}
      \sum_{\substack{n_2=1 }}^{\infty}
              \frac{A(n_2,n_1)}{n_2}   V\left(\frac{ n_1^2n_2  Q^3}{r  N^{2}X^3}\right)\\
    \cdot
              \sum_{\substack{q\sim P \\ n_1\mid qr}} \;\frac{1}{q}U\left(\frac{q}{P}\right)
    \sum_{\substack{m\geq1 \\ m\asymp M }} \frac{ \lambda_{f}(m) }{m^{1/2}} \left(\frac{m}{q^2}\right)^{it} W\left(\frac{m}{M}\right)
  \mathcal{C}^{\pm_1}(n_2,n_1,r,m,q) \mathcal{I}^{\pm_1}(n_2,n_1,r,m,q) ,
\end{multline}
where
\begin{equation}\label{eqn:C-mathcal}
  \begin{split}
    \mathcal{C}^{\pm_1}(n_2,n_1,r,m,q)  & :=   \sideset{}{^\star}\sum_{\substack{a\bmod{q} }}
  e\left(\frac{\pm_1\bar{a} m}{q}\right)
  S\left(-r\bar{a},\pm n_2;\frac{rq}{n_1}\right)
  \\
  & =
  \;  \sideset{}{^\star}\sum_{\alpha \bmod rq/n_1}
  e\left(\frac{\pm n_2\bar\alpha}{rq/n_1}\right)
  \sideset{}{^\star}\sum_{\substack{a\bmod{q} }}
  e\left(\frac{- n_1\bar{a}  \alpha \pm_1\bar{a} m}{q}\right)
  \\
  & =    \sum_{d\mid q} d \mu\left(\frac{q}{d}\right)
  \;  \sideset{}{^\star}\sum_{\substack{\alpha \bmod rq/n_1 \\ \pm_1 m \equiv n_1\alpha  \bmod d}}
  e\left(\frac{\pm n_2 \bar\alpha}{rq/n_1}\right).
  \end{split}
\end{equation}
Here we have used the following identity for the Ramanujan sum
\begin{align*}
   R_q(b) & = \;\sideset{}{^\star}\sum_{a\bmod{q}}
   e\left( \frac{b \overline{a}}{q} \right)
  =
    \sum_{d\mid (q,b)} d \mu\left(\frac{q}{d}\right)
   .
\end{align*}

\subsection{The non oscillating case}
If $\frac{NX}{PQ}\ll T^\varepsilon$, then we have
$X\ll \frac{PQ} {N}T^\varepsilon$ and
\begin{align*}
  S_{r}^\pm(N,X,P)  & =
    \frac{1}{Q}\sum_{\substack{q\sim P}} \int_\mathbb{R} V\left(\frac{\pm x}{X}\right) \;  \sideset{}{^\star}\sum_{\substack{a\bmod{q} }}
    \sum_{\pm} \sum_{n_1|qr} \sum_{n_2=1}^{\infty}
              \frac{A(n_2,n_1)}{n_1n_2} \\
  &  \cdot  S\left(-r\bar{a},\pm n_2;\frac{rq}{n_1}\right)
              \Psi_x^{\pm}\left(\frac{n_1^2n_2}{q^3r}\right)
   \sum_{\substack{m\geq1}} \lambda_f(m) e\left(\frac{m a}{q}\right)
   e\left(\frac{m x}{qQ}\right) m^{-it} W\left(\frac{m}{N}\right)
  \mathrm{d}x  .
\end{align*}
We now apply the Voronoi summation formula (see Lemma \ref{lem:VSF2}) to the sum over $m$ getting
\begin{equation}\label{eqn:m-sum==2}
  \sum_{\substack{m\geq1}} \lambda_f(m) e\left(\frac{m a}{q}\right)
   e\left(\frac{m x}{qQ}\right) m^{-it} W\left(\frac{m}{N}\right)
  \\
  = q \sum_{\pm_1}  \sum_{m\geq1} \frac{ \lambda_{f}(m) }{m}
  e\left(\frac{\pm_1\bar{a} m}{q}\right) G^{\pm_1} \left(\frac{m}{q^2} \right),
\end{equation}
where  $g(m)= e\left(\frac{m x}{qQ}\right) m^{-it} W\left(\frac{m}{N}\right)$
and $G^{\pm_1}$ is defined as in \eqref{eqn:G}.

\begin{lemma}\label{lemma:G<<}
  Assume $x\asymp X$ and $q\sim P$.  If $\frac{NX}{PQ}\ll T^{\varepsilon}$, then we have
   $G^{\pm_1}(y)\ll y^{-6} T^{-A}$  unless $yN\asymp T |T'|$, in which case we have
    $G^{\pm_1}(y) \ll  T^{1/2+\varepsilon} |T'|^{1/2}$.
\end{lemma}

\begin{proof}
  We first consider the case  $yN\gg T^{2+\varepsilon}$. By the same argument as in the proof of Lemma \ref{lemma:G>>} (i), we get $G^{\pm_1}(y)\ll y^{-6} T^{-A}$  if  $yN\gg T^{2+\varepsilon}$.

  Now assume $yN\ll T^{2+\varepsilon}$.
  As in the proof of Lemma \ref{lemma:G>>}, we have
  \begin{equation*}
    G^{\pm_1}(y) = \frac{\epsilon_f^{(1\mp_1 1)/2}}{4\pi^2}  \int_{\mathbb{R}} (\pi^2 y)^{1/2-i\tau+it}
    \gamma_2^{\pm_1}(-1/2+i\tau-it) \tilde{g}(1/2-i\tau+it) \dd \tau.
  \end{equation*}
  If $\frac{NX}{PQ}\ll T^{\varepsilon}$, then
  \begin{align*}
    \tilde{g}(1/2-i\tau+it)
    & = \int_{\mathbb{R}}
    e\left(\frac{u x}{qQ}\right) u^{-it} W\left(\frac{u}{N}\right) u^{1/2-i\tau+it-1} \dd u \\
    & = N^{1/2-i\tau} \int_{\mathbb{R}}
    e\left(\frac{Nx}{qQ} \xi- \frac{1}{2\pi} \tau \log \xi \right)  W\left(\xi\right)  \xi^{-1/2}  \dd \xi .
  \end{align*}
  By  Lemma \ref{lemma:repeated_integration_by_parts} with
  \[
    \textrm{ $X=1$, $V=T^{-\varepsilon}$, $Y=R=|\tau|$ and $Q=1$,}
  \]
  we have $\tilde{g}(1/2-i\tau+it) \ll |\tau|^{-A}$ if $|\tau|\gg T^{2\varepsilon}$.
  By Stirling's formula we have
  \begin{multline*}
    G^{\pm_1}(y) = (\pi^2 y)^{it}  (Ny)^{1/2} \int_{\mathbb{R}} \int_{\mathbb{R}} (\pi^2 y N\xi)^{-i\tau}
    \exp\left( i(\tau - T)\log\frac{T-\tau}{2e} + i(\tau-T')\log\frac{|T'-\tau|}{2e} \right)
    \\ \cdot w_{\epsilon_f,J}^{\pm_1}(\tau-T) w_{\epsilon_f,J}^{\pm_1}(\tau-T')
    U\left( \frac{\tau}{T^{2\varepsilon}} \right)
     \dd \tau  \;
    e\left(\frac{Nx}{qQ} \xi\right) W\left(\xi\right)  \xi^{-1/2}  \dd \xi  + O(T^{-A}),
  \end{multline*}
  where $U$ is a fixed compactly supported smooth function satisfying that  $U^{(j)}(u) \ll_j 1$ for all $j\geq0$, and $U(u)=1$ if $u\in[-1,1]$.
  Let (temporarily)
  \[
    h(\tau) = - \tau \log (\pi^2 y N\xi) + (\tau - T)\log\frac{T-\tau}{2e} + (\tau-T')\log\frac{|T'-\tau|}{2e} .
  \]
  Then we have
  \[
    h'(\tau) = -  \log (\pi^2 y N\xi)
    +  \log\frac{T-\tau}{2}
    +  \log\frac{|T'-\tau|}{2} ,
  \]
  \[
    h''(\tau) =   \frac {-1}{T-\tau}  -  \frac{1} {T'-\tau} , \quad
    h^{(j)}(\tau) \ll |T'|^{-j+1}, \quad j\geq2.
  \]
  Note that the weight function $w(\tau) = w_{\epsilon_f,J}^{\pm_1}(\tau-T) w_{\epsilon_f,J}^{\pm_1}(\tau-T')
    U\left( \frac{\tau}{T^{2\varepsilon}} \right)$ satisfies that
    $w^{(j)}(\tau) \ll T^{-2j\varepsilon}$.
  By Lemma \ref{lemma:repeated_integration_by_parts} with
  \[
    \textrm{ $X=1$, $V=T^{2\varepsilon}$, $Y=Q=|T'|$ and $R=1$,}
  \]
  we have $G^{\pm_1}(y) \ll T^{-A}$ unless $yN\asymp T |T'|$, in which case we have
  $G^{\pm_1}(y) \ll (yN)^{1/2} T^\varepsilon \ll T^{1/2+\varepsilon}|T'|^{1/2}$.
\end{proof}

By \eqref{eqn:Sr=3} and \eqref{eqn:m-sum==2} we have
\begin{align*}
  S_{r}^\pm(N,X,P)
  & =\frac{1}{Q}\sum_{\substack{q\sim P}} \int_\mathbb{R} V\left(\frac{\pm x}{X}\right) \;  \sum_{\pm} \sum_{n_1|qr} \sum_{n_2=1}^{\infty}
              \frac{A(n_2,n_1)}{n_1n_2}
              \Psi_x^{\pm}\left(\frac{n_1^2n_2}{q^3r}\right)\\
  & \hskip 60pt \cdot
  q \sum_{\pm_1}  \sum_{m\geq1} \frac{ \lambda_{f}(m) }{m}
  \mathcal{C}^{\pm_1}(n_2,n_1,r,m,q) G^{\pm_1} \left(\frac{m}{q^2} \right)
  \mathrm{d}x ,
\end{align*}
where $\mathcal{C}^{\pm_1}(n_2,n_1,r,m,q)$ is defined in \eqref{eqn:C-mathcal}. Note that we have
\[
  \mathcal{C}^{\pm_1}(n_2,n_1,r,m,q) \ll
  \sum_{d\mid q} d
  \;  \sideset{}{^\star}\sum_{\substack{\alpha \bmod rq/n_1 \\ \pm_1 m \equiv n_1\alpha  \bmod d}} 1
  \ll rq^{1+\varepsilon}.
\]
By Lemmas \ref{lemma:Psi} and \ref{lemma:G<<}, we obtain
\begin{align*}
  S_{r}^\pm(N,X,P)  & \ll
    \frac{N^\varepsilon}{Q}\sum_{\substack{q\sim P}} rq X
   \sum_{n_1|qr} \frac{1}{n_1} \sum_{n_2 \ll \frac{rP^3}{Nn_1^2}T^\varepsilon}
              \frac{|A(n_2,n_1)|}{n_2}
  q   \sum_{m\asymp \frac{P^2 T |T'|}{N}} \frac{ |\lambda_{f}(m)| }{m} T^{1/2} |T'|^{1/2}
  + O(T^{-A}).
\end{align*}
By \eqref{eqn:RS2} and \eqref{eqn:RS3-1} we get
\begin{align*}
  S_{r}^\pm(N,X,P)  & \ll N^\varepsilon
    \frac{1}{Q}\sum_{\substack{q\sim P}} rq X
   \sum_{n_1|qr} \frac{1}{n_1} n_1^{\theta_3 }
  q  T^{1/2} |T'|^{1/2} + O(T^{-A})  \nonumber \\
  & \ll N^\varepsilon   \frac{rP^3 XT^{1/2} |T'|^{1/2}}{Q} + O(T^{-A}).
\end{align*}
Note that by our assumption we have $X\ll \frac{PQ} {N}T^\varepsilon$. Hence we get
\begin{align}\label{eqn:SrN<<1}
  S_{r}^\pm(N,X,P)  &
   \ll   N^\varepsilon  \frac{rP^3  PQ T^{1/2} |T'|^{1/2}}{QN}
   \ll N^\varepsilon   \frac{r Q^4 T^{1/2} |T'|^{1/2}}{N} \nonumber \\
  &
  \ll N^{1/2+\varepsilon} \left(T^{7/8} |T'|^{19/40} +  T^{57/56} |T'|^{17/56} \right),
\end{align}
provided $N\ll \frac{T^{3/2+\varepsilon}|T'|^{3/2}}{r^2}$ and $Q=\sqrt{\frac{N}{K}}$ with
\begin{equation}\label{eqn:K>>}
  K\geq  \left\{ \begin{array}{ll}
    T^{3/16}|T'|^{31/80}, & \textrm{if $T'\gg T^{5/6}$,} \\
    T^{13/112}|T'|^{53/112}, & \textrm{if $T^{3/5+\varepsilon} \ll T'\ll T^{5/6}$.}
   \end{array}\right.
\end{equation}

\section{Applying Cauchy and Poisson}\label{sec:Cauchy-Poisson}

Assume $\frac{NX}{PQ}\gg T^\varepsilon$.
Write  $q=q_1q_2$ with $\frac{n_1}{(n_1,r)}\mid q_1 \mid (n_1r)^\infty$ and $(q_2,n_1r)=1$.
By \eqref{eqn:S(NXPM)} we have
\begin{multline*}
   S_{r}^\pm(N,X,P,M)  =
    \frac{ Q }{r N^{1/2} P X  M^{1/2}}
    \sum_{\pm_1} \sum_{n_1\ll Pr}  n_1
      \sum_{\substack{n_2=1 }}^{\infty}
        A(n_2,n_1)   V\left(\frac{ n_1^2n_2  Q^3}{r  N^{2}X^3}\right)
              \frac{r  N^{2}X^3}{ n_1^2n_2  Q^3}
      \sum_{\frac{n_1}{(n_1,r)}\mid q_1 \mid (n_1r)^\infty} \\
    \cdot
      \sum_{\substack{q_2\sim P/q_1 \\ (q_2,n_1r)=1}} \;\frac{P}{q}U\left(\frac{q}{P}\right)q^{-2it}
    \sum_{\substack{m\geq1 \\ m\asymp M }} \frac{ \lambda_{f}(m) M^{1/2}}{m^{1/2}} m^{it} W\left(\frac{m}{M}\right)
  \mathcal{C}^{\pm_1}(n_2,n_1,r,m,q) \mathcal{I}^{\pm_1}(n_2,n_1,r,m,q) .
\end{multline*}
Since the cases $\pm_1=+$ and $\pm_1=-$ can be estimated in the same way, we do not write down $\pm_1$ explicitly from now on.  Hence we have
\begin{align*}
  S_{r}^\pm(N,X,P,M)
  &\ll  \frac{ Q }{r N^{1/2} P X  M^{1/2}}
   \sum_{n_1\ll Pr} n_1
     \sum_{\frac{n_1}{(n_1,r)}\mid q_1 \mid (n_1r)^\infty}  \sum_{\substack{n_2\geq 1 \\n_2 \asymp \frac{r  N^{2}X^3} {n_1^2 Q^3} }}
             |A(n_2,n_1)|  |\mathcal B(n_2,n_1,q_1)| ,
\end{align*}
where
\begin{equation*}
  \mathcal B(n_2,n_1,q_1) = \sum_{\substack{q_2\sim P/q_1 \\ (q_2,n_1r)=1}}  b_{q_2}
    \sum_{\substack{ m\asymp M}} c_m
    \; \mathcal{C}(n_2,n_1,r,m,q_1q_2)  \mathcal{I}(n_2,n_1,r,m,q_1q_2)
\end{equation*}
with  $b_{q_2}=\frac{P}{q_1q_2}U\left(\frac{q_1q_2}{P}\right)(q_1q_2)^{-2it}$ and $c_m=\frac{ \lambda_{f}(m) M^{1/2}}{m^{1/2}} m^{it} W\left(\frac{m}{M}\right)$ such that
\begin{equation}\label{eqn:b&c}
  |b_{q_2}| \ll 1 \quad \textrm{and} \quad
  \sum_{m\sim M} |c_m|^2 \ll M.
\end{equation}
Here we have used  \eqref{eqn:RS2} for the $m$-sum.
By the Cauchy--Schwarz inequality, we get
\begin{multline*}
  S_{r}^\pm(N,X,P,M)
  \ll
      \frac{ Q }{r N^{1/2} P X  M^{1/2}}
      \sum_{n_1} \left(\sum_{\substack{n_2}}
             |A(n_2,n_1)|^2  \right)^{1/2}
             n_1
     \sum_{q_1}
     \left(\sum_{\substack{n_2}}|\mathcal B(n_2,n_1,q_1)|^2\right)^{1/2}
   \\
   \ll
      \frac{ Q }{r N^{1/2} P X  M^{1/2}}
      \sum_{n_1} \left(\sum_{\substack{n_2}}
             |A(n_2,n_1)|^2  \right)^{1/2}
             n_1 \left(\sum_{q_1}1\right)^{1/2}
     \left(\sum_{q_1} \sum_{\substack{n_2}}|\mathcal B(n_2,n_1,q_1)|^2\right)^{1/2} .
\end{multline*}
Note that we have
\[
  \sum_{q_1}1 \leq \sum_{\substack{\frac{n_1}{(n_1,r)}\mid q_1 \mid (n_1r)^\infty \\ q_1\ll P}} 1 \ll P^\varepsilon \sum_{q_1 \mid (n_1r)^\infty} q_1^{-\varepsilon}
  =P^\varepsilon \prod_{p \mid n_1r} \sum_{k=1}^\infty p^{-k\varepsilon}
  \ll P^\varepsilon \exp\left(\sum_{p \mid n_1r} O(p^{-\varepsilon}) \right) \ll N^\varepsilon.
\]
Here we have used $\sum_{p \mid n_1r} O(p^{-\varepsilon})=O(\omega(n_1r))=o(\log n_1r)+O(1) = o(\log N)$. Hence we have
\begin{multline} \label{eqn:S<<Omega}
   S_{r}^\pm(N,X,P,M) \ll
     \frac{ N^\varepsilon Q }{r N^{1/2} P X  M^{1/2}}
      \left(\sum_{n_1}\sum_{\substack{n_2}}
             |A(n_2,n_1)|^2  \right)^{1/2}
        \left( \sum_{n_1}     n_1^2
     \sum_{q_1} \sum_{\substack{n_2}}|B(n_2,n_1,q_1)|^2 \right)^{1/2}
   \\
  \ll \frac{ N^\varepsilon Q }{r N^{1/2} P X  M^{1/2}}
      \left(\frac{r  N^{2}X^3} {Q^3} \right)^{1/2}
      \left(\sum_{n_1\ll Pr}   n_1^2
   \sum_{\frac{n_1}{(n_1,r)}\mid q_1 \mid (n_1r)^\infty}  \Omega_\pm \right)^{1/2} ,
   \hskip 30pt
\end{multline}
where
\begin{equation*}
  \Omega_\pm =
   \sum_{\substack{n_2\geq1}}   W\left(\frac{n_2}{N_2}\right)
  \bigg| \sum_{\substack{q_2\sim P/q_1 \\ (q_2,n_1r)=1}}  b_{q_2}
    \sum_{\substack{m\geq1 \\ m\asymp M}} c_m
    \; \mathcal{C}(n_2,n_1,r,m,q_1q_2) \mathcal{I}(n_2,n_1,r,m,q_1q_2) \bigg|^2,
\end{equation*}
where $N_2 =  \frac{r  N^{2}X^3} {n_1^2 Q^3} $ and $W$ is a $1$-inert function with compact support in $\mathbb{R}_{>0}$.
Opening the square and rearranging the sums, we get
\begin{multline*}
  \Omega_\pm =  \sum_{\substack{q_2\sim P/q_1 \\ (q_2,n_1r)=1}}  b_{q_2}
    \sum_{\substack{m\geq1 \\ m\asymp M}} c_m
  \sum_{\substack{q_2'\sim P/q_1 \\ (q_2',n_1r)=1}} \overline{ b_{q_2'} }
    \sum_{\substack{m'\geq1 \\ m'\asymp M}} \overline{c_{m'}}
   \sum_{\substack{n_2\geq1}}   W\left(\frac{n_2}{N_2}\right)
    \\
    \cdot
   \; \mathcal{C}(n_2,n_1,r,m,q_1q_2) \overline{\mathcal{C}(n_2,n_1,r,m',q_1q_2')}
   \; \mathcal{I}(n_2,n_1,r,m,q_1q_2) \overline{\mathcal{I}(n_2,n_1,r,m',q_1q_2')}.
\end{multline*}
Applying Poisson summation on the sum over $n_2$ modulo $ rq_1 q_2q_2'/n_1 $ (Lemma \ref{lem:Poisson}), we arrive at
\begin{equation}\label{eqn:Omega=}
  \Omega_\pm =
    \sum_{\substack{q_2\sim P/q_1 \\ (q_2,n_1r)=1}}  b_{q_2}
    \sum_{\substack{m\geq1 \\ m\asymp M}} c_m
  \sum_{\substack{q_2'\sim P/q_1 \\ (q_2',n_1r)=1}} \overline{ b_{q_2'} }
    \sum_{\substack{m'\geq1 \\ m'\asymp M}} \overline{c_{m'}}
   \sum_{\substack{n\in \mathbb{Z}}}  \mathfrak{C}(n) \mathfrak{I}(n),
\end{equation}
where
\begin{align*}
  \mathfrak{C}(n)
  & = \frac{1}{rq_1 q_2q_2'/n_1} \sum_{\beta(rq_1 q_2q_2'/n_1)} \mathcal{C}(\beta,n_1,r,m,q_1q_2) \overline{\mathcal{C}(\beta,n_1,r,m',q_1q_2')}
  e\left( \frac{n\beta}{rq_1 q_2q_2'/n_1}\right) \\
  &  = \sum_{d\mid q} d \mu\left(\frac{q}{d}\right)
  \sum_{d'\mid q'} d' \mu\left(\frac{q'}{d'}\right)\;
  \mathop{ \sideset{}{^\star}\sum_{\substack{\alpha \bmod rq/n_1 \\ \pm_1 m \equiv n_1\alpha  \bmod d}}
  \;  \sideset{}{^\star}\sum_{\substack{\alpha' \bmod rq'/n_1 \\ \pm_1 m' \equiv n_1\alpha'  \bmod d'}} }_{   \pm q_2' \bar\alpha \mp q_2 \bar\alpha' \equiv -n \bmod rq_1 q_2q_2'/n_1 } 1
\end{align*}
and
\begin{align}\label{eqn:I(n)}
  \mathfrak{I}(n)
  &= \int_{\mathbb{R}}   W\left(\frac{u}{N_2}\right)
    \; \mathcal{I}(u,n_1,r,m,q_1q_2) \overline{\mathcal{I}(u,n_1,r,m',q_1q_2')}
       e\left(\frac{-nu}{rq_1 q_2q_2'/n_1} \right) \dd u \nonumber \\
  &= N_2 \int_{\mathbb{R}}   W\left(\xi\right)
     \; \mathcal{I}(N_2\xi,n_1,r,m,q_1q_2) \overline{\mathcal{I}(N_2\xi,n_1,r,m',q_1q_2')}
       e\left(\frac{-nN_2 \xi}{rq_1 q_2q_2'/n_1} \right) \dd \xi  .
\end{align}

The following lemma on the character sums is essentially due to Munshi \cite{Munshi2018} and is actually the same as the results in Huang--Xu \cite{HuangXu}.
\begin{lemma}\label{lem:C}
  We have  $\mathfrak{C}(0)=0$ unless
        \[ q = q',\]
  in which case we have
  \[
    \mathfrak{C}(0)
        \ll \mathop{\sum_{d\mid q}  \sum_{d'\mid q}}_{(d,d')\mid (m-m')} (d,d') qr.
  \]
  If $n\neq0$, then we have
  \begin{multline*}
    \mathfrak{C}(n) \ll  \frac{rq_1}{n_1} \sum_{d_1\mid q_1}
  \sum_{d_1'\mid q_1}  \min\left\{ d_1' (d_1,n_1) \delta_{(d_1,n_1)\mid m},
  d_1  (d_1',n_1) \delta_{(d_1',n_1)\mid m'} \right\} \\
  \cdot
     \mathop{\sum\sum}_{\substack{d_2 \mid (q_2, \pm n_1 q_2'\pm_1 mn) \\
    d_2' \mid (q_2', \mp n_1 q_2 \pm_1 m'n)}} d_2d_2' \;
    \min\left\{\frac{q_2}{[q_2/(q_2,q_2'),d_2]},\frac{q_2'}{[q_2'/(q_2,q_2'),d_2']} \right\} \; \delta_{(q_2,q_2')\mid n}.
  \end{multline*}
  Here $\delta_{(d_1,n_1)\mid m}=1$ if $(d_1,n_1)\mid m$ holds,  otherwise we have $\delta_{(d_1,n_1)\mid m}=0$.
\end{lemma}

\begin{proof}
  If $n=0$, then $ \pm q_2' \bar\alpha \mp q_2 \bar\alpha' \equiv 0 \bmod rq_1 q_2q_2'/n_1$. Since $(\alpha,rq_1 q_2/n_1)=(\alpha',rq_1 q_2'/n_1)=1$, we have $q_2=q_2'$ and then $ \pm \bar\alpha \mp \bar\alpha' \equiv 0 \bmod rq_1 q_2/n_1$, which give $\alpha\equiv \alpha' \bmod rq_1 q_2/n_1$. Hence
  \[
    \mathfrak{C}(0)
    \ll \sum_{d\mid q} d   \sum_{d'\mid q} d'
    \sideset{}{^\star}\sum_{\substack{\alpha \bmod rq/n_1 \\ \pm_1 m \equiv n_1\alpha  \bmod d \\ \pm_1 m' \equiv n_1\alpha  \bmod d'}} 1
        \ll \mathop{\sum_{d\mid q}  \sum_{d'\mid q}}_{(d,d')\mid (m-m')} (d,d') qr.
  \]

  If $n\neq0$, then by the Chinese Remainder Theorem,
  we have $|\mathfrak{C}(n)| \leq \mathfrak{C}_1(n)  \mathfrak{C}_2(n) $, where
  \begin{equation*}
    \mathfrak{C}_1(n) = \sum_{d_1\mid q_1} d_1
  \sum_{d_1'\mid q_1} d_1'
  \mathop{ \sideset{}{^\star}\sum_{\substack{\alpha_1 \bmod rq_1/n_1 \\ \pm_1 m \equiv n_1\alpha_1   \bmod d_1}}
  \;  \sideset{}{^\star}\sum_{\substack{\alpha_1' \bmod rq_1/n_1 \\ \pm_1 m' \equiv n_1\alpha_1'  \bmod d_1'}} }_{   \pm q_2' \bar\alpha_1 \mp q_2 \bar\alpha_1' \equiv -n \bmod rq_1/n_1 } 1
  \end{equation*}
  and
  \begin{equation*}
    \mathfrak{C}_2(n) = \sum_{d_2\mid q_2} d_2
  \sum_{d_2'\mid q_2'} d_2'
  \mathop{ \sideset{}{^\star}\sum_{\substack{\alpha_2 \bmod q_2 \\ \pm_1 m \equiv n_1\alpha_2   \bmod d_2}}
  \;  \sideset{}{^\star}\sum_{\substack{\alpha_2' \bmod q_2' \\ \pm_1 m' \equiv n_1\alpha_2'  \bmod d_2'}} }_{   \pm q_2' \bar\alpha_2 \mp q_2 \bar\alpha_2' \equiv -n \bmod q_2q_2' } 1.
  \end{equation*}

  We first consider $\mathfrak{C}_2(n)$.
  From the congruence $\pm q_2' \bar\alpha_2 \mp q_2 \bar\alpha_2' \equiv -n \bmod q_2q_2' $ we have $(q_2,q_2')\mid n$ and $\pm \frac{q_2'}{(q_2,q_2')}\bar{\alpha_2} \equiv -\frac{n}{(q_2,q_2')} \bmod \frac{q_2}{(q_2,q_2')}$. Since $(n_1,q_2)=1$, we have $\alpha_2 \equiv \pm_1 m \bar{n}_1 \bmod d_2$ and also $\pm q_2' \bar\alpha_2 \equiv -n \bmod d_2$.
  Therefore we get $\pm n_1 q_2'\pm_1 mn \equiv 0 \bmod d_2$. Similarly we have $\mp n_1 q_2 \pm_1 m'n \equiv 0 \bmod d_2'$.
  Note that the congruences $ \bmod\; \frac{q_2}{(q_2,q_2')}$ and $\bmod \; d_2$ determine $\alpha_2 \bmod [q_2/(q_2,q_2'),d_2]$. For each given $\alpha_2$ we have at most one solution of $\alpha_2' \bmod q_2'$. Hence we have
  \[
    \mathfrak{C}_2(n) \ll  \mathop{\sum\sum}_{\substack{d_2 \mid (q_2, \pm n_1 q_2'\pm_1 mn) \\
    d_2' \mid (q_2', \mp n_1 q_2 \pm_1 m'n)}} d_2d_2' \; \frac{q_2}{[q_2/(q_2,q_2'),d_2]} \; \delta_{(q_2,q_2')\mid n}.
  \]
  Similarly we have
  \[
    \mathfrak{C}_2(n) \ll \mathop{\sum\sum}_{\substack{d_2 \mid (q_2, \pm n_1 q_2'\pm_1 mn) \\
    d_2' \mid (q_2', \mp n_1 q_2 \pm_1 m'n)}} d_2d_2' \; \frac{q_2'}{[q_2'/(q_2,q_2'),d_2']} \; \delta_{(q_2,q_2')\mid n}.
  \]

  In $\mathfrak{C}_1(n)$, for each value of $\alpha_1$, the congruence condition $\mod rq_1/n_1$ determines the value of $\alpha_1'$,
  and hence we have
  \begin{equation*}
    \mathfrak{C}_1(n) \leq  \sum_{d_1\mid q_1} d_1
  \sum_{d_1'\mid q_1} d_1'
    \sideset{}{^\star}\sum_{\substack{\alpha_1 \bmod rq_1/n_1 \\ \pm_1 m \equiv n_1\alpha_1   \bmod d_1 }}  1
    .
  \end{equation*}
  Note that $\alpha_1$ is uniquely determined modulo $d_1/(d_1,n_1)$. Since $(\frac{d_1}{(d_1,n_1)},\frac{n_1}{(d_1,n_1)})=1$,  $\frac{d_1}{(d_1,n_1)}\mid \frac{q_1}{(d_1,n_1)}$ and $\frac{n_1}{(d_1,n_1)}\mid \frac{rq_1}{(d_1,n_1)}$, we have
  $\frac{d_1}{(d_1,n_1)}\mid \frac{rq_1}{n_1}$. Also $\pm_1 m \equiv n_1\alpha_1   \bmod d_1$ has solutions only if $(d_1,n_1)\mid m$. Hence we get
  \[
    \mathfrak{C}_1(n)  \ll \frac{rq_1}{n_1} \sum_{d_1\mid q_1}
  \sum_{d_1'\mid q_1} d_1' (d_1,n_1) \delta_{(d_1,n_1)\mid m}.
  \]
  Similarly by considering $\alpha_1$-sum first we have
  \[
    \mathfrak{C}_1(n)  \ll \frac{rq_1}{n_1} \sum_{d_1\mid q_1}
  \sum_{d_1'\mid q_1} d_1 (d_1',n_1) \delta_{(d_1',n_1)\mid m'}.
  \]
  This completes the proof of the lemma.
\end{proof}

We will also need bounds for $\mathfrak{I}(n)$. In Lemma \ref{lem:I<<} below we give bounds when $T^\varepsilon \ll \frac{NX}{PQ}\ll |T'|^{1-\varepsilon}$. In this case, in order to get a better bound for counting, we also need to find a condition for $m$ and $m'$ such that $\mathfrak{I}(0)$ is not negligibly small.
In Lemma \ref{lem:I>>} below we consider the case $ \frac{NX}{PQ}\gg |T'|^{1-\varepsilon}$. In this case, we only give  relatively easy bounds based on $L^2$-norm bounds for $\mathcal{I}$, which is good enough to prove our uniform bounds.

\begin{lemma}\label{lem:I<<}
  Assume $T^\varepsilon \ll \frac{NX}{PQ}\ll |T'|^{1-\varepsilon}$.  Then we have
  \begin{itemize}
    \item[(i)] For any $n\in \mathbb{Z}$, we have
       \[
          \mathfrak{I}(n)
          \ll N_2;
       \]
    \item[(ii)] If $n\gg \frac{P Q^2 n_1}{q_1 NX^2} T^\varepsilon$, then we have $\mathfrak{I}(n) \ll n^{-6} T^{-A}$;
    \item[(iii)] If $ \frac{N n_1}{q_1 P  |T'|^2}  T^\varepsilon + \frac{P^2 Q^3 n_1}{q_1 N^2 X^3} N^\varepsilon  \ll |n|\ll \frac{P Q^2 n_1}{q_1 NX^2} T^\varepsilon $, then
       \[
          \mathfrak{I}(n)
          \ll   N_2 \left( \frac{|n|N_2}{rq_1 q_2q_2'/n_1} \right)^{-1/2};
       \]
    \item [(iv)] If $q=q'$, then we have $\mathfrak{I}(0)
          \ll T^{-A}$ unless $m-m' \ll  M \left(\frac{PQ}{NX} +\left(\frac{NX}{PQ}\right)^2 |T'|^{-2} \right) T^{\varepsilon} $.
  \end{itemize}
\end{lemma}

\begin{proof}
  (i)
  By \eqref{eqn:I-mathcal1} and \eqref{eqn:I(n)}, we have
  \begin{multline*}
  \mathfrak{I}(n)
   = N_2 \int_{\mathbb{R}}   W\left(\xi\right)
     e\bigg(
      \pm \frac{B \xi^{1/3}}{2\pi} \sum_{0\leq \ell \leq L} Q_\ell^\pm \left(\frac{B\xi^{1/3}}{T},\frac{B\xi^{1/3}}{T'} \right) \xi_0^{\ell+1}
   \\
      \mp \frac{B'\xi^{1/3}}{2\pi} \sum_{0\leq \ell \leq L} Q_\ell^\pm \left(\frac{B'\xi^{1/3}}{T},\frac{B'\xi^{1/3}}{T'} \right) \xi_0'^{\ell+1} - \frac{nN_2 \xi}{rq_1 q_2q_2'/n_1} \bigg) \dd \xi  ,
  \end{multline*}
  where
  $B=\frac{ N^{1/3} ( n_1^2N_2)^{1/3} }{ r^{1/3} q} $,
  $\xi_0 = \left(\frac{2\pi q^2 T|T'|}{mN}\right)^{1/3}$,
  $B'=\frac{ N^{1/3} ( n_1^2N_2)^{1/3} }{ r^{1/3} q'} $,
  $\xi_0' = \left(\frac{2\pi q'^2 T|T'|}{m'N}\right)^{1/3}$.
  This gives $\mathfrak{I}(n)\ll N_2$, and hence proves (i).

  (ii)
  Let $h(\xi)$ be the phase function above.
  Recall that $B\asymp B'\asymp \frac{NX}{PQ}$ and $\xi_0\asymp \xi_0'\asymp 1$.
  If $\frac{nN_2}{rq_1 q_2q_2'/n_1} \gg \frac{NX}{PQ} T^{\varepsilon}$, then we have $h'(\xi) \gg \frac{|n|N_2}{rq_1 q_2q_2'/n_1}$ and $h^{(j)}(\xi) \ll \frac{NX}{PQ}$ for $j\geq2$.
  By Lemma \ref{lemma:repeated_integration_by_parts} 
   with
  \[
    \textrm{$X=V=1$,  $Y=\frac{NX}{PQ}$, $Q=1$, and $R= \frac{|n| N_2}{rq_1 q_2q_2'/n_1}$,}
  \]
  we have
  $\mathfrak{I}(n) \ll n^{-6} T^{-A}$ if $\frac{nN_2}{rq_1 q_2q_2'/n_1} \gg \frac{NX}{PQ}T^{\varepsilon}$. Since $q=q_1q_2\sim P$ and $N_2=\frac{r  N^{2}X^3} {n_1^2 Q^3} $, this is equivalent to $n\gg \frac{P Q^2 n_1}{q_1 NX^2} T^\varepsilon$.

  (iii)
  Note that
  \begin{equation*}
    h'(\xi) = \pm \frac{1}{2\pi } \left( B\xi_0 - B'\xi_0' \right) \xi^{-2/3}
    - \frac{\left(\frac{1}{T}+\frac{1}{T'}\right)}{6\pi }
       \left( B^2\xi_0^2 - B'^2\xi_0'^2 \right) \xi^{1/3}  - \frac{nN_2 }{rq_1 q_2q_2'/n_1}+ O\left(\frac{(\frac{NX}{PQ})^3}{|T'^2|}\right).
  \end{equation*}
  If $(1+(\frac{NX}{PQ})^3/|T'|^2)T^{\varepsilon}  \ll \frac{nN_2}{rq_1 q_2q_2'/n_1} \ll \frac{NX}{PQ}T^{\varepsilon}$, i.e.
  $ \frac{N n_1}{q_1 P  |T'|^2}  T^\varepsilon + \frac{P^2 Q^3 n_1}{q_1 N^2 X^3} N^\varepsilon  \ll |n|\ll \frac{P Q^2 n_1}{q_1 NX^2} T^\varepsilon $,
  then we have
  $h'(\xi) \gg |B\xi_0 - B'\xi_0'| + |\frac{nN_2}{rq_1 q_2q_2'/n_1}|$
  unless
  $|B\xi_0 - B'\xi_0'| \asymp |\frac{nN_2}{rq_1 q_2q_2'/n_1}|$
  and
  $h^{(j)}(\xi) \ll |B\xi_0 - B'\xi_0'| + (\frac{NX}{PQ})^3/|T'|^2$ for $j\geq2$.
  By Lemma \ref{lemma:repeated_integration_by_parts} 
   with
  \[
    \textrm{$X=V=1$,  $Y=|B\xi_0 - B'\xi_0'| + (\frac{NX}{PQ})^3/|T'|^2$, $Q=1$, and $R= |B\xi_0 - B'\xi_0'| + |\frac{nN_2}{rq_1 q_2q_2'/n_1}|$,}
  \]
  we have
  $\mathfrak{I}(n) \ll T^{-A}$ unless $|B\xi_0 - B'\xi_0'| \asymp |\frac{nN_2}{rq_1 q_2q_2'/n_1}|$, in which case we have
  $h^{(j)}(\xi) \asymp |\frac{nN_2}{rq_1 q_2q_2'/n_1}| $ for $j\geq2$. Hence by
  Lemma \ref{lemma:stationary_phase} with
  \[
    \textrm{$X=1$, $t_1=\xi$,  $X_1=1$, and
    $Y=\frac{|n|N_2}{rq_1 q_2q_2'/n_1}$,}
  \]
  we get
  $\mathfrak{I}(n)
          \ll   N_2 \left( \frac{|n|N_2}{rq_1 q_2q_2'/n_1} \right)^{-1/2}$.

  (iv) If $q=q'$ then $B=B'$ and
   \begin{multline*}
   \mathfrak{I}(0)
   = N_2 \int_{\mathbb{R}}   W\left(\xi\right)
     e\bigg( \pm \frac{B\xi_0}{ 2\pi }
     \left(1- \frac{m^{1/3}}{m'^{1/3}}\right) \xi^{1/3}
      - \frac{B^2 \xi_0^2 }{ 4\pi } \left(\frac{1}{T}+\frac{1}{T'}\right)  \left(1- \frac{m^{2/3}}{m'^{2/3}}\right) \xi^{2/3}
      \\
      \pm \frac{B \xi^{1/3}}{2\pi} \sum_{2\leq \ell \leq L} Q_\ell^\pm \left(\frac{B\xi^{1/3}}{T},\frac{B\xi^{1/3}}{T'} \right) (\xi_0^{\ell+1} - \xi_0'^{\ell+1})  \bigg) \dd \xi .
  \end{multline*}
  By Lemma \ref{lemma:repeated_integration_by_parts} as above, we have $\mathfrak{I}(0) \ll T^{-A}$ unless
  \[
    B \left(1- \frac{m^{1/3}}{m'^{1/3}}\right) \ll \left(1+\left(\frac{NX}{PQ}\right)^3/|T'|^2 \right) T^{\varepsilon},
  \]
  that is,
  \[
    m-m' \ll  M \left(\frac{PQ}{NX} +\left(\frac{NX}{PQ}\right)^2 |T'|^{-2} \right) T^{\varepsilon} .
  \]
  This completes the proof of Lemma \ref{lem:I<<}.
\end{proof}

\begin{lemma}\label{lem:I>>}
  Assume $ \frac{NX}{PQ}\gg |T'|^{1-\varepsilon}$. Then we have
  \begin{itemize}
    \item[(i)] If $n\gg \frac{P Q^2 n_1}{q_1 NX^2} N^\varepsilon$, then we have $\mathfrak{I}(n) \ll n^{-6} T^{-A}$;
    \item[(ii)] For  any $n\in\mathbb{Z}$, we have
       \[
          \mathfrak{I}(n)
          \ll   N_2 T^\varepsilon.
       \]
  \end{itemize}
\end{lemma}

\begin{proof}
  (i) By \eqref{eqn:I-mathcal2} and \eqref{eqn:I(n)}, we have
  \begin{multline*}
    \mathfrak{I}(n)
    = N_2 \int_{\mathbb{R}}   W\left(\xi\right)
      \left( \frac{PQ}{NX}\right)^{1/2}
    \int_{\mathbb{R}}
    \left(\frac{  n_1^2N_2\xi}{ mq} \right)^{i\tau}
    W_1\left(\frac{ n_1^2  Q^3}{r  N^{2}X^3}N_2\xi, \frac{\pm \tau} { \frac{NX}{PQ} } ,\frac{q}{P} \right)   w^{\pm_1}(\tau)  \dd \tau
     \\
    \cdot  \left( \frac{PQ}{NX}\right)^{1/2}
    \int_{\mathbb{R}} \left(\frac{  n_1^2 N_2\xi}{ m'q'} \right)^{-i\tau}
    \overline{W_1\left(\frac{ n_1^2  Q^3}{r  N^{2}X^3}N_2\xi, \frac{\pm \tau'} { \frac{NX}{PQ} } ,\frac{q'}{P} \right)   w^{\pm_1}(\tau')}
     \dd \tau'
       e\left(\frac{-nN_2 \xi}{rq_1 q_2q_2'/n_1} \right) \dd \xi .
  \end{multline*}
  Changing the order of integration, we get
  \begin{multline*}
    \mathfrak{I}(n)
    = N_2
       \frac{PQ}{NX}
    \int_{\mathbb{R}}\left(\frac{  n_1^2N_2 }{ mq} \right)^{i\tau}
     w^{\pm_1}(\tau)
    \int_{\mathbb{R}}\left(\frac{  n_1^2 N_2}{ m'q'} \right)^{-i\tau'}
    \overline{  w^{\pm_1}(\tau')}
     \\ \cdot
    \int_{\mathbb{R}}   W\left(\xi\right)  \xi^{i\tau-i\tau'}
    e\left(\frac{-nN_2 \xi}{rq_1 q_2q_2'/n_1} \right)
     W_1\left(\frac{ n_1^2  Q^3}{r  N^{2}X^3}N_2\xi, \frac{\pm \tau} { \frac{NX}{PQ} } ,\frac{q}{P} \right)
    \overline{W_1\left(\frac{ n_1^2  Q^3}{r  N^{2}X^3}N_2\xi, \frac{\pm \tau'} { \frac{NX}{PQ} } ,\frac{q'}{P} \right)  }
    \dd \xi
     \dd \tau
     \dd \tau'
      .
  \end{multline*}
  By Lemma \ref{lemma:repeated_integration_by_parts}   in the $\xi$-integral we have
  $\mathfrak{I}(n)\ll n^{-6} T^{-A}$ if $\frac{-nN_2 }{rq_1 q_2q_2'/n_1} \gg \frac{NX}{PQ}T^\varepsilon $, that is,
   \[ n \gg \frac{rP^2}{n_1q_1 } \frac {n_1^2 Q^3} {r  N^{2}X^3}\frac{NX}{PQ}T^\varepsilon
   = \frac{n_1 P Q^2}{q_1 N X^2}  T^\varepsilon . \]
  This proves (i).

  (ii) By the Cauchy inequality we have
  \begin{equation*}
    \mathfrak{I}(n)
    \ll N_2 \left( \int_{\mathbb{R}}   W\left(\xi\right)
     | \mathcal{I}(N_2\xi,n_1,r,m,q_1q_2) |^2   \dd \xi \right)^{1/2}
     \left( \int_{\mathbb{R}}   W\left(\xi\right)
     | \mathcal{I}(N_2\xi,n_1,r,m',q_1q_2') |^2   \dd \xi \right)^{1/2}   .
  \end{equation*}
  By \eqref{eqn:I-mathcal2} we have
  \begin{multline*}
     \int_{\mathbb{R}}   W\left(\xi\right)
     | \mathcal{I}(N_2\xi,n_1,r,m,q_1q_2) |^2   \dd \xi
    =    \frac{PQ}{NX}
    \int_{\mathbb{R}}\left(\frac{  n_1^2N_2 }{ mq} \right)^{i\tau}
     w^{\pm_1}(\tau)
    \int_{\mathbb{R}}\left(\frac{  n_1^2 N_2}{ mq} \right)^{-i\tau'}
    \overline{  w^{\pm_1}(\tau')}
     \\ \cdot
    \int_{\mathbb{R}}   W\left(\xi\right)  \xi^{i\tau-i\tau'}
    W_1\left(\frac{ n_1^2  Q^3}{r  N^{2}X^3}N_2\xi, \frac{\pm \tau} { \frac{NX}{PQ} } ,\frac{q}{P} \right)
    \overline{W_1\left(\frac{ n_1^2  Q^3}{r  N^{2}X^3}N_2\xi, \frac{\pm \tau'} { \frac{NX}{PQ} } ,\frac{q}{P} \right)  }
    \dd \xi
     \dd \tau
     \dd \tau'
      .
  \end{multline*}
  By Lemma \ref{lemma:repeated_integration_by_parts}  in the $\xi$-integral we have
  $\mathfrak{I}(n)\ll T^{-A}$ if $|\tau-\tau'| \gg T^{\varepsilon}$.
  So we have
  \[
     \int_{\mathbb{R}}   W\left(\xi\right)
     | \mathcal{I}(N_2\xi,n_1,r,m,q_1q_2) |^2   \dd \xi  \ll T^\varepsilon ,
  \]
  and hence $\mathfrak{I}(n) \ll N_2 T^\varepsilon$.
\end{proof}

\section{The zero frequency}\label{sec:zero-freq}

In this section, we bound the contribution from $n=0$ in \eqref{eqn:Omega=}. Denote this by $\Omega_0$.

We first deal with the case  $T^\varepsilon \ll \frac{NX}{PQ}\ll |T'|^{1-\varepsilon}$.
By Lemmas \ref{lem:C} and \ref{lem:I<<} we get
\begin{equation*}
  \Omega_0 \ll   P r \frac{r  N^{2}X^3} {n_1^2 Q^3}  \sum_{\substack{q_2\sim P/q_1 \\ (q_2,n_1r)=1}}
  \sum_{d\mid q}  \sum_{d'\mid q}  (d,d') \\
    \sum_{\substack{m\geq1 \\ m\asymp M}} |c_m|^2
    \sum_{\substack{m'\geq1 \\ (d,d')\mid (m-m') \\ m'-m \ll  M \left(\frac{PQ}{NX} +\left(\frac{NX}{PQ}\right)^2 |T'|^{-2} \right) N^{\varepsilon} }} 1.
\end{equation*}
Here we have used $|c_{m} c_{m'}| \leq |c_{m}|^2 + |c_{m'}|^2$ and without loss of generality only give details for the case with $|c_{m}|^2$.
Note that
\[
   \sum_{\substack{m'\geq1 \\ (d,d')\mid (m-m') \\ m'-m \ll M \left(\frac{PQ}{NX} +\left(\frac{NX}{PQ}\right)^2 |T'|^{-2} \right) N^{\varepsilon} }} 1
   \ll N^\varepsilon \left( \frac{1}{(d,d')}  M \left(\frac{PQ}{NX} + \left(\frac{NX}{PQ}\right)^2 |T'|^{-2} \right)    +  1 \right) .
\]
By \eqref{eqn:b&c} we have
\begin{align*}
   \Omega_0
    & \ll N^\varepsilon
     \frac{r^2 P  N^{2}X^3} {n_1^2 Q^3}  \sum_{\substack{q_2\sim P/q_1 \\ (q_2,n_1r)=1}}
  \sum_{d\mid q}  \sum_{d'\mid q}  (d,d')
   M
    \left( \frac{M}{(d,d')}  \left(\frac{PQ}{NX} + \left(\frac{NX}{PQ}\right)^2 |T'|^{-2} \right)  +  1 \right)
   \\
   & \ll N^\varepsilon \frac{1}{n_1^2 q_1} \left( \frac{r^2 P^4  N^2 T^2 X^5} {Q^5}
    +
       \frac{r^2 P^7  T^2 |T'|^2 X^2} {N Q^2}
    +
     \frac{r^2 P^5  N  T |T'| X^3} {Q^3} \right) .
\end{align*}
Here we have used $M\asymp \frac{P^2 T|T'|}{N}$.  Note that
\begin{equation}\label{eqn:sum-n1&q1}
   \sum_{n_1\ll Pr}   n_1^2
   \sum_{\frac{n_1}{(n_1,r)}\mid q_1 \mid (n_1r)^\infty} \frac{1}{n_1^2 q_1}
   \ll N^\varepsilon \sum_{n_1\ll Pr}  \frac{(n_1,r)}{n_1}
   \sum_{q_1 \mid (n_1r)^\infty} \frac{1}{q_1}
   \ll N^\varepsilon.
\end{equation}
By \eqref{eqn:S<<Omega}, the contribution from the zero frequency to $S_{r}^\pm(N,X,P) $ is bounded by
\begin{align*}
   & \ll  N^\varepsilon
       \frac {Q} {r P^2 T^{1/2} |T'|^{1/2} X}
      \left(\frac{r  N^{2}X^3} {Q^3} \right)^{1/2}
      \\
   & \hskip 60pt  \cdot
      \left( \frac{r  P^2  N  T  X^{5/2}} {Q^{5/2}}
    +
      \frac{r  P^{7/2}  T |T'| X} { N^{1/2} Q}
    +
     \frac{r  P^{5/2}  N^{1/2}  T^{1/2} |T'|^{1/2}  X^{3/2}} {Q^{3/2}}  \right)
     \\
    & \ll  N^\varepsilon  \frac{r^{1/2} N^2 T^{1/2}} { Q^{3} |T'|^{1/2} }
    +
      N^\varepsilon   r^{1/2} N^{1/2}  T^{1/2} |T'|^{1/2}
         +
     N^\varepsilon  \frac{r^{1/2} N^{3/2} }
     { Q^{3/2}}.
\end{align*}
Here we have used $X\ll T^\varepsilon$ and $P\leq Q$.  By $Q=\frac{ N^{1/2}}{K^{1/2}}$ and $N\ll \frac{T^{3/2+\varepsilon} |T'|^{3/2}}{r^2}$, the above is bounded by
\begin{equation} \label{eqn:zero-freq-1}
     \ll  N^{1/2+\varepsilon}  \frac{r^{1/2}  K^{3/2} T^{1/2}} {   |T'|^{1/2} }
    +  N^{1/2+\varepsilon}   r^{1/2}  T^{1/2} |T'|^{1/2}
         +  N^{1/2+\varepsilon}  T^{3/8} |T'|^{3/8}  K^{3/4}.
\end{equation}

\medskip

Now we treat the case  $\frac{NX}{PQ}\gg |T'|^{1-\varepsilon}$.
By Lemmas \ref{lem:C} and \ref{lem:I>>} we get
\begin{equation*}
  \Omega_0 \ll   P r \frac{r  N^{2}X^3} {n_1^2 Q^3}  \sum_{\substack{q_2\sim P/q_1 \\ (q_2,n_1r)=1}}
  \sum_{d\mid q}  \sum_{d'\mid q}  (d,d') \\
    \sum_{\substack{m\asymp M}} |c_m|^2
    \sum_{\substack{m' \asymp M \\ (d,d')\mid (m-m') }} 1.
\end{equation*}
Note that
\[
   \sum_{\substack{m' \asymp M \\ (d,d')\mid (m-m') }} 1
   \ll   \frac{1}{(d,d')} M
    +  1   .
\]
By \eqref{eqn:b&c} we have
\begin{align*}
   \Omega_0
   & \ll N^\varepsilon
     \frac{r^2 P  N^{2}X^3} {n_1^2 Q^3}
     \sum_{\substack{q_2\sim P/q_1 \\ (q_2,n_1r)=1}}
  \sum_{d\mid q}  \sum_{d'\mid q}  (d,d')
   M
    \left( \frac{1}{(d,d')} M
     +  1 \right)
    \\
    & \ll
     \frac{ N^\varepsilon } {n_1^2 q_1 }
    \left( \frac{r^2 P^2  N^{2}X^3 M^2 } { Q^3}
    + \frac{r^2 P^3  N^{2}X^3 M} {Q^3} \right) .
\end{align*}
By \eqref{eqn:sum-n1&q1} and \eqref{eqn:S<<Omega}, the contribution from the zero frequency to $S_{r}^\pm(N,X,P,M) $ is bounded by
\begin{align*} 
  & \ll  N^\varepsilon \frac{  Q }{r N^{1/2} P X  M^{1/2}}
      \left(\frac{r^{1/2}  N X^{3/2}} {Q^{3/2}} \right)
      \left( \frac{r P N X^{3/2} M  } { Q^{3/2}}
             + \frac{r  P^{3/2}  N X^{3/2} M^{1/2}} {Q^{3/2}} \right)
     \\
   &  \ll  N^\varepsilon \frac{ r^{1/2}  N   PT}{   Q^{2} }
      + N^\varepsilon  \frac{r^{1/2} N^2} { Q^{3}  }
    +  N^\varepsilon \frac{ r^{1/2}  N^{3/2} P^{1/2}    }{   Q^{2} }.
\end{align*}
Here we have used $X\ll T^\varepsilon$.
Note that by the assumption $\frac{NX}{PQ}\gg |T'|^{1-\varepsilon}$, we have $P\ll \frac{N}{Q |T'|^{1-\varepsilon}}$.
Together with $Q=\frac{ N^{1/2}}{K^{1/2}}$ and $N\ll \frac{T^{3/2+\varepsilon} |T'|^{3/2}}{r^2}$, the above is bounded by
\begin{equation}\label{eqn:zero-freq-2}
      \ll
      N^\varepsilon \frac{ r^{1/2}  N^2 T}{   Q^{3} |T'| }
    + N^\varepsilon  \frac{r^{1/2} N^{3/2} }
     { Q^{3/2}}
     \ll
      N^{1/2+\varepsilon}  \frac{ r^{1/2}   K^{3/2} T}{   |T'| }
    + N^{1/2+\varepsilon}  T^{3/8} |T'|^{3/8}  K^{3/4}.
\end{equation} 

\section{The non-zero frequencies, I}\label{sec:non-zero-freq-I}

Denote the contribution from the non-zero frequencies in \eqref{eqn:Omega=} by $\Omega_{\neq}$. Our method to bound $\Omega_{\neq}$ is similar to  \cite{HuangXu}.
In this section, we deal with the case  $T^\varepsilon \ll \frac{NX}{PQ}\ll |T'|^{1-\varepsilon}$.
By Lemmas \ref{lem:C} and \ref{lem:I<<} we get
\begin{equation}\label{eqn:Omega<<12}
  \Omega_{\neq} \ll  N^\varepsilon \sup_{1\ll N_*\ll \frac{P Q^2 n_1}{q_1 NX^2} N^\varepsilon} \Omega(N_*),
\end{equation}
where
\begin{multline*}
  \Omega(N_*)
   = H(N_*)\cdot N_2 \sum_{\substack{q_2\sim P/q_1 \\ (q_2,n_1r)=1}}
    \sum_{\substack{m\geq1 \\ m\asymp \frac{P^2 T  |T'|}{N}}}
  \sum_{\substack{q_2'\sim P/q_1 \\ (q_2',n_1r)=1}}
    \sum_{\substack{m'\geq1 \\ m'\asymp \frac{P^2 T |T'|}{N}}}|c_{m'}|^2 \;
    \frac{rq_1}{n_1} \sum_{\substack{d_1\mid q_1 \\ (d_1,n_1)\mid m}}
  \sum_{d_1'\mid q_1}   d_1' (d_1,n_1)  \\
  \cdot
     \mathop{\sum\sum}_{\substack{d_2 \mid (q_2, \pm n_1 q_2'\pm_1 mn) \\
    d_2' \mid (q_2', \mp n_1 q_2 \pm_1 m'n)}} d_2d_2'
    \min\left\{\frac{q_2}{[q_2/(q_2,q_2'),d_2]},\frac{q_2'}{[q_2'/(q_2,q_2'),d_2']} \right\}
     \sum_{\substack{n\asymp N_* \\ (q_2,q_2')\mid n}}  1 ,
\end{multline*}
with
\begin{equation}\label{eqn:H}
  H(N_*)
  := \left\{\begin{array}{ll}
         \left( \frac{N_* N_2 n_1 q_1}{r P^2 }\right)^{-1/2} , & \textrm{if   $ \frac{N n_1}{q_1 P  |T'|^2}  N^\varepsilon + \frac{P^2 Q^3 n_1}{q_1 N^2 X^3} N^\varepsilon  \ll  N_* \ll \frac{P Q^2 n_1}{q_1 NX^2} N^\varepsilon$,} \\
         1,    & \textrm{if   $1\ll N_* \ll\frac{N n_1}{q_1 P  |T'|^2}  N^\varepsilon + \frac{P^2 Q^3 n_1}{q_1 N^2 X^3} N^\varepsilon$}.
            \end{array} \right.
\end{equation}
Here we have used $|c_{m} c_{m'}| \leq |c_{m}|^2 + |c_{m'}|^2$ and without loss of generality only give details for the case with $|c_{m'}|^2$.

Changing the order of summations we have
\begin{multline*}
  \Omega(N_*)
  = H(N_*)\cdot N_2  \sum_{\substack{q_2\sim P/q_1 \\ (q_2,n_1r)=1}}
  \sum_{\substack{q_2'\sim P/q_1 \\ (q_2',n_1r)=1}}
     \frac{rq_1}{n_1} \sum_{d_1\mid q_1}
  \sum_{d_1'\mid q_1}   d_1' (d_1,n_1)
      \sum_{d_2 \mid  q_2} \sum_{d_2' \mid q_2'} d_2d_2'
    \\
  \cdot
    \min\left\{\frac{q_2}{[q_2/(q_2,q_2'),d_2]},\frac{q_2'}{[q_2'/(q_2,q_2'),d_2']} \right\}
     \sum_{\substack{n\asymp N_*  \\ (q_2,q_2')\mid n}}
    \sum_{\substack{m'\asymp \frac{P^2 T  |T'|}{N} \\ \mp n_1 q_2 \pm_1 m'n \equiv 0 \bmod d_2'}} |c_{m'}|^2
    \sum_{\substack{ m\asymp \frac{P^2 T  |T'|}{N} \\ \pm n_1 q_2'\pm_1 mn \equiv 0 \bmod d_2 \\
    (d_1,n_1)\mid m}} 1.
\end{multline*}
Writing $q_3=(q_2,q_2')$, rewriting $q_2$ as $q_2q_3$ and $q_2'$ and $q_2'q_3$, and changing the order of summations, we get
\begin{multline*}
  \Omega(N_*)
   \ll H(N_*)\cdot N_2   \frac{rq_1}{n_1} \sum_{d_1\mid q_1}
  \sum_{d_1'\mid q_1}   d_1' (d_1,n_1)  \sum_{\substack{q_3\ll P/q_1 \\ (q_3,n_1r)=1}} q_3
   \sum_{\substack{q_2\sim P/q_1q_3 \\ (q_2,n_1r)=1}}
  \sum_{\substack{q_2'\sim P/q_1q_3 \\ (q_2',n_1r)=1 \\ (q_2',q_2)=1}}
     \\
  \cdot  \sum_{d_2 \mid  q_2q_3} \sum_{d_2' \mid q_2'q_3} d_2d_2' \min\left\{\frac{q_2}{[q_2,d_2]},\frac{q_2'}{[q_2',d_2']} \right\}
  \sum_{\substack{n\asymp N_*  \\ q_3\mid n}}
   \\
  \cdot
    \sum_{\substack{m'\asymp \frac{P^2 T |T'|}{N} \\ \mp n_1 q_2q_3 \pm_1 m'n \equiv 0 \bmod d_2'}} |c_{m'}|^2
    \sum_{\substack{ m\asymp \frac{P^2 T |T'|}{N (d_1,n_1)} \\ \pm n_1 q_2'q_3\pm_1 (d_1,n_1)mn \equiv 0 \bmod d_2 }} 1.
\end{multline*}
Note that $((d_1,n_1),d_2)=1$. From the congruence condition $\pm n_1 q_2'q_3\pm_1 (d_1,n_1) mn \equiv 0 \bmod d_2$, we know $m$ is uniquely determined modulo $d_2/(d_2,n)$. So
\begin{equation}\label{eqn:sum-m-Omega}
   \sum_{\substack{ m\asymp \frac{P^2 T |T'|}{N (d_1,n_1)} \\ \pm n_1 q_2'q_3\pm_1 (d_1,n_1)mn \equiv 0 \bmod d_2 }} 1
   \ll \left(\frac{P^2 T |T'|}{N (d_1,n_1)}\frac{(d_2,n)}{d_2} + 1\right) \delta_{(d_2,n)\mid n_1q_2'q_3}.
\end{equation}
Note that $(d_2,n_1)=1$  and hence $(d_2,n)\mid q_2'q_3$.
Now rewrite $d_2$ as $d_2d_3$ with $d_2\mid q_2$ and $d_3\mid q_3$ and similarly for $d_2'$ and $d_3'$.
Since $d_3\mid q_3 \mid n$, we have $(d_2d_3,n) = d_3 (d_2,n/d_3)\mid q_2'q_3$ and therefore $(d_2,n/d_3)\mid q_2'q_3/d_3$. Then $(q_2',q_2)=1$ and $d_2\mid q_2$ imply that $((d_2,n/d_3),q_2')=1$ and hence
\begin{equation*}
  \Big(d_2,\frac{n}{d_3}\Big)  \mid \frac{q_3}{d_3}   .
\end{equation*}
We arrive at
\begin{multline*}
  \Omega(N_*)
  \ll H(N_*)\cdot N_2   \frac{rq_1}{n_1} \sum_{d_1\mid q_1}
  \sum_{d_1'\mid q_1}   d_1' (d_1,n_1)
  \sum_{\substack{q_3\ll P/q_1 \\ (q_3,n_1r)=1}} q_3
  \sum_{d_3\mid q_3}d_3
      \sum_{d_2 \ll  P/q_1q_3}   d_2
  \sum_{\substack{q_2\sim P/q_1q_3 
  \\ d_2\mid q_2}}
    \\
  \cdot
    \sum_{d_3'\mid q_3} d_3'
    \sum_{\substack{d_2' \ll P/q_1q_3 }} d_2'
    \sum_{\substack{q_2'\sim P/q_1q_3 \\ d_2' \mid q_2'}}
   \min\left\{ \frac{(q_2/d_2,d_3)}{d_3},\frac{(q_2'/d_2',d_3')}{d_3'}\right\}
     \\
  \cdot
  \sum_{\substack{n\asymp N_*  \\ q_3\mid n \\ (d_2,\frac{n}{d_3})  \mid \frac{q_3}{d_3}}}
   \sum_{\substack{m'\asymp \frac{P^2 T |T'|}{N} \\  \mp n_1 q_2q_3 \pm_1 m'n \equiv 0 \bmod d_2'd_3'}} |c_{m'}|^2
   \left(\frac{P^2 T |T'|}{N (d_1,n_1)}\frac{(d_2,n/d_3)}{d_2} + 1\right) .
\end{multline*}
According to whether $\mp n_1 q_2 q_3 \pm_1 m'n$ is zero or not, we have
\begin{equation}\label{eqn:Omega1<<}
  \Omega(N_*) \ll \Omega_{1}+\Omega_{2},
\end{equation}
where
\begin{multline*}
  \Omega_{1}
   = H(N_*)\cdot N_2  \frac{rq_1^3}{n_1}
    \sum_{\substack{q_3\ll P/q_1 \\ (q_3,n_1r)=1}} q_3
  \sum_{d_3\mid q_3}d_3
      \sum_{d_2 \ll  P/q_1q_3}   d_2\sum_{\substack{q_2\sim P/q_1q_3   \\ d_2\mid q_2}}  \frac{(q_2/d_2,d_3)}{d_3}
    \sum_{\substack{m'\asymp \frac{P^2 T |T'|}{N} }} |c_{m'}|^2
      \\
  \cdot
  \sum_{\substack{n\asymp N_*  \\ q_3\mid n
   \\ \mp n_1 q_2q_3 \pm_1 m'n \neq 0}} \left( \frac{P^2 T |T'|}{N q_1}\frac{(d_2,n/d_3)}{d_2} + 1\right)
    \sum_{d_3'\mid q_3} d_3'
    \sum_{\substack{d_2' \ll P/q_1q_3 \\ \mp n_1 q_2q_3 \pm_1 m'n \equiv 0 \bmod d_2'd_3' }} d_2'
    \sum_{\substack{q_2'\sim P/q_1q_3 \\ d_2' \mid q_2'}}1
\end{multline*}
and
\begin{multline*}
  \Omega_{2}
   = H(N_*)\cdot N_2  \frac{rq_1^3}{n_1}
    \sum_{\substack{q_3\ll P/q_1 \\ (q_3,n_1r)=1}} q_3
     \sum_{d_3'\mid q_3} d_3'
    \sum_{\substack{d_2' \ll P/q_1q_3 }} d_2'
    \sum_{\substack{q_2'\sim P/q_1q_3 \\ d_2' \mid q_2'}}  \frac{(q_2'/d_2',d_3')}{d_3'}
    \sum_{\substack{m'\asymp \frac{P^2 T |T'|}{N} }} |c_{m'}|^2
      \\
  \cdot
  \sum_{\substack{n\asymp N_*  \\ q_3\mid n }}
  \sum_{d_3\mid q_3}d_3
  \sum_{\substack{d_2 \ll  P/q_1q_3  \\ (d_2,\frac{n}{d_3})  \mid \frac{q_3}{d_3}} }   d_2
      \left(  \frac{P^2 T |T'|}{Nq_1}\frac{(d_2,n/d_3)}{d_2} + 1\right)
  \sum_{\substack{q_2\sim P/q_1q_3 \\ d_2\mid q_2 \\ \mp n_1 q_2q_3 \pm_1 m'n =0 }}
  1.
\end{multline*}

\subsection{The non-zero cases}
If $\mp n_1 q_2 q_3 \pm_1 m'n \neq 0$ then the innermost sums over $d_3'$, $d_2'$ and $q_2'$ in $\Omega_1$ are bounded by
\begin{equation}\label{eqn:sum-d'}
  \ll \sum_{d_3'\mid q_3} d_3'
    \sum_{\substack{d_2' \ll P/q_1q_3 \\ \mp n_1 q_2q_3 \pm_1 m'n \equiv 0 \bmod d_2'd_3'}} \frac{P}{q_1q_3 }
  \ll N^\varepsilon \frac{P}{q_1}.
\end{equation}
Hence we have
\begin{multline*}
  \Omega_{1} \ll N^\varepsilon H(N_*)\cdot N_2 \frac{rq_1^3}{n_1} \frac{P}{q_1}  \sum_{\substack{q_3\ll P/q_1 \\ (q_3,n_1r)=1}} q_3
  \sum_{d_3\mid q_3}
      \sum_{d_2 \ll  P/q_1q_3}   d_2\sum_{\substack{q_2\sim P/q_1q_3 \\ (q_2,n_1r)=1 \\ d_2\mid q_2}}   (q_2/d_2,d_3)
      \\
  \cdot
      \sum_{\substack{m'\asymp \frac{P^2 T |T'|}{N} }} |c_{m'}|^2
  \sum_{\substack{n\asymp N_*  \\ q_3\mid n}} \left( \frac{P^2 T |T'|}{Nq_1}\frac{(d_2d_3,n)}{d_2d_3} + 1\right).
\end{multline*}
Note that  we have
\begin{multline}\label{eqn:sum-n}
  q_3 \sum_{\substack{ n\sim N_*  \\ q_3 \mid n}}  (d,n)
   \ll  q_3  \sum_{n\sim N_*/q_3} (d,q_3 n)
  \leq q_3  \sum_{n\sim N_*/q_3} (d,n)  (d,q_3) \\
  \leq  (d,q_3)  q_3  \sum_{d'\mid d} d' \sum_{n\sim N_*/q_3,\; d'\mid n} 1
  \ll d^{\varepsilon} (d,q_3) N_*.
\end{multline}
By \eqref{eqn:b&c},  we have
\begin{align*}
  \Omega_{1} & \ll N^\varepsilon H(N_*)\cdot N_2 N_*  \frac{rq_1^3}{n_1}
  \frac{P}{q_1}
  \frac{P^2 T |T'|}{N} \sum_{\substack{q_3\ll P/q_1 \\ (q_3,n_1r)=1}}
  \sum_{d_3\mid q_3}
  \\ & \hskip 60pt \cdot    \sum_{d_2 \ll  P/q_1q_3}   d_2
  \sum_{\substack{q_2\sim P/q_1q_3 \\ (q_2,n_1r)=1 \\ d_2\mid q_2}} (q_2/d_2,d_3) \left(\frac{P^2 T |T'|}{Nq_1}\frac{  (d_2,q_3) }{d_2} + 1\right)
  \\
  & \ll N^\varepsilon H(N_*)\cdot N_2  N_*  \frac{rq_1^3}{n_1}
  \frac{P}{q_1}   \frac{P^2 T |T'|}{N}
  \left(\frac{P^2 T |T'|}{Nq_1}\frac{P}{q_1}  + \frac{P^2}{q_1^2} \right) .
\end{align*}
Recall that  $N_2 =  \frac{r  N^{2}X^3} {n_1^2 Q^3} $. If $\frac{N n_1}{q_1 P  |T'|^2}  N^\varepsilon + \frac{P^2 Q^3 n_1}{q_1 N^2 X^3} N^\varepsilon  \ll  N_* \ll \frac{P Q^2 n_1}{q_1 NX^2} N^\varepsilon$, then by \eqref{eqn:H} we obtain
\begin{align*}
  \Omega_{1}
  &  \ll N^\varepsilon \left( \frac{r  N^{2}X^3} {n_1^2 Q^3}\right)^{1/2}  \frac{rq_1^3}{n_1} \frac{r^{1/2} P}{ q_1^{1/2}  n_1^{1/2}} \frac{P}{q_1} \left(\frac{P Q^2 n_1}{q_1 NX^2}\right)^{1/2}  \frac{P^2 T |T'|}{N}
  \left(\frac{P^2 T |T'|}{Nq_1}\frac{P}{q_1}  + \frac{P^2}{q_1^2} \right)
    \\
  & \ll  \frac{r^2  P^{15/2}   T^2  |T'|^2 X^{1/2}} { n_1^{2} q_1  Q^{1/2} N^{3/2} }
    +  \frac{r^2  P^{13/2}   T |T'| X^{1/2}}  { n_1^{2} q_1  Q^{1/2} N^{1/2} }  .
\end{align*}
By \eqref{eqn:S<<Omega} and \eqref{eqn:sum-n1&q1}, the contribution from $\Omega_{1}$ to $S_{r}^\pm(N,X,P) $ is bounded by
\begin{align*}
   & \ll  N^\varepsilon  \frac {Q} {r P^2 T^{1/2}  |T'|^{1/2} X}
      \left(\frac{r  N^{2}X^3} {Q^3} \right)^{1/2}
    \left( \frac{r  P^{15/4}   T |T'| X^{1/4}}
             {   Q^{1/4} N^{3/4} }
    +  \frac{r   P^{13/4} T^{1/2}  |T'|^{1/2}  X^{1/4}}
             {  Q^{1/4} N^{1/4} }  \right)
              \\
  & \ll  N^\varepsilon r^{1/2}  N^{1/4} Q   T^{1/2}  |T'|^{1/2}
    +  N^\varepsilon r^{1/2}  N^{3/4} Q^{1/2}.
\end{align*}
By  $Q=\frac{ N^{1/2}}{K^{1/2}}$ and $N\ll \frac{T^{3/2+\varepsilon} |T'|^{3/2}}{r^2}$, the above is bounded by
\begin{equation}\label{eqn:non-zero-freq-11}
   \ll  N^\varepsilon   N^{1/2}  \frac{ T^{7/8}  |T'|^{7/8}}{K^{1/2}}
    +  N^\varepsilon   N^{1/2} \frac{  T^{3/4}  |T'|^{3/4}}{K^{1/4}}
   .
\end{equation}

If $1\ll N_* \ll \frac{N n_1}{q_1 P  |T'|^2}  N^\varepsilon + \frac{P^2 Q^3 n_1}{q_1 N^2 X^3} N^\varepsilon $, then by \eqref{eqn:H} we obtain
\begin{align*}
  \Omega_{1}
    & \ll N^\varepsilon \frac{r  N^{2}X^3} {n_1^2 Q^3}  \frac{rq_1^3}{n_1} \frac{P}{q_1}
  \left( \frac{N n_1}{q_1 P  |T'|^2}  + \frac{P^2 Q^3 n_1}{q_1 N^2 X^3} \right)
   \frac{P^2 T|T'|}{N}
  \left(\frac{P^2 T|T'|}{Nq_1}\frac{P}{q_1}  + \frac{P^2}{q_1^2} \right)
    \\
  & \ll  N^\varepsilon \frac{r^2  N  P^5 T|T'| X^3} {n_1^2 q_1  Q^3}
  \left( \frac{T}{|T'|}
  + \frac{P^3 Q^3 T|T'|}{N^3 X^3}
  + \frac{N}{P  |T'|^2}
  + \frac{P^2 Q^3}{N^2 X^3} \right)  .
\end{align*}
By \eqref{eqn:S<<Omega} and \eqref{eqn:sum-n1&q1}, the contribution from $\Omega_{1}$ to $S_{r}^\pm(N,X,P) $ is bounded by
\begin{align*}
   & \ll  N^\varepsilon  \frac {Q} {r P^2 T^{1/2}|T'|^{1/2} X}
      \left( \frac{r  N^{2}X^3} {Q^3} \right)^{1/2}
   \frac{r   N^{1/2}  P^{5/2} T^{1/2}|T'|^{1/2}  X^{3/2}} {Q^{3/2}}
      \\
   & \hskip 60pt \cdot
  \left( \frac{T^{1/2}}{|T'|^{1/2}}  + \frac{P^{3/2} Q^{3/2} T^{1/2}|T'|^{1/2}}{N^{3/2} X^{3/2}} +
   \frac{N^{1/2}}{P^{1/2} |T'|}  + \frac{P Q^{3/2}}{N X^{3/2}} \right)
              \\
   & \ll   N^\varepsilon
       \frac{r^{1/2} N^{3/2} } {Q^{3/2}}\frac{T^{1/2}}{|T'|^{1/2}}
    +  N^\varepsilon r^{1/2} Q^{3/2}  T^{1/2}|T'|^{1/2}
    +  N^\varepsilon \frac{r^{1/2} N^{2}} {Q^{2} |T'|}
   +  N^\varepsilon r^{1/2} N^{1/2}  Q .
\end{align*}
Here we have used $X\ll T^\varepsilon$.
By  $Q=\frac{ N^{1/2}}{K^{1/2}}$ and $N\ll \frac{T^{3/2+\varepsilon} |T'|^{3/2}}{r^2}$, the above is bounded by
\begin{equation}\label{eqn:non-zero-freq-21}
    \ll
         N^{1/2+\varepsilon} \frac{K^{3/4}  T^{7/8}}{|T'|^{1/8}}
    +   N^{1/2+\varepsilon} \frac{ T^{7/8}|T'|^{7/8}}{K^{3/4}}
    +   N^{1/2+\varepsilon} \frac{T^{3/4} K} {  |T'|^{1/4}}
   +   N^{1/2+\varepsilon}  \frac{T^{3/4}|T'|^{3/4}}{K^{1/2}}.
\end{equation}


\subsection{The zero case}
If $\mp n_1 q_2 q_3 \pm_1 m'n = 0$ then we have $n\asymp \frac{n_1 q_2 q_3}{m'} \ll \frac{n_1 N}{q_1 PT|T'|}$. Since $n\asymp N_*$, we have $\Omega_2=0$ unless $N_* \ll \frac{n_1 N}{q_1 PT|T'|}$ in which case, by \eqref{eqn:H}, we have $H(N_*)=1$.
In $\Omega_2$, we should consider the sums over $d_2,\ q_2, \ q_3$ first.
Since $d_2\mid q_2$, we rewrite $q_2$ as $d_2 q_2$.
Hence
\begin{multline*}
  \Omega_{2}
   \ll  N_2   \frac{rq_1^3}{n_1} \sum_{\substack{q_3\ll P/q_1 \\ (q_3,n_1r)=1}} q_3
    \sum_{d_3'\mid q_3}
   \sum_{\substack{d_2' \ll P/q_1q_3 }} d_2'
    \sum_{\substack{q_2'\sim P/q_1q_3 \\ d_2' \mid q_2'}}  (q_2'/d_2',d_3')
    \sum_{\substack{m'\asymp \frac{P^2 T |T'|}{N} }} |c_{m'}|^2
      \\
  \cdot
  \sum_{\substack{1  \ll |n|\ll \frac{n_1 N}{q_1 PT |T'|}  \\ q_3\mid n}}
  \sum_{d_3\mid q_3}  d_3
  \sum_{\substack{d_2 \ll  P/q_1q_3  \\ (d_2,\frac{n}{d_3})  \mid \frac{q_3}{d_3} }}   d_2
  \left(\frac{P^2 T |T'|}{Nq_1}\frac{(d_2,n/d_3)}{d_2} + 1\right)
     \sum_{\substack{q_2\sim P/q_1q_3d_2 \\  \mp n_1 d_2 q_2q_3 \pm_1 m'n =0 } } 1
     .
\end{multline*}
Note that $d_2d_3\leq d_2q_3\ll P/q_1$ and the choices of $(d_2,d_3,q_2)$ are at most $N^\varepsilon$ for each fixed $m',n$ as $d_2d_3q_2 \mid d_2 q_2q_3 \mid m'n$. Note that $ (d_2,\frac{n}{d_3})  \mid \frac{q_3}{d_3}$ gives $(d_2,n/d_3) d_3 \leq q_3$.  Hence
\begin{multline*}
  \Omega_{2}
   \ll  N_2  \frac{rq_1^3}{n_1}  \sum_{\substack{q_3\ll P/q_1 \\ (q_3,n_1r)=1}} q_3
    \sum_{d_3'\mid q_3}
   \sum_{\substack{d_2' \ll P/q_1q_3 }} d_2'
    \sum_{\substack{q_2'\sim P/q_1q_3d_2' }}  (q_2',d_3')
      \\
  \cdot
    \sum_{\substack{m'\asymp \frac{P^2 T |T'|}{N} }} |c_{m'}|^2
  \sum_{\substack{1  \ll |n|\ll \frac{n_1 N}{q_1 PT |T'|}  \\ q_3\mid n}}  \left(\frac{P^2 T |T'|}{Nq_1} q_3  + \frac{P}{q_1} \right) .
\end{multline*}
By \eqref{eqn:b&c} and \eqref{eqn:sum-n}, we have
\begin{equation*}
  \Omega_{2}
  \ll
   N_2   \frac{rq_1^3}{n_1}
    \sum_{\substack{q_3\ll P/q_1 \\ (q_3,n_1r)=1}}
    \sum_{d_3'\mid q_3}
   \sum_{\substack{d_2' \ll P/q_1q_3 }} d_2'
    \sum_{\substack{q_2'\sim P/q_1q_3d_2' }}  (q_2',d_3')
  \left( \frac{n_1 N}{q_1 PT |T'|} \right)
   \left(\frac{P^2 T |T'|}{Nq_1} q_3  + \frac{P}{q_1} \right) \frac{P^2 T |T'|}{N}.
\end{equation*}
Changing the order of summations, we get
\begin{align*}
  \Omega_{2} & \ll
   N_2  \frac{rq_1^3}{n_1}
   \frac{n_1 N}{q_1 PT |T'|}  \frac{P^2 T |T'|}{N}
   \sum_{\substack{q_3\ll P/q_1 \\ (q_3,n_1r)=1}}
   \left(\frac{P^2 T |T'|}{Nq_1} q_3  + \frac{P}{q_1} \right)
    \sum_{d_3'\mid q_3}
   \sum_{\substack{d_2' \ll P/q_1q_3 }} d_2'
    \sum_{\substack{q_2'\sim P/q_1q_3d_2' }}  (q_2',d_3')
  \\
  & \ll  N^\varepsilon
   N_2   rq_1^2  P
   \sum_{\substack{q_3\ll P/q_1 \\ (q_3,n_1r)=1}}
   \left(\frac{P^2 T |T'|}{Nq_1} q_3  + \frac{P}{q_1} \right)
     \frac{P^2}{q_1^2q_3^2}
  \ll N^\varepsilon
   \frac{r^2   N^{2} P^4 X^3} {n_1^2 q_1 Q^3}
   \left(\frac{P T |T'|}{N}   + 1 \right)
     .
\end{align*}
By \eqref{eqn:S<<Omega} and \eqref{eqn:sum-n1&q1}, the contribution from $\Omega_{2}$ to $S_{r}^\pm(N,X,P) $ is bounded by
\begin{align*}
  & \ll  N^\varepsilon  \frac {Q} {r P^2 T^{1/2}|T'|^{1/2} X}
     \frac{r^{1/2}  N X^{3/2}} {Q^{3/2}}
     \frac{r   N P^2 X^{3/2}} { Q^{3/2}}
   \left(\frac{P^{1/2} T^{1/2}|T'|^{1/2}}{N^{1/2}}   + 1 \right)
    \\
  & \ll N^\varepsilon  \frac {r^{1/2} N^{3/2} } { Q^{3/2} }
    +  N^\varepsilon  \frac {r^{1/2} N^2 } { Q^{2} T^{1/2}|T'|^{1/2}}.
\end{align*}
Here we have used $P\leq Q$ and $X\ll T^\varepsilon$.
By  $Q=\frac{ N^{1/2}}{K^{1/2}}$ and  $N\ll \frac{T^{3/2+\varepsilon} |T'|^{3/2}}{r^2}$, the above is bounded by
\begin{equation}\label{eqn:non-zero-freq-22}
    \ll   N^{1/2+\varepsilon } T^{3/8}|T'|^{3/8} K^{3/4}
    +    N^{1/2+\varepsilon }  T^{1/4}|T'|^{1/4} K
   .
\end{equation}

\section{The non-zero frequencies, II} \label{sec:non-zero-freq-II}

In this section we treat the case  $\frac{NX}{PQ}\gg |T'|^{1-\varepsilon}$.
By Lemmas \ref{lem:C} and \ref{lem:I>>} we get
$\Omega_{\neq} \ll N^\varepsilon \Omega_3$, where
\begin{multline*}
  \Omega_{3}
  = \sum_{\substack{q_2\sim P/q_1 \\ (q_2,n_1r)=1}}
    \sum_{\substack{m\geq1 \\ m\asymp M}}
  \sum_{\substack{q_2'\sim P/q_1 \\ (q_2',n_1r)=1}}
    \sum_{\substack{m'\geq1 \\ m'\asymp M}}  |c_{m'}|^2 \;
     \frac{rq_1}{n_1} \sum_{d_1\mid q_1}
  \sum_{d_1'\mid q_1}   d_1' (d_1,n_1) \delta_{(d_1,n_1)\mid m}\\
  \cdot
     \mathop{\sum\sum}_{\substack{d_2 \mid (q_2, \pm n_1 q_2'\pm_1 mn) \\
    d_2' \mid (q_2', \mp n_1 q_2 \pm_1 m'n)}} d_2d_2'
   \min\left\{\frac{q_2}{[q_2/(q_2,q_2'),d_2]},\frac{q_2'}{[q_2'/(q_2,q_2'),d_2']} \right\}  \sum_{\substack{  1\ll  |n|\ll \frac{P Q^2 n_1}{q_1 NX^2} N^\varepsilon \\ (q_2,q_2')\mid n } } N_2.
\end{multline*}
Here we have used $|c_{m} c_{m'}| \leq |c_{m}|^2 + |c_{m'}|^2$ and without loss of generality only give details for the case with $|c_{m'}|^2$.
By the same arguments as in \S \ref{sec:non-zero-freq-I}, we have
\begin{equation}\label{eqn:Omega3<<}
  \Omega_{3} \ll \Omega_{31}+\Omega_{32},
\end{equation}
where
\begin{multline*}
  \Omega_{31}
   = N_2  \frac{rq_1^3}{n_1}
    \sum_{\substack{q_3\ll P/q_1 \\ (q_3,n_1r)=1}} q_3
  \sum_{d_3\mid q_3}d_3
      \sum_{d_2 \ll  P/q_1q_3}   d_2\sum_{\substack{q_2\sim P/q_1q_3 \\ (q_2,n_1r)=1 \\ d_2\mid q_2}}  \frac{(q_2/d_2,d_3)}{d_3}
    \sum_{\substack{m'\asymp M }} |c_{m'}|^2
      \\
  \cdot
  \sum_{\substack{|n|\ll \frac{P Q^2 n_1}{q_1 NX^2} N^\varepsilon    \\ q_3\mid n  \\ \mp n_1 q_2q_3 \pm_1 m'n \neq 0}}
  \left( \frac{M}{ q_1}\frac{(d_2d_3,n)}{d_2d_3} + 1\right)
    \sum_{d_3'\mid q_3} d_3'
    \sum_{\substack{d_2' \ll P/q_1q_3 \\ \mp n_1 q_2q_3 \pm_1 m'n \equiv 0 \bmod d_2'd_3' }} d_2'
    \sum_{\substack{q_2'\sim P/q_1q_3 \\ d_2' \mid q_2'}}1
\end{multline*}
and
\begin{multline*}
  \Omega_{32}
   = N_2  \frac{rq_1^3}{n_1}
    \sum_{\substack{q_3\ll P/q_1 \\ (q_3,n_1r)=1}} q_3
     \sum_{d_3'\mid q_3} d_3'
    \sum_{\substack{d_2' \ll P/q_1q_3 }} d_2'
    \sum_{\substack{q_2'\sim P/q_1q_3 \\ d_2' \mid q_2'}}  \frac{(q_2'/d_2',d_3')}{d_3'}
    \sum_{\substack{m'\asymp M }} |c_{m'}|^2
      \\
  \cdot
  \sum_{\substack{|n|\ll \frac{P Q^2 n_1}{q_1 NX^2} N^\varepsilon    \\ q_3\mid n \\  \mp n_1 q_2q_3 \pm_1 m'n =0}}
  \sum_{d_3\mid q_3}d_3
      \sum_{\substack{d_2 \ll  P/q_1q_3  \\ (d_2,\frac{n}{d_3})  \mid \frac{q_3}{d_3}  }}   d_2
  \left(  \frac{M}{ q_1}\frac{(d_2,n/d_3)}{d_2} + 1\right)
  \sum_{\substack{q_2\sim P/q_1q_3 \\ (q_2,n_1r)=1 \\ d_2\mid q_2}} 1
     .
\end{multline*}

\subsection{The non-zero cases} \label{subsec:non-zeor-non-zero-freq-II}
If $\mp n_1 q_2 q_3 \pm_1 m'n \neq 0$, then by \eqref{eqn:sum-d'} we have
\begin{multline*}
  \Omega_{31} \ll N^\varepsilon N_2 \frac{rq_1^3}{n_1} \frac{P}{q_1} \sum_{\substack{q_3\ll P/q_1 \\ (q_3,n_1r)=1}} q_3
  \sum_{d_3\mid q_3}
      \sum_{d_2 \ll  P/q_1q_3}   d_2\sum_{\substack{q_2\sim P/q_1q_3 \\ (q_2,n_1r)=1 \\ d_2\mid q_2}}  (q_2/d_2,d_3)
      \\
  \cdot
  \sum_{\substack{|n|\ll \frac{P Q^2 n_1}{q_1 NX^2} N^\varepsilon  \\ q_3\mid n}}
  \left(\frac{M}{q_1}\frac{(d_2d_3,n)}{d_2d_3} + 1\right)
    \sum_{\substack{m'\asymp M}} |c_{m'}|^2.
\end{multline*}
By \eqref{eqn:b&c} and \eqref{eqn:sum-n},  we have
\begin{align*}
  \Omega_{31} & \ll N^\varepsilon N_2   \frac{rq_1^3}{n_1}   \frac{P}{q_1}
   \frac{P Q^2 n_1}{q_1 NX^2}   M \sum_{\substack{q_3\ll P/q_1 \\ (q_3,n_1r)=1}}
  \sum_{d_3\mid q_3}
  \sum_{d_2 \ll  P/q_1q_3}   d_2
  \sum_{\substack{q_2\sim P/q_1q_3 \\ (q_2,n_1r)=1 \\ d_2\mid q_2}} (q_2/d_2,d_3) \left(\frac{M}{q_1}\frac{  (d_2,q_3) }{d_2} + 1\right)
  \\
  & \ll N^\varepsilon N_2   \frac{rq_1^3}{n_1} \frac{P}{q_1}
  \frac{P Q^2 n_1}{q_1 NX^2}   M
  \left(\frac{M}{q_1}\frac{P}{q_1}  + \frac{P^2}{q_1^2} \right) .
\end{align*}
Recall that  $N_2 =  \frac{r  N^{2}X^3} {n_1^2 Q^3} $. We obtain
\begin{equation}\label{eqn:Omega31<<}
  \Omega_{31}
    \ll N^\varepsilon \frac{r  N^{2}X^3} {n_1^2 Q^3}  \frac{rq_1^3}{n_1} \frac{P}{q_1}
  \frac{P Q^2 n_1}{q_1 NX^2}   M
  \left(\frac{M}{q_1}\frac{P}{q_1}  + \frac{P^2}{q_1^2} \right)
    \\
    \ll  N^\varepsilon \frac{r^2   N P^2   X}{n_1^2 q_1 Q   }  M
  \left( M P  +  P^2  \right) .
\end{equation}
By \eqref{eqn:S<<Omega} and \eqref{eqn:sum-n1&q1}, the contribution from $\Omega_{31}$ to $S_{r}^\pm(N,X,P,M) $ is bounded by
\begin{align*}
  & \ll N^\varepsilon \frac{ Q }{r N^{1/2} P X M^{1/2} }
      \frac{r^{1/2}   N X^{3/2}} {Q^{3/2}}
      \left( \frac{r   N^{1/2} P X^{1/2}}{Q^{1/2}} M^{1/2}
  \left( \frac{PT}{N^{1/2}}P^{1/2} + \frac{N^{1/2}X}{Q} P^{1/2}  +  P  \right)  \right)
  \\
  & \ll  N^\varepsilon \frac{ r^{1/2} N^{1/2} TP^{3/2} }{ Q }
   + N^\varepsilon \frac{ r^{1/2} N^{3/2} P^{1/2} }{ Q^2 }
   +   N^\varepsilon\frac{ r^{1/2} N }{ Q } P  .
\end{align*}
Here we have used $X\ll T^\varepsilon$ and $M\ll \frac{P^2T^2}{N}+\frac{NX^2}{Q^2}$ as in \eqref{eqn:S<<SM2}. Note that by the assumption $\frac{NX}{PQ}\gg |T'|^{1-\varepsilon}$, we have $P\ll \frac{N}{Q |T'| } T^\varepsilon$.
Together with $Q=\frac{ N^{1/2}}{K^{1/2}}$ and $N\ll \frac{T^{3/2+\varepsilon} |T'|^{3/2}}{r^2}$, the above is bounded by
\begin{align}\label{eqn:non-zero-freq-3}
   & \ll N^\varepsilon \frac{ r^{1/2} N^{2} T }{ Q^{5/2} |T'|^{3/2} }
   + N^\varepsilon \frac{ r^{1/2} N^{3/2}  }{ Q^{3/2} }  +
    N^\varepsilon\frac{ r^{1/2} N^2  }{ Q^2 |T'| } \nonumber
   \\
   & \ll  N^{1/2+\varepsilon} \frac{ T^{11/8} K^{5/4} }{|T'|^{9/8}}
    +     N^{1/2+\varepsilon} T^{3/8}|T'|^{3/8} K^{3/4}
    +     N^{1/2+\varepsilon} \frac{ T^{3/4} K}{|T'|^{1/4}} .
\end{align}

\subsection{The zero case}\label{subsec:zeor-non-zero-freq-II}
In $\Omega_{32}$, we should consider the sums over $d_2,\ d_3, \ q_2$ first.
Since $d_2\mid q_2$, we rewrite $q_2$ as $d_2 q_2$.  Hence we have
\begin{multline*}
  \Omega_{32}
   \ll  N_2   \frac{rq_1^3}{n_1} \sum_{\substack{q_3\ll P/q_1 \\ (q_3,n_1r)=1}} q_3
    \sum_{d_3'\mid q_3}
   \sum_{\substack{d_2' \ll P/q_1q_3 }} d_2'
    \sum_{\substack{q_2'\sim P/q_1q_3 \\ d_2' \mid q_2'}}  (q_2'/d_2',d_3')
    \sum_{\substack{m'\asymp M }} |c_{m'}|^2
      \\
  \cdot
  \sum_{\substack{1  \ll |n|\ll \frac{P Q^2 n_1}{q_1 NX^2} N^\varepsilon   \\ q_3\mid n}}
   \sum_{d_3\mid q_3}  \sum_{d_2 \ll  P/q_1q_3}
  \left(\frac{M}{q_1} \Big(d_2,\frac{n}{d_3}\Big)d_3  +  d_3 d_2 \right)
     \sum_{\substack{q_2\sim P/q_1q_3d_2 \\  \mp n_1 d_2 q_2q_3 \pm_1 m'n =0 } } 1
     .
\end{multline*}
Note that $d_2d_3\leq d_2q_3\ll P/q_1$ and the choices of $(d_2,d_3,q_2)$ are at most $N^\varepsilon$ for each fixed $m',n$ as $d_2 d_3 q_2 \mid d_2 q_2q_3 \mid m'n$. Note that $(d_2,\frac{n}{d_3})  \mid \frac{q_3}{d_3} $ implies $(d_2,n/d_3) d_3 \leq q_3$.
Hence
\begin{equation*}
  \Omega_{32}
   \ll  N_2  \frac{rq_1^3}{n_1}  \sum_{\substack{q_3\ll P/q_1 \\ (q_3,n_1r)=1}}
    \sum_{d_3'\mid q_3}
   \sum_{\substack{d_2' \ll P/q_1q_3 }} d_2'
    \sum_{\substack{q_2'\sim P/q_1q_3d_2' }}  (q_2',d_3')
  q_3
  \sum_{\substack{1  \ll |n|\ll\frac{P Q^2 n_1}{q_1 NX^2} N^\varepsilon \\ q_3\mid n}}  \left(\frac{M}{q_1} q_3  + \frac{P}{q_1} \right)
    \sum_{\substack{m'\asymp M}} |c_{m'}|^2 .
\end{equation*}
By \eqref{eqn:b&c}, we have
\begin{align*}
  \Omega_{32}
  & \ll N^\varepsilon
   N_2   \frac{rq_1^3}{n_1}
    \sum_{\substack{q_3\ll P/q_1 \\ (q_3,n_1r)=1}}
    \sum_{d_3'\mid q_3}
   \sum_{\substack{d_2' \ll P/q_1q_3 }} d_2'
    \sum_{\substack{q_2'\sim P/q_1q_3d_2' }}  (q_2',d_3')
  \frac{P Q^2 n_1}{q_1 NX^2}
   \left(\frac{M}{q_1} q_3  + \frac{P}{q_1} \right)M
      \\
  &
  \ll N^\varepsilon
   \frac{r^2  N P^3 X } {n_1^2 q_1 Q}   M
   \left( M  + P \right)
     .
\end{align*}
Note that this bound is the same as the bound for $\Omega_{31}$ in \eqref{eqn:Omega31<<}.
Hence we get the same bound for the contribution from $\Omega_{32}$ to $S_{r}^\pm(N,X,P,M) $.

\section{Proof of Theorem \ref{thm:main}}\label{sec:Proof}

By \eqref{eqn:zero-freq-1} and \eqref{eqn:zero-freq-2} we have the contribution from the zero frequency to $S_{r}^\pm(N,X,P) $ is bounded by
\begin{equation*}
   \ll N^{1/2+\varepsilon}  \left( \frac{ r^{1/2}  K^{3/2} T}{   |T'| }
   + r^{1/2}    T^{1/2} |T'|^{1/2}
         +
      T^{3/8} |T'|^{3/8}  K^{3/4}   \right).
\end{equation*}
By \eqref{eqn:non-zero-freq-11}, \eqref{eqn:non-zero-freq-21}, \eqref{eqn:non-zero-freq-22}, and \eqref{eqn:non-zero-freq-3}, we have the contribution from the non-zero frequencies to $S_{r}^\pm(N,X,P) $ is bounded by
\begin{multline*}
  \ll
  N^{1/2+\varepsilon}
  \bigg(  \frac{ T^{7/8}  |T'|^{7/8}}{K^{1/2}}
    +    \frac{  T^{3/4}  |T'|^{3/4}}{K^{1/4}} + T^{5/8}|T'|^{5/8}
    +   T^{1/2}|T'|^{1/2} K^{1/4}
    \\
   +  \frac{K^{3/4}  T^{7/8}}{|T'|^{1/8}}
    +  T^{3/8}|T'|^{3/8} K^{3/4}
    +  \frac{T^{3/4} K} {  |T'|^{1/4}}
    +    T^{1/4}|T'|^{1/4} K
    + \frac{ T^{11/8} K^{5/4} }{|T'|^{9/8}}
   \bigg).
\end{multline*}
Assume $K\leq T^{1/2}|T'|^{1/2}$. Note that $T'\ll T$.  Hence we have
\begin{equation*}
   S_{r}^\pm(N,X,P) \ll
  N^{1/2+\varepsilon}
  \bigg(  \frac{ R^{1/2}  K^{3/2} T}{   |T'| }
   + R^{1/2}    T^{1/2} |T'|^{1/2}
   +   \frac{ T^{7/8}  |T'|^{7/8}}{K^{1/2}}
   + \frac{K^{3/4}  T^{7/8}}{|T'|^{1/8}}
    +\frac{ T^{11/8} K^{5/4} }{|T'|^{9/8}}
   \bigg).
\end{equation*}


If $T'\gg T^{5/6}$, then we take $K=|T'|^{4/5}$, and hence by \eqref{eqn:R} we have
\[
  S_{r}^\pm(N,X,P) \ll N^{1/2+\varepsilon}  T^{7/8} |T'|^{19/40}.
\]
If $T^{3/5} \ll T'\ll T^{5/6}$, then we take $K=|T'|^{8/7}T^{-2/7}$,
and hence by \eqref{eqn:R} we have
\[
  S_{r}^\pm(N,X,P) \ll N^{1/2+\varepsilon}  T^{57/56} |T'|^{17/56}.
\]
Note that $K$ satisfies the condition \eqref{eqn:K>>}.
By \eqref{eqn:SrN<<SrNXP} and \eqref{eqn:SrN<<1} we prove Proposition \ref{prop:SrN}, and hence Theorem \ref{thm:main}.

\section*{Acknowledgements}

The author would like to thank the referees for their very helpful comments and suggestions.



\begin{thebibliography}{10}


\bibitem{ASS}
R. Acharya, P. Sharma, S. Singh,
$t$-aspect subconvexity for $\GL(2) \times \GL(2)$ $L$-function.
\emph{ArXiv preprint} (2020), arXiv:2011.01172.

\bibitem{Aggarwal}
K. Aggarwal,
A new subconvex bound for $\GL(3)$ $L$-functions in the $t$-aspect.
\emph{Int. J. Number Theory} 17 (2021), no. 5, 1111--1138.


\bibitem{BernsteinReznikov2010}
J. Bernstein and A. Reznikov,
Subconvexity bounds for triple $L$-functions and representation theory.
\emph{Ann. of Math. (2)} 172 (2010), no. 3, 1679--1718.




\bibitem{blomer2012subconvexity}
V. Blomer,
Subconvexity for twisted $L$-functions on $\GL(3)$.
\emph{Amer. J. Math.} 134 (2012), no. 5, 1385--1421.

\bibitem{BlomerButtcane}
V. Blomer and J. Buttcane,
On the subconvexity problem for $L$-functions on $\GL(3)$.
\emph{Ann. Sci. \'Ec. Norm. Sup\'er.} (4) 53 (2020), no. 6, 1441--1500. 




\bibitem{BlomerJanaNelson}
V. Blomer, S. Jana and P. Nelson,
The Weyl bound for triple product L-functions.
\emph{ArXiv preprint} (2021), arXiv:2101.12106.


\bibitem{BlomerKhanYoung}
V.~Blomer, R.~Khan, and M.~Young,
\newblock Distribution of mass of holomorphic cusp forms.
\newblock {\em Duke Math. J.} 162 (2013), no. 14, 2609--2644.









\bibitem{Burgess}
D. Burgess,
\newblock On character sums and {$L$}-series. {II}.
\newblock \emph{Proc. London Math. Soc. (3)}, 1963, 13:524--536.

\bibitem{DukeFriedlanderIwaniec1993}
W. Duke, J. Friedlander, and H. Iwaniec,
Bounds for automorphic L-functions.
\emph{Invent. Math.} 112 (1993), no. 1, 1--8.


%

\bibitem{DukeFriedlanderIwaniec2002}
W. Duke, J. Friedlander, and H. Iwaniec,
The subconvexity problem for Artin $L$-functions.
\emph{Invent. Math.} 149 (2002), no. 3, 489--577.

\bibitem{EMOT}
A. Erd\'{e}lyi, W. Magnus, F. Oberhettinger, and F. Tricomi,
Higher transcendental functions. Vol. II. Based, in part, on notes left by Harry Bateman. \emph{McGraw-Hill Book Company, Inc., New York-Toronto-London}, 1953. xvii+396 pp.


%

\bibitem{goldfeld2006automorphic}
D. Goldfeld,
\emph{Automorphic forms and $L$-functions for the group $\GL(n,\mathbb R)$}. With an appendix by Kevin A. Broughan. Cambridge Studies in Advanced Mathematics, 99. Cambridge University Press, Cambridge, 2006. xiv+493 pp.
%



\bibitem{goldfeld2006voronoi}
D. Goldfeld and X. Li,
\newblock Voronoi formulas on {${\rm GL}(n)$}.
\emph{Int. Math. Res. Not.} 2006, Art. ID 86295, 25 pp.

\bibitem{Good}
A. Good,
The square mean of Dirichlet series associated with cusp forms.
\emph{Mathematika}, 29 (1982), 278--295. 


%

\bibitem{HarcosMichel}
G. Harcos and Ph. Michel,
The subconvexity problem for Rankin-Selberg $L$-functions and equidistribution of Heegner points. II.
\emph{Invent. Math.} 163 (2006), no. 3, 581--655.

\bibitem{Heath-Brown}
D. R. Heath-Brown,
Hybrid bounds for Dirichlet $L$-functions.
\emph{Invent. Math.} 47 (1978), no. 2, 149--170. 




\bibitem{Huang2021}
B. Huang,
Hybrid subconvexity bounds for twisted $L$-functions on GL(3). \emph{Sci. China Math.} 64 (2021), no. 3, 443--478. 


\bibitem{Huang}
B. Huang,
On the Rankin--Selberg problem.
\emph{Math. Ann.} 381 (2021), no. 3-4, 1217--1251.


\bibitem{HuangXu}
B. Huang and Z. Xu,
Hybrid subconvexity bounds for twists of $\GL(3)\times \GL(2)$ $L$-functions.
\emph{ArXiv preprint} (2021), arXiv:2103.11361.



\bibitem{Ivic}
A. Ivi\'c,
On sums of Hecke series in short intervals.
\emph{J. Th\'eor. Nombres Bordeaux} 13 (2001), no. 2, 453--468.


%

\bibitem{iwaniec1990spectral}
H. Iwaniec,
The spectral growth of automorphic $L$-functions.
\emph{J. Reine Angew. Math.} 428 (1992), 139--159. 

\bibitem{IwaniecKowalski2004analytic}
H.~Iwaniec and E.~Kowalski,
\newblock {\em Analytic number theory}, volume~53 of {\em American Mathematical
  Society Colloquium Publications}.
\newblock American Mathematical Society, Providence, RI, 2004.


%


\bibitem{JutilaMotohashi}
M. Jutila and Y. Motohashi,
Uniform bound for Hecke $L$-functions.
\emph{Acta Math.} 195 (2005), 61--115.

\bibitem{JutilaMotohashi2006}
M. Jutila and Y. Motohashi,
Uniform bounds for Rankin-Selberg $L$-functions.
\emph{Multiple Dirichlet series, automorphic forms, and analytic number theory}, 243--256, Proc. Sympos. Pure Math., 75, \emph{Amer. Math. Soc., Providence, RI,} 2006. 


%


\bibitem{Kim2003}
H. Kim,
Functoriality for the exterior square of $\GL_4$ and the symmetric fourth of $\GL_2$.
\emph{J. Amer. Math. Soc.} 16 (2003), no. 1, 139--183.
With appendix 1 by Ramakrishnan and appendix 2 by Kim and Sarnak. 


\bibitem{KPY}
E. Kiral, I. Petrow, and M. Young,
Oscillatory integrals with uniformity in parameters.
\emph{J. Th\'eor. Nombres Bordeaux} 31 (2019), no. 1, 145--159.


\bibitem{KowalskiMichelVanderKam2002rankin}
E. Kowalski, Ph. Michel, and J. VanderKam,
\newblock Rankin--{S}elberg {$L$}-functions in the level aspect.
\newblock \emph{Duke Math. J.}, 2002, 114(1):123--191.



\bibitem{Kumar}
 S. Kumar,
 Subconvexity bound for $ \GL (3)\times \GL (2) $ $ L $-functions in $ \GL (2) $ spectral aspect. \emph{ArXiv preprint} (2020), arXiv:2007.05043.


\bibitem{LauLiuYe}
Y. Lau, J. Liu and Y. Ye,
A new bound $k^{2/3+\varepsilon}$ for Rankin-Selberg $L$-functions for Hecke congruence subgroups.
\emph{IMRP Int. Math. Res. Pap.} 2006, Art. ID 35090, 78 pp.


%


\bibitem{li2011bounds}
X.~Li,
\newblock Bounds for {${\rm GL}(3)\times {\rm GL}(2)$} {$L$}-functions and
  {${\rm GL}(3)$} {$L$}-functions.
\newblock {\em Ann. of Math. (2)}, 173(1):301--336, 2011.


\bibitem{Lin}
Y. Lin,
Bounds for twists of $\GL(3)$ $L$-functions.
\emph{J. Eur. Math. Soc.} 23 (2021), no. 6, 1899--1924.

\bibitem{LinSun}
Y. Lin and Q. Sun,
Analytic Twists of $\GL_3 \times \GL_2$ Automorphic Forms.
\emph{Int. Math. Res. Not. IMRN} 2021, no. 19, 15143--15208.

\bibitem{LiuYe2002}
J. Liu and Y. Ye, Subconvexity for Rankin-Selberg $L$-functions of Maass forms.
\emph{Geom. Funct. Anal.}, 12 (2002), 1296--1323. 


%


%
%


%



\bibitem{McKeeSunYe}
M. McKee, H. Sun and Y. Ye,
Improved subconvexity bounds for $\GL(2)\times \GL(3)$ and $\GL(3)$ $L$-functions by weighted stationary phase.
\emph{Trans. Amer. Math. Soc.} 370(5) (2018) 3745--3769.

\bibitem{Meurman}
T. Meurman,
On the order of the Maass $L$-function on the critical line.
\emph{Number theory, Vol. I} (Budapest, 1987), 325–354, Colloq. Math. Soc. János Bolyai, 51, \emph{North-Holland, Amsterdam}, 1990.


\bibitem{MichelVenkatesh}
Ph. Michel and A. Venkatesh,
The subconvexity problem for $\GL_2$.
\emph{Publ. Math. Inst. Hautes \'Etudes Sci.} No. 111 (2010), 171--271.


\bibitem{MillerSchmid2006automorphic}
S. Miller and W. Schmid,
\newblock Automorphic distributions, {$L$}-functions, and {V}oronoi summation
  for {${\rm GL}(3)$}.
\newblock \emph{Ann. of Math. (2)}, 2006, 164(2):423--488.







\bibitem{Munshi2015circleIII}
R. Munshi,
The circle method and bounds for $L$-functions--III: $t$-aspect subconvexity for $\GL(3)$ $L$-functions.
\emph{J. Amer. Math. Soc.} 28 (2015), no. 4, 913--938.


\bibitem{Munshi2015circleIV}
R. Munshi,
The circle method and bounds for $L$-functions--IV: Subconvexity for twists of $\GL(3)$ $L$-functions. \emph{Ann. of Math. (2)} 182 (2015), no. 2, 617--672.

\bibitem{Munshi2018}
R. Munshi,
Subconvexity for $\rm GL (3)\times GL (2) $ $ L $-functions in $t$-aspect.
To appear in \emph{J. Eur. Math. Soc. (JEMS)} (2021), 24 pp.



\bibitem{PetrowYoung}
I. Petrow and M. Young,
The Weyl bound for Dirichlet $L$-functions of cube-free conductor.
\emph{Ann. of Math. (2)} 192 (2020), no. 2, 437--486. 

\bibitem{PetrowYoung2019}
I. Petrow and M. Young,
The fourth moment of Dirichlet $L$-functions along a coset and the Weyl bound.
\emph{ArXiv preprint} (2019), arXiv:1908.10346.

\bibitem{Sarnak}
P. Sarnak,
Estimates for Rankin-Selberg $L$-functions and quantum unique ergodicity.
\emph{J. Funct. Anal.} 184 (2001), no. 2, 419--453. 

\bibitem{Sharma2019}
P. Sharma,
Subconvexity for $\GL(3)\times \GL(2)$ twists in level aspect.
\emph{ArXiv preprint} (2019), arXiv:1906.09493.




\bibitem{Venkatesh2010}
A. Venkatesh,
Sparse equidistribution problems, period bounds and subconvexity.
\emph{Ann. of Math. (2)} 172 (2010), no. 2, 989--1094.

\bibitem{Wu}
H. Wu,
Explicit subconvexity for $\GL_2$.
\emph{ArXiv preprint} (2018), arXiv:1812.04391.

\bibitem{Weyl}
H. Weyl,
Zur absch\"atzung von $\zeta(1+ti)$.
\emph{Math. Z.}, 10 (1921), 88--101.

\end{thebibliography}
\end{document}